\documentclass[11pt]{amsart}
\usepackage[T1]{fontenc}
\usepackage[utf8]{inputenc}
\usepackage[margin=1in]{geometry}
\usepackage{amsmath,amssymb,amsfonts,mathtools,mathrsfs,bm}
\usepackage{esint}
\usepackage{enumitem}
\usepackage{array}
\usepackage{microtype}
\usepackage{hyperref}
\usepackage[capitalize]{cleveref}
\hypersetup{
 colorlinks=true,
 linkcolor=blue,
 citecolor=blue,
 urlcolor=blue,
 pdftitle={Lipschitz regularity of harmonic map heat flows from RCD spaces into CAT(0) spaces}
}
\numberwithin{equation}{section}

\theoremstyle{plain}
\newtheorem{theorem}{Theorem}[section]
\newtheorem{proposition}[theorem]{Proposition}
\newtheorem{lemma}[theorem]{Lemma}
\newtheorem{corollary}[theorem]{Corollary}

\theoremstyle{definition}
\newtheorem{definition}[theorem]{Definition}

\theoremstyle{remark}

\crefname{theorem}{Theorem}{Theorems}
\crefname{proposition}{Proposition}{Propositions}
\crefname{lemma}{Lemma}{Lemmas}
\crefname{corollary}{Corollary}{Corollaries}
\crefname{claim}{Claim}{Claims}
\crefname{definition}{Definition}{Definitions}
\crefname{notation}{Notation}{Notations}
\crefname{assumption}{Assumption}{Assumptions}
\crefname{remark}{Remark}{Remarks}
\Crefname{theorem}{Theorem}{Theorems}
\Crefname{proposition}{Proposition}{Propositions}
\Crefname{lemma}{Lemma}{Lemmas}
\Crefname{corollary}{Corollary}{Corollaries}
\Crefname{claim}{Claim}{Claims}
\Crefname{definition}{Definition}{Definitions}
\Crefname{notation}{Notation}{Notations}
\Crefname{assumption}{Assumption}{Assumptions}
\Crefname{remark}{Remark}{Remarks}

\newcommand{\R}{\mathbb R}
\newcommand{\dd}{\mathrm d}
\newcommand{\m}{\mathfrak m}
\newcommand{\dist}{\mathsf d}
\newcommand{\Lip}{\operatorname{Lip}}
\newcommand{\lip}{\operatorname{lip}}
\newcommand{\supp}{\operatorname{supp}}
\newcommand{\esssup}{\operatorname*{ess\,sup}}
\newcommand{\essinf}{\operatorname*{ess\,inf}}
\newcommand{\loc}{\mathrm{loc}}
\newcommand{\Pheat}{\mathsf P}
\newcommand{\Hheat}{\mathsf H}
\newcommand{\Z}{\mathsf Z}
\newcommand{\RCD}{\ensuremath{\mathrm{RCD}}}
\newcommand{\CAT}{\ensuremath{\mathrm{CAT}}}
\newcommand{\KS}{\ensuremath{\mathrm{KS}}}
\newcommand{\EVI}{\ensuremath{\mathrm{EVI}}}
\newcommand{\Ch}{\operatorname{Ch}}
\newcommand{\ip}[2]{\left\langle#1,#2\right\rangle}
\newcommand{\HeatDelta}{\Delta_{\mathrm H}^{+}}
\newcommand{\DeltaZ}{\boldsymbol\Delta_{\Z}}
\newcommand{\MapE}{\mathsf E}
\newcommand{\FlowE}{\mathcal G}
\newcommand{\md}{\mathsf{md}}

\title[Harmonic map heat flows from $\RCD$ spaces]{Lipschitz regularity of
harmonic map heat flows from $\RCD$ spaces into $\CAT(0)$ spaces}

\author{Bang-Xian Han}
\address{School of Mathematics, Shandong University, Jinan 250100, China}
\email{hanbx@sdu.edu.cn}

\author{Hui-Chun Zhang}
\address{Department of Mathematics, Sun Yat-sen University, Guangzhou 510275, China}
\email{zhanghc3@mail.sysu.edu.cn}

\author{Xi-Ping Zhu}
\address{Department of Mathematics, Sun Yat-sen University, Guangzhou 510275, China}
\email{stszxp@mail.sysu.edu.cn}

\thanks{The authors were supported in part by the National Key R\&D Program
of China (Nos.~2022YFA1005400 and 2021YFA1002200), the Shandong Provincial
Natural Science Foundation (ZR2025QB05), and the Taishan Scholars Program of
Shandong Province (tsqn202408059).}

\makeatletter
\@namedef{subjclassname@2020}{\textup{2020} Mathematics Subject Classification}
\makeatother

\begin{document}

\begin{abstract}
We prove positive-time regularity for harmonic map heat flows from
finite-dimensional $\RCD(K,N)$ spaces into complete $\CAT(0)$ spaces, without
assuming that either the source or the target is smooth.  For bounded-image
initial data, the \EVI{} gradient flow of the Dirichlet energy admits a
representative that is locally Lipschitz jointly in space and time.  Moreover,
its spatial pointwise Lipschitz constant satisfies an Eells--Sampson-type
parabolic Bochner inequality.  The main difficulty is that the smooth
parabolic perturbations used to select contact points are unavailable on an
$\RCD{}$ source.  We overcome it by a sliced contact-selection principle that
converts the elliptic ABP estimate on $\RCD{}$ spaces into the space--time
contact selection required by the Hamilton--Jacobi argument.
\end{abstract}

\maketitle

\section{Introduction}

\subsection{Background and motivation}

The harmonic map heat flow between Riemannian manifolds is the gradient flow of the Dirichlet energy for
maps.  Eells--Sampson
\cite{EellsSampson1964} proved long-time smooth existence and convergence
when the target has nonpositive sectional curvature.  Hamilton
\cite{Hamilton1975} subsequently treated the initial-boundary value problem.
These results rely on the convexity of squared distance in the target and on
the parabolic calculus available on a smooth source manifold.

Gromov--Schoen \cite{GromovSchoen1992} initiated the variational theory for
singular nonpositively curved targets.  Korevaar--Schoen
\cite{KorevaarSchoen1993,KorevaarSchoen1997} developed an intrinsic Sobolev
and energy theory for maps from Riemannian domains to complete metric
targets, including compactness and variational existence results.
Independently, Jost \cite{Jost94} introduced an energy of the same
finite-scale averaging form for maps between metric spaces, thus already
allowing general metric source spaces, and related the resulting notions of
equilibrium and harmonic maps to mean-value properties and
$\Gamma$-convergence.  He subsequently developed the convex-functional and
generalized Dirichlet-form aspects of this theory in
\cite{Jost95,Jost97}.  Building on these finite-scale energy constructions,
Gigli--Tyulenev \cite{GigliTyulenev2021} established the Sobolev theory for
maps from strongly rectifiable source spaces, which admit, up to null sets,
countably many bi-Lipschitz Euclidean charts with distortion arbitrarily
close to one.  Complete $\CAT(0)$ spaces provide the synthetic model for
globally nonpositively curved targets; see Reshetnyak
\cite{Reshetnyak1968} and Bridson--Haefliger
\cite{BridsonHaefliger1999}.  For such a target, $L^2(\Omega,Y)$ is again
$\CAT(0)$, so the quadratic energy functional of Sobolev maps generates a unique gradient flow in
the evolution variational inequality (\EVI{}) sense.  This
point of view was developed by Mayer
\cite{Mayer1998}, Jost \cite{Jost1998}, and Sturm
\cite{Sturm2001,Sturm2005}; it gives global existence and uniqueness without
introducing coordinates on the target.

In the present paper, the sources are finite-dimensional $\RCD(K,N)$ spaces.
They form the broadest currently available synthetic class with a lower Ricci
curvature bound, an upper dimension bound, and a Hilbertian first-order
calculus; see Ambrosio--Gigli--Savar\'e
\cite{AmbrosioGigliSavare2014,AmbrosioGigliSavareDuke2014} and
Erbar--Kuwada--Sturm \cite{ErbarKuwadaSturm2015}.  Such spaces are strongly
rectifiable by \cite[Theorem~2.19]{GigliTyulenev2021}; see also
\cite{GigliP2021} for the behavior of the reference measure in rectifiable
charts.

This variational framework is accompanied by a strong elliptic regularity
theory.  In the smooth-source case, Korevaar--Schoen
\cite{KorevaarSchoen1993} proved local Lipschitz continuity for harmonic maps
into $\CAT(0)$ spaces.  For Alexandrov source spaces with curvature bounded
below, Lin \cite{Lin97} and Jost \cite{Jost97} established H\"older
continuity.  Lin \cite{Lin97} conjectured, and Jost \cite{Jost1998} asked more
generally, whether such harmonic maps are locally Lipschitz; the second and
third authors resolved this problem in \cite{ZhangZhu18}.  For
finite-dimensional $\RCD(K,N)$ sources, Gigli \cite{Gigli2023} and
Mondino--Semola \cite{MondinoSemola2026} proved local Lipschitz regularity and
Bochner--Eells--Sampson estimates for harmonic maps into $\CAT(0)$ spaces.
These results provide essential fixed-time tools, but they concern energy
minimizers and do not by themselves control the time evolution of an \EVI{}
flow.

For smooth source manifolds, the regularity of this variational flow has
recently been studied by several methods.  Lin--Segatti--Sire--Wang
\cite{LinSegattiSireWang2026} obtained positive-time regularity through
elliptic approximation and a parabolic frequency argument.  In
\cite{ZhangZhu2026}, the second and third authors proved space--time Lipschitz
regularity for the semigroup solution by combining \EVI{} inequalities,
Campanato estimates, and a metric Hamilton--Jacobi method.  Lin--Wang
\cite{LinWang2026} subsequently gave an alternative,
Korevaar--Schoen-inspired proof of this local Lipschitz regularity.

The present paper extends this theory to finite-dimensional
$\RCD(K,N)$ sources while retaining an arbitrary complete $\CAT(0)$ target.
Thus both sides of the flow may be nonsmooth within the natural synthetic
curvature classes used in the variational theory.  This is not a direct
application of elliptic regularity to the spatial slices: the \EVI{} solution
is initially only an $L^2$-valued curve, its slices are not energy minimizers,
and the pointwise parabolic inequalities needed for the Hamilton--Jacobi
argument are available only almost everywhere.  We prove both positive-time
local Lipschitz continuity jointly in space and time and a parabolic Bochner
inequality for the spatial pointwise Lipschitz constant.

\subsection{Main result}\label{subsec:main-result}

Let $(X,\dist,\m)$ be a finite-dimensional $\RCD(K,N)$ space,
$\Omega\subset X$ a bounded open set, and $(Y,\dist_Y)$ a complete
$\CAT(0)$ space.  For a space--time map $u$, we write
$u^t:=u(\cdot,t)$ for its spatial slice at time $t$.  We use the notation of
\cite{GigliTyulenev2021}: $\KS^{1,2}(\Omega,Y)$ denotes the Sobolev space of
metric-valued maps and $\MapE$ its quadratic energy.  The scale density,
scale energy, Sobolev class, and limiting
energy are fixed in
\eqref{eq:KS-scale-density}--\eqref{eq:KS-scale-energy},
\eqref{eq:KS-space-definition}, and
\eqref{eq:KS-energy-definition}.  Every
finite-dimensional $\RCD(K,N)$ source has an essential dimension $n$, namely
the almost-everywhere dimension of its Euclidean tangent charts; see
\cref{subsec:KS-energy}.  We use the following normalization of the Dirichlet
energy generating the flow:
\begin{equation}\label{eq:flow-energy-normalization}
 \FlowE(u):=\frac{n+2}{2}\MapE(u).
\end{equation}

Given a boundary datum
$\psi\in\KS^{1,2}(\Omega,Y)$, define
\[
 \KS^{1,2}_\psi(\Omega,Y)
 :=\left\{u\in\KS^{1,2}(\Omega,Y):
       \dist_Y(u,\psi)\in W^{1,2}_0(\Omega,\dist,\m)\right\}.
\]
We regard $\FlowE$ as an extended functional on $L^2(\Omega,Y)$ by setting it
equal to $+\infty$ outside the fixed-trace class
$\KS^{1,2}_\psi(\Omega,Y)$.  This class is closed and geodesically convex,
and the extended functional generates an \EVI{} flow; the precise flow
convention is specified in \cref{subsec:semigroup-HMHF}.  See also
Mayer \cite[Sections~1--2]{Mayer1998}, Gigli--Tyulenev
\cite[Theorems~3.13 and~4.16]{GigliTyulenev2021}.

For later use, $V_2(\Omega\times(0,T))$ denotes the class of functions
$f$ satisfying
\begin{equation*}
 \esssup_{0<t<T}\|f(\cdot,t)\|_{L^2(\Omega)}
 +\left(\int_0^T\int_\Omega|\mathrm D f(\cdot,t)|^2\,\dd\m\,\dd t\right)^{1/2}
 <\infty.
\end{equation*}
We write $V_{2,\mathrm{loc}}(\Omega\times(0,T))$ if $f$ belongs to
$V_2(\Omega'\times(a,b))$ for every
$\Omega'\Subset\Omega$ and $0<a<b<T$.

Here and below, a representative of a space--time map is a pointwise-defined
map agreeing with the original map
$\m\otimes\mathcal L^1$-almost everywhere.

\begin{theorem}[Lipschitz regularity and Bochner inequality]\label{main-thm}
Let $(X,\dist,\m)$ be an $\RCD(K,N)$ space with $K\in\mathbb R$ and
$1<N<\infty$, let $\Omega\subset X$ be bounded and open, and let
$(Y,\dist_Y)$ be a complete $\CAT(0)$ space.  Let
$\psi\in\KS^{1,2}(\Omega,Y)$, let
$u_0\in\KS^{1,2}_\psi(\Omega,Y)$, and let $u^t$ be the \EVI{} gradient flow of
$\FlowE$ in the fixed-trace class.  Assume that there exist $P_0\in Y$ and
$M<\infty$ such that
\(u_0(x)\in\overline{B}_M(P_0)\) for \(\m\)-almost every \(x\in\Omega\).
Then the following conclusions hold.
\begin{itemize}
\item[$\mathrm{(i)}$] The flow $u$ has a representative that is locally
Lipschitz on $\Omega\times(0,\infty)$.  More precisely, if
$B_{16R}(x_0)\Subset\Omega$ and $0<t_*<T<\infty$, then
\[
 \dist_Y(u(x,t),u(y,s))\le C\bigl(\dist(x,y)+|t-s|\bigr)
\]
for $x,y\in B_R(x_0)$ and $s,t\in[t_*,T]$.  The constant depends only on
$K,N,R,t_*,T$, the fixed local doubling--Poincar\'e and volume data on
$B_{16R}(x_0)$, the image radius $M$, and $\FlowE(u_0)$.

\item[$\mathrm{(ii)}$] For every $t>0$, define the pointwise spatial
Lipschitz constant by
\begin{equation*}
 \lip_x u(x,t)
 :=\limsup_{\substack{y\to x\\y\ne x}}
 \frac{\dist_Y\bigl(u(x,t),u(y,t)\bigr)}{\dist(x,y)}.
\end{equation*}
Then $\lip_x u$ belongs to
$V_{2,\mathrm{loc}}(\Omega\times(0,\infty))
 \cap L^\infty_{\mathrm{loc}}(\Omega\times(0,\infty))$ and satisfies
\begin{equation*}
 (\boldsymbol\Delta-\partial_t)(\lip_x u)^2
 \ge 2|\mathrm D\lip_x u|^2+2K(\lip_x u)^2
\end{equation*}
on $\Omega\times(0,\infty)$ in the sense of distributions.
\end{itemize}
\end{theorem}

Part~\(\mathrm{(i)}\) gives pointwise control of the flow in both variables,
starting from its variational construction as an $L^2$-valued curve.
Part~\(\mathrm{(ii)}\) is the parabolic counterpart of the
Bochner--Eells--Sampson inequality for harmonic maps: it controls the
evolution of the spatial slope and retains the Ricci-curvature term of the
source.  Thus the theorem provides both positive-time regularization and a
differential estimate for that regularization.

The bounded-image hypothesis is the intrinsic metric-valued counterpart of
bounded initial data in the standard positive-time Lipschitz estimate for the
Euclidean heat equation.  If $\Pheat_t^{\mathbb R^n}$ denotes the Euclidean
heat semigroup on functions, then
$\|\nabla \Pheat_t^{\mathbb R^n}f\|_{L^\infty}
\lesssim t^{-1/2}\|f\|_{L^\infty}$.  Thus a bound
on the initial datum is built into the basic Lipschitz estimate already in the
scalar theory; it is not a technical restriction of our proof.  In the
present setting the image bound is
preserved by the flow through metric projection onto a closed ball; see
\eqref{eq:bounded-image-ball} in \cref{subsec:semigroup-HMHF}.

\subsection{New ingredients for a nonsmooth source}

The proof centers on an elliptic-to-parabolic contact-selection principle on
the source.  It replaces the smooth perturbations used in
\cite[Sections~7--8]{ZhangZhu2026} and supplies the contact information needed
by the metric Hamilton--Jacobi envelope.  The complete argument rests on four
ingredients specific to an \RCD{} source.

\smallskip
\noindent\emph{Two-flow inequality.}
We derive the energy variation and strong two-flow inequality from the
tensorized calculus of Sobolev maps in $\KS^{1,2}(\Omega,Y)$.

\smallskip
\noindent\emph{Simultaneous pointwise identities.}
We recover the fixed-target inequalities and then use heat-kernel blow-up and
parabolic differentiation to place the identities required by the
Hamilton--Jacobi argument on a common full-measure set.

\smallskip
\noindent\emph{Supersolution formulation.}
We replace the smooth viscosity argument by heat-test supersolutions and prove
that this formulation implies the required weak inequality.

\smallskip
\noindent\emph{Contact selection.}
The remaining difficulty is to select a touching point for the
Hamilton--Jacobi envelope.  The heat-kernel identities hold only on a
prescribed full-measure subset of $X\times X\times(0,\infty)$, whereas an
arbitrary minimum need not lie in that set.  Moreover, the affine
perturbations and the classical parabolic ABP argument used on a smooth source
have no counterpart on a general \RCD{} space.  We instead perturb on
$X\times X$ with separated squared-distance functions.  Laplacian comparison
and tensorization provide elliptic upper bounds on each time slice.  The
variation of the running minimum identifies a positive-measure family of
slices on which the elliptic contact estimate of Mondino--Semola
\cite[Theorem~4.3]{MondinoSemola2026} applies.  Integrating the resulting
slice estimates yields a space--time contact set of positive measure, and
hence a contact point in the prescribed full-measure set.  This
elliptic-to-parabolic selection principle replaces the unavailable smooth
space--time perturbation and closes the Hamilton--Jacobi argument.  Together
with a local Laplacian bound for the $p$-power diagonal penalty, it also gives
the parabolic Bochner inequality.

\subsection{Organization of the paper}

\Cref{sec:prelim} collects the scalar \RCD{} estimates, the
$\CAT(0)$ geometry, and the theory of Sobolev spaces for maps used later.
\Cref{sec:two-flow} proves the energy variation and the strong two-flow
inequality.  Positive-time H\"older regularity and pointwise time-Lipschitz
continuity are established in \cref{sec:positive-time}.
\Cref{sec:fixed-target-good-set} develops the fixed-target inequalities and
constructs the simultaneous heat-kernel good set.  The quadratic barriers and
the sliced elliptic contact selection are proved in
\cref{sec:contact-moving}.  \Cref{sec:HJ-final} applies the
Hamilton--Jacobi argument.  The quadratic envelope gives the spatial
Lipschitz bound, while its \(p\)-power version and the limiting argument of
\cite[Section~8]{ZhangZhu2026} give the Eells--Sampson-type Bochner
inequality.

\medskip
\noindent\textbf{Acknowledgments.}
We thank Professor J\"urgen Jost for his interest in this work and for
bringing to our attention the question of convergence of the heat flow to a
harmonic map.

\section{Preliminaries}\label{sec:prelim}

We work with $\RCD(K,N)$ spaces whose
reference measures have full support and are finite on bounded sets.  Under
our convention the ambient metric space is geodesic, complete and separable.  The
curvature--dimension condition makes it locally doubling and locally compact.  Constants are
local unless otherwise stated and may depend on $K$, $N$, and a fixed ambient
ball.

We use the first- and second-order \RCD{} calculus of
Ambrosio--Gigli--Savar\'e
\cite{AmbrosioGigliSavare2014,AmbrosioGigliSavareDuke2014} and Gigli
\cite{Gigli2015}, Sturm's local parabolic theory
\cite{Sturm1995II,Sturm1996III}, and the heat-kernel estimates of Jiang
\cite{Jiang2015} and Jiang, Li, and the second author
\cite{JiangLiZhang2016}.  The metric-valued Sobolev theory is taken from Gigli--Tyulenev
\cite{GigliTyulenev2021}.  We use one-sided Laplacian bounds in the equivalent
distributional and heat-flow formulations of Gigli--Mondino--Semola
\cite[Theorem~1.1]{GigliMondinoSemola2024}, always with a continuous right-hand side when
that equivalence is invoked.

\subsection{Cheeger energy, Laplacian, and scalar parabolic theory}

For a Lipschitz function $f$, let
\(\lip f(x):=\limsup_{y\to x}|f(y)-f(x)|/\dist(x,y)\) be its local Lipschitz
constant.  The Cheeger energy is defined by relaxation:
\begin{equation}
 \Ch(f)
 :=\inf_{\substack{f_j\in\Lip_b(X)\cap L^2(X,\m)\\
                    f_j\to f\ \mathrm{in}\ L^2(X,\m)}}
       \liminf_{j\to\infty}\frac12\int_X(\lip f_j)^2\,\dd\m
 =\frac12\int_X|\mathrm D f|^2\,\dd\m.
\end{equation}
Here $|\mathrm D f|$ denotes the minimal weak upper gradient; see Cheeger
\cite{Cheeger1999} and Shanmugalingam \cite{Shanmugalingam2000}.
We define
\(W^{1,2}(X,\dist,\m):=\{f\in L^2(X,\m):\Ch(f)<\infty\}\).
The infinitesimal Hilbertianity in the definition of an \RCD{} space says
precisely that the Cheeger energy is quadratic.  We write
\(\ip{\nabla f}{\nabla g}:=\frac14\bigl(|\mathrm D(f+g)|^2-
|\mathrm D(f-g)|^2\bigr)\) for the corresponding pointwise polarization.

\begin{definition}[Measure-valued Laplacian]
\label{def:measure-valued-laplacian}
Let $U\subset X$ be open.  A function
$f\in W^{1,2}_{\loc}(U,\dist,\m)$ belongs to
$D_{\loc}(\boldsymbol\Delta,U)$ if there is a locally finite signed Radon
measure $\mu_f$ on $U$ such that
\begin{equation}
 \int_U\phi\,\dd\mu_f
 =-\int_U\ip{\nabla f}{\nabla\phi}\,\dd\m
 \qquad\text{for every }\phi\in\Lip_c(U).
\end{equation}
In this case we set $\boldsymbol\Delta f:=\mu_f$.
\end{definition}

Throughout the paper, \(\boldsymbol\Delta\) denotes this measure-valued
Laplacian. (The symbol \(\HeatDelta\), introduced in \cref{sec:contact-moving},
denotes the corresponding heat-flow upper Laplacian.)  An inequality involving
\(\boldsymbol\Delta-\partial_t\) is always understood in the following weak
sense unless membership in \(D_{\loc}(\boldsymbol\Delta,\cdot)\) is stated
explicitly.
For a space--time function $q(x,t)$, the inequality
\begin{equation}
 (\boldsymbol\Delta-\partial_t)q\ge F
\end{equation}
means that for every nonnegative
$\phi\in\Lip_c(X\times\R)$,
\[
 -\int\!\!\int\ip{\nabla q}{\nabla\phi}\,\dd\m\,\dd t
 +\int\!\!\int q\,\partial_t\phi\,\dd\m\,\dd t
 \ge\int\!\!\int F\phi\,\dd\m\,\dd t.
\]
\medskip

Finite-dimensional \RCD{} spaces are locally doubling
\cite[Corollary~2.4]{SturmActaII} and support a local
$L^2$-Poincar\'e inequality
\cite[Theorem~1.2]{Rajala2012}.  On each
fixed ambient ball $B_R(p)\subset X$ there are constants
$C_D,C_P<\infty$ and $\lambda\ge1$, depending only on the local \RCD{} data of $B_R(p)$, such
that, whenever $B_{\lambda r}(x)\Subset B_R(p)$,
\begin{align}
 \m(B_{2r}(x))&\le C_D\m(B_r(x)),\\
 \fint_{B_r(x)}|f-f_{B_r(x)}|^2\,\dd\m
 &\le C_P r^2\fint_{B_{\lambda r}(x)}|\mathrm D f|^2\,\dd\m.
\end{align}

\begin{proposition}[Good cut-offs, {\cite[Lemma~3.1]{MondinoNaber2019}}]\label{prop:RCD-cutoffs}
Let $B_R(p)\subset X$ and $B_r(x)\subset B_{2r}(x)\Subset B_R(p)$.  There is a test function
$\chi$ such that $0\le\chi\le1$, $\chi\equiv1$ on $B_r(x)$,
$\supp\chi\subset B_{2r}(x)$, and
\begin{equation}
 |\mathrm D\chi|\le C/r,\qquad
 \boldsymbol\Delta\chi=g_\chi\m,\qquad
 g_\chi\in L^\infty(\m),\qquad
 \|g_\chi\|_{L^\infty}\le C r^{-2},
\end{equation}
where the constant depends only on the
local \RCD{} data of $B_R(p)$.
\end{proposition}

For $z_0=(x_0,t_0)\in X\times \R$ put
\[
 Q_r^-(z_0):=B_r(x_0)\times(t_0-r^2,t_0),
 \qquad
 Q_r(z_0):=B_r(x_0)\times(t_0-r^2,t_0+r^2).
\]
We shall use the following scalar parabolic estimates.

\begin{proposition}[Scalar local estimates]
\label{prop:scalar-parabolic-estimates}
Let $B_R(p)\subset X$ and let
\[
Q_{2r}^-(z_0)=B_{2r}(x_0)\times(t_0-4r^2,t_0)
 \Subset B_R(p)\times\mathbb R.
\]
Assume that \(q\) belongs locally to
\(L^2((t_0-4r^2,t_0);W^{1,2}(B_{2r}(x_0),\dist,\m))\), and interpret the
differential inequality in the weak sense specified above.
There are constants \(p_0>0\) and \(C<\infty\), determined by the local
doubling--Poincar\'e data of the ambient ball $B_R(p)$, such that, for every $A\ge0$:
\begin{enumerate}[label=(\roman*),leftmargin=2em]
\item If \(q\ge0\) and \((\boldsymbol\Delta-\partial_t)q\ge-A\) on
\(Q_{2r}^-(z_0)\), then
\begin{equation}\label{eq:expanded-local-boundedness}
 \esssup_{Q_r^-(z_0)}q
 \le \frac{C}{r^2\m(B_{2r}(x_0))}
 \int_{Q_{2r}^-(z_0)}q\,\dd\m\,\dd t+CAr^2.
\end{equation}
\item If \(q\ge0\) and \((\boldsymbol\Delta-\partial_t)q\le A\) on
\(Q_{2r}^-(z_0)\), then
\begin{equation}\label{eq:expanded-weak-harnack}
 \left(\fint_{Q_{r/2}^-(x_0,t_0-2r^2)}q^{p_0}\,\dd\m\,\dd t\right)^{1/p_0}
 \le C\left(\essinf_{Q_{r/2}^-(x_0,t_0)}q+Ar^2\right).
\end{equation}
\item As an elliptic consequence, if
\(B_{2r}(x)\Subset X\), \(h\in W^{1,2}_{\loc}(B_{2r}(x))\), \(h\ge0\), and
\(\boldsymbol\Delta h\ge-A\m\) on \(B_{2r}(x)\), then
\begin{equation}\label{eq:elliptic-local-boundedness}
 \esssup_{B_r(x)}h
 \le C\left(\fint_{B_{2r}(x)}h\,\dd\m+Ar^2\right).
\end{equation}
\end{enumerate}
\end{proposition}

\begin{proof}
The homogeneous local boundedness and weak Harnack inequalities follow from
Sturm's local subsolution estimate
\cite[Theorem~2.1]{Sturm1995II} and parabolic Harnack theorem
\cite{Sturm1996III}, since the Cheeger energy is a strongly local regular
Dirichlet form and the space is locally doubling and supports a local
Poincar\'e inequality.

For (i), set
\(q^{+}(x,t):=q(x,t)+A(t_0-t)\).  Then
\(q^{+}\ge0\) and
\((\boldsymbol\Delta-\partial_t)q^{+}\ge0\).  The homogeneous local
boundedness estimate applied to \(q^{+}\), together with
\(0\le A(t_0-t)\le4Ar^2\), proves Part~(i).

For (ii), put
\(q^{-}(x,t):=q(x,t)+A(t-(t_0-4r^2))\).  This function is nonnegative and
satisfies \((\boldsymbol\Delta-\partial_t)q^{-}\le0\).  Apply the
homogeneous weak Harnack inequality to the two separated cylinders in
Part~(ii).  Since
\(0\le q^{-}-q\le4Ar^2\), this proves the stated estimate.

For (iii), regard $h$ as a time-independent function on
$B_{2r}(x)\times(-4r^2,0)$.  It satisfies
$(\boldsymbol\Delta-\partial_t)h\ge-A$.  Part~(i) gives
the asserted elliptic estimate after canceling the length of the time interval.
\end{proof}

\subsection{Heat kernel and transport contraction}

Let $(\Pheat_t)_{t\ge0}$ be the heat semigroup on functions and let
$p_t(x,y)$ be its kernel.  On every fixed scale the Gaussian estimates of Jiang, Li, and the
second author \cite[Theorems~1.1--1.2]{JiangLiZhang2016} give
\begin{equation}\label{eq:Gaussian}
 \frac{c^{-1}}{\m(B_{\sqrt t}(x))}
 e^{-\dist^2(x,y)/(c_1t)-c_2t}
 \le p_t(x,y)\le
 \frac{c}{\m(B_{\sqrt t}(x))}
 e^{-\dist^2(x,y)/(c_3t)+c_2t}.
\end{equation}
Here one may take \(c_2=0\) when \(K\ge0\); the factors
\(e^{\pm c_2t}\) cover the general case \(K<0\).

The semigroup is strongly continuous and Markovian on every $L^p$ space by
Ambrosio--Gigli--Savar\'e
\cite[Theorem~4.16]{AmbrosioGigliSavare2014}; in particular, it preserves
nonnegativity, satisfies \(0\le f\le1\Rightarrow0\le\Pheat_t f\le1\), and is
an \(L^p\)-contraction.  If
$f\in L^1_{\loc}$, then $\Pheat_t f\to f$ at almost every Lebesgue point;
this follows from the Gaussian bound \eqref{eq:Gaussian} and local doubling by
splitting the heat-kernel integral into \(B_{A\sqrt t}(x)\) and its annular
complement, and then letting \(A\to\infty\).

We write $(\Hheat_t)_{t\ge0}$ for the dual heat flow on measures, defined by
\begin{equation}\label{eq:heat-flow-duality}
 \int_X f\,\dd(\Hheat_t\mu)=\int_X\Pheat_t f\,\dd\mu
\end{equation}
for bounded Borel functions $f$ and $\mu\in\mathcal P_2(X)$.  In particular,
$\Hheat_t\delta_x=p_t(x,\cdot)\m$, whereas
$\Pheat_t f(x)=\int_X f(y)p_t(x,y)\,\dd\m(y)$.  Thus $\Pheat_t$ always acts
on functions below, while $\Hheat_t$ acts on measures.

The $L^p$-Wasserstein contraction of the heat flow is
\begin{equation}\label{eq:Wp-contraction}
 W_p(\Hheat_t\delta_x,\Hheat_t\delta_y)
 \le e^{-Kt}\dist(x,y),\quad \forall p\in(1,+\infty).
\end{equation}
We use \eqref{eq:Gaussian} for all off-diagonal localization errors and
\eqref{eq:Wp-contraction} with an optimal coupling in the Hamilton--Jacobi
argument.  The latter estimate follows from the heat-flow gradient estimate via
Kuwada's duality \cite[Theorem~2.2]{Kuwada2010}; see also von Renesse--Sturm
\cite{VonRenesseSturm2005}, Ambrosio--Gigli--Savar\'e
\cite{AmbrosioGigliSavare2014}, and Erbar--Kuwada--Sturm
\cite{ErbarKuwadaSturm2015}, for $p=2$, and Savar\'e \cite{Savare2014} for general $p\in(1,+\infty)$.

\subsection{\texorpdfstring{$\CAT(0)$}{CAT(0)} geometry and barycenters}
\label{subsec:CAT0-barycenters}

Let $(Y,\dist_Y)$ be a complete $\CAT(0)$ space.  Geodesics are unique;
we write $(1-s)P+sQ$ for the point at parameter $s$ on the geodesic from $P$
to $Q$.  For $P,Q,R\in Y$ and $s\in[0,1]$, the strong convexity of squared
distance reads
\begin{equation}\label{eq:CAT0-strong-convexity}
 \dist_Y^2\bigl(R,(1-s)P+sQ\bigr)
 \le (1-s)\dist_Y^2(R,P)+s\dist_Y^2(R,Q)
      -s(1-s)\dist_Y^2(P,Q).
\end{equation}
For later reference, Reshetnyak's four-point inequality is
\begin{equation}\label{eq:Reshetnyak-four-point}
 \dist_Y^2(P,S)+\dist_Y^2(Q,R)
 \le \dist_Y^2(P,R)+\dist_Y^2(Q,S)
      +2\dist_Y(P,Q)\dist_Y(R,S)
\end{equation}
for every $P,Q,R,S\in Y$; see Reshetnyak \cite{Reshetnyak1968},
Bridson--Haefliger
\cite[Chapter~II.2]{BridsonHaefliger1999} and Ba\v{c}\'ak
\cite[Chapters~2--3]{Bacak2014}.
The metric space $L^2(\Omega,Y)$, equipped with
\begin{equation}
 \dist_{L^2}(u,v):=\sqrt{\int_\Omega\dist_Y^2(u(x),v(x))\,\dd\m(x)},
\end{equation}
is complete $\CAT(0)$.  Indeed, integrating
\eqref{eq:CAT0-strong-convexity} gives the $\CAT(0)$ inequality, while an
$L^2$-Cauchy sequence has a subsequence \((u_{j_k})\) with
\(\sum_k\dist_{L^2}(u_{j_{k+1}},u_{j_k})<\infty\); hence this subsequence is
pointwise Cauchy almost everywhere and converges to a \(Y\)-valued limit by
completeness of \(Y\); see also
\cite[Chapters~2--3]{Bacak2014}.

For a probability measure $\nu\in\mathcal P_2(Y)$, its barycenter
$\operatorname{bar}(\nu)$ is the unique minimizer of the functional
\(P\mapsto\int\dist_Y^2(P,Q)\,\dd\nu(Q)\).
The integrated form of \eqref{eq:CAT0-strong-convexity} makes every minimizing
sequence Cauchy; completeness gives a minimizer, and the same inequality
gives uniqueness.  It also shows that the barycenter belongs to the closed
convex hull of $\supp\nu$; see \cite[Chapters~2--3]{Bacak2014}.
If $E$ is a spatial or space--time set of finite positive measure, we write
\(\bar u_E\) for the barycenter of the pushforward by $u$ of the normalized
restriction to $E$ of, respectively, $\m$ or $\m\otimes\mathcal L^1$.

\subsection{Sobolev energy for maps on \RCD{} spaces}
\label{subsec:KS-energy}

The metric-valued Sobolev theory used here originates with
Korevaar--Schoen \cite{KorevaarSchoen1993} and has closely related
formulations due to Heinonen--Koskela--Shanmugalingam--Tyson
\cite{HeinonenKoskelaShanmugalingamTyson2001}, Kuwae--Shioya
\cite{KuwaeShioya2003}, and Ohta \cite{Ohta2004}.  We use the strongly
rectifiable formulation of Gigli--Tyulenev \cite{GigliTyulenev2021}.

Let $(Y,\dist_Y,\bar y)$ be a
pointed complete metric space and let $u:X\to Y$ be Borel with
$\dist_Y(u,\bar y)\in L^2$.  Its scale-$r$ squared density and scale energy
are
\begin{align}
 e_r[u](x)
 &:=\fint_{B_r(x)}\frac{\dist_Y^2(u(x),u(y))}{r^2}\,\dd\m(y),
 \label{eq:KS-scale-density}\\
 \MapE_r(u)&:=\int_X e_r[u]\,\dd\m.
 \label{eq:KS-scale-energy}
\end{align}
Following the notation of \cite{GigliTyulenev2021}, define
\begin{equation}\label{eq:KS-space-definition}
 u\in\KS^{1,2}(X,Y)
 \quad\Longleftrightarrow\quad
 \varliminf_{r\downarrow0}\MapE_r(u)<\infty.
\end{equation}
Under local doubling and Poincar\'e, Gigli--Tyulenev
\cite[Corollary~3.10]{GigliTyulenev2021} identify this class with the
metric-valued Newtonian/Haj\l asz--Sobolev space and show that the corresponding
limsup and liminf finiteness conditions are equivalent.

By Bru\'e--Semola \cite{BrueSemola2020} and Mondino--Naber
\cite{MondinoNaber2019}, every finite-dimensional \(\RCD(K,N)\) space has an
essential dimension
$n\in\{1,\dots,\lfloor N\rfloor\}$.  Using rectifiable coordinate charts,
Gigli--Tyulenev \cite[Definition~3.3 and Theorem~3.13]{GigliTyulenev2021}
associate to $u$, at almost every point, an approximate metric differential
$\md_xu$.  This is the seminorm on $\mathbb R^n$ describing the first-order
distance behavior of \(u\) in those coordinates.  They prove that the squared
density \(e_u(x):=\fint_{B_1(0)}(\md_xu(z))^2\,\dd\mathcal L^n(z)\) satisfies
\begin{equation}\label{eq:GT-density-convergence}
 e_r[u]\longrightarrow e_u
 \qquad\m\text{-a.e.\ and in }L^1_{\loc}.
\end{equation}
In particular, the corresponding quadratic Sobolev energy is
\begin{equation}\label{eq:KS-energy-definition}
 \MapE(u):=\lim_{r\downarrow0}\MapE_r(u)=\int_X e_u\,\dd\m.
\end{equation}

For a real-valued Sobolev function \(f\), the preceding definition gives
\(|\mathrm D f|^2=(n+2)e_f\) \(\m\)-a.e.
Consequently, the normalization \eqref{eq:flow-energy-normalization} agrees
with the scalar Dirichlet energy:
\[
 \FlowE(f)=\frac{n+2}{2}\int e_f\,\dd\m
 =\frac12\int|\mathrm D f|^2\,\dd\m.
\]
With the heat semigroup generated by the Laplacian, the Euclidean Gaussian
second moment introduces one further factor $2$.  This accounts for the
constant $2(n+2)$ in the heat-kernel recovery formula
\cref{lem:continuous-KS-heat-recovery}: the factor $n+2$ comes from the
preceding scalar identity, and the factor $2$ from the Gaussian variance.

Whenever a map is defined only on an open set $U\subset X$, put
\(U_r:=\{x\in U:\dist(x,X\setminus U)>r\}\).
The scale density $e_r[u](x)$ is then used only for $x\in U_r$, so that the
ball $B_r(x)$ in \eqref{eq:KS-scale-density} is contained in $U$.  Every local
integral and every limiting argument below is taken first on a fixed
$U'\Subset U$, with $r<\dist(U',X\setminus U)$; the limiting density and
energy on $U$ are obtained by exhaustion.  Thus no extension of a locally
defined map across $\partial U$ is implicit in the notation.  Unless another
domain is displayed, $e_r[u]$, $e_u$, $\MapE_r(u)$, $\MapE(u)$, and
$\FlowE(u)$ refer to this local convention.

For a time-dependent map $u$, we write its spatial energy density as the
space--time function \(e(x,t):=e_{u^t}(x)\).

\begin{proposition}[Finite-scale Korevaar--Schoen estimates]
\label{prop:finite-scale-KS-domination}
Let \((X,\dist,\m)\) be an \(\RCD(K,N)\) space, let
\(\Omega\subset X\) be bounded and open, and let
\(u\in\KS^{1,2}(\Omega,Y)\), where \(Y\) is complete.  Fix \(r_0>0\) and set
\(\Omega_{r_0}:=\{x\in\Omega:\dist(x,X\setminus\Omega)>r_0\}\).
Then, for every \(0<r<r_0/4\),
\begin{equation}\label{eq:finite-scale-KS-domination}
 \int_{\Omega_{r_0}} e_r[u]\,\dd\m
 \le C(K,N,r_0,\operatorname{diam}\Omega)\,\MapE(u).
\end{equation}
\end{proposition}

\begin{proof}
Gigli--Tyulenev's finite-scale estimate
\cite[Proposition~2.17 and Corollary~3.10]{GigliTyulenev2021} bounds
\(e_r[u]\) by a constant multiple of
\(G_R^2(x)+\fint_{B_r(x)}G_R^2\,\dd\m\), where \(R\) is fixed below the
distance to the boundary and
\(G_R\in L^2\) is a local maximal-function upper gradient whose \(L^2\)-norm
is controlled by $\MapE(u)$.  Integrating the
second term and using local doubling gives the asserted estimate without
taking the supremum over \(r\) inside the integral.
\end{proof}

The following proposition collects the properties of
\(\KS^{1,2}(\Omega,Y)\) used below.

\begin{proposition}\label{prop:KS-properties}
Let \(Y\) be a complete \(\CAT(0)\) space, and let
\(u,v\in\KS^{1,2}_{\loc}(\Omega,Y)\) have locally bounded images.  Then:
\begin{enumerate}[label=(\alph*)]
\item \(\dist_Y(u,v)\in W^{1,2}_{\loc}(\Omega,\dist,\m)\) and
\begin{equation}\label{eq:distance-upper-gradient}
 |\mathrm D\dist_Y(u,v)|^2\le2(n+2)e_u+2(n+2)e_v.
\end{equation}
In particular, for every \(P\in Y\),
\begin{equation}\label{eq:fixed-P-upper-gradient}
 |\mathrm D\dist_Y(u,P)|^2\le (n+2)e_u.
\end{equation}
\item If \(0\le\phi\le1\) is Lipschitz, then the pointwise geodesic
interpolation \((1-\phi)u+\phi v\) belongs to
\(\KS^{1,2}_{\loc}(\Omega,Y)\).  Moreover, the approximate energies of these
interpolations are computed by the same densities \(e_r[\cdot]\) in
\eqref{eq:KS-scale-density}.
\item For scalar functions
\(f,g\in W^{1,2}_{\loc}(\Omega,\dist,\m)\), with one compactly supported
factor, the ball-scale polarization
\begin{equation}
 B_r(f,g):=\frac12\int_X\fint_{B_r(x)}
 \frac{(f(x)-f(y))(g(x)-g(y))}{r^2}\,\dd\m(y)\,\dd\m(x)
\end{equation}
satisfies
\begin{equation}
 B_r(f,g)\longrightarrow\frac{1}{2(n+2)}
 \int_X\ip{\nabla f}{\nabla g}\,\dd\m.
\end{equation}
\item If \(B_{2r}\Subset\Omega\) and \(\bar u_B\) denotes the barycenter from
\cref{subsec:CAT0-barycenters} of the normalized pushforward of \(\m|_B\), then
\begin{equation}\label{eq:metric-Poincare}
 \fint_{B_r}\dist_Y^2(u,\bar u_{B_r})\,\dd\m
 \le Cr^2\fint_{B_{2r}}e_u\,\dd\m.
\end{equation}
\end{enumerate}
\end{proposition}

\begin{proof}
Item (a) is Gigli--Tyulenev's metric Sobolev
calculus, specifically \cite[Lemma~5.9]{GigliTyulenev2021}; the fixed-target
estimate is the special case \(v\equiv P\).

Item (b) follows by applying $\CAT(0)$ convexity of squared distance
to the scale densities and then using \eqref{eq:GT-density-convergence} from
\cref{subsec:KS-energy}.  Item (c)
follows by applying the convergence \eqref{eq:GT-density-convergence} from
\cref{subsec:KS-energy} to the scalar maps
\(f+g\) and \(f-g\), and then using the polarization identity and the
quadraticity of the Cheeger energy.  The factor $1/2$ comes from the
definition of $B_r$.

For Item~(d), the metric-valued Poincar\'e inequality gives the
pairwise estimate
\[
 \fint_{B_r}\fint_{B_r}\dist_Y^2(u(x),u(y))\,\dd\m(x)\,\dd\m(y)
 \le Cr^2\fint_{B_{2r}}e_u\,\dd\m.
\]
The minimizing property of the $\CAT(0)$ barycenter from
\cref{subsec:CAT0-barycenters} yields
\eqref{eq:metric-Poincare}; see
\cite{KoskelaShanmugalingamTyson2004,Guo2021,GigliTyulenev2021}.
\end{proof}

\subsection{The harmonic map heat flow}
\label{subsec:semigroup-HMHF}

Fix a trace class $\KS^{1,2}_\psi(\Omega,Y)$.  It is a closed geodesically
convex subset of $L^2(\Omega,Y)$, and $\MapE$ is lower semicontinuous and
convex there; see Gigli--Tyulenev
\cite[Theorems~3.13 and~4.16]{GigliTyulenev2021} and Mayer
\cite[Sections~1--2]{Mayer1998}.  Mayer's resolvent construction
\cite[Theorem~1.13]{Mayer1998}, as used for singular metric spaces by Guo
\cite{Guo2021}, and the differential theory of metric gradient flows of
Gigli--Nobili \cite{GigliNobili2021} therefore apply to the energy
\(\FlowE\) fixed in \eqref{eq:flow-energy-normalization}.

Given \(u_0\in\KS^{1,2}_\psi(\Omega,Y)\), an \emph{\EVI{} harmonic map heat
flow} starting from \(u_0\) is a curve
\(u:[0,\infty)\to L^2(\Omega,Y)\) such that
\(u\in AC_{\loc}((0,\infty);L^2(\Omega,Y))\),
\(u^t\in\KS^{1,2}_\psi(\Omega,Y)\) for \(t>0\),
\(u^t\to u_0\) in \(L^2(\Omega,Y)\) as \(t\downarrow0\), and, for every
\(v\in\KS^{1,2}_\psi(\Omega,Y)\),
 \begin{equation}\label{eq:EVI-HMHF}
  \frac12\frac{\dd}{\dd t}\dist_{L^2}^2(u^t,v)+\FlowE(u^t)
  \le \FlowE(v)
  \qquad\text{in }\mathcal D'(0,\infty).
 \end{equation}
Equivalently, for every nonnegative
\(\eta\in C_c^1((0,\infty))\),
 \begin{equation}\label{eq:EVI-HMHF-integrated}
 \begin{split}
  -\frac12\int_0^\infty\eta'(t)\dist_{L^2}^2(u^t,v)\,\dd t
  +\int_0^\infty\eta(t)\FlowE(u^t)\,\dd t
  \le \FlowE(v)\int_0^\infty\eta(t)\,\dd t.
 \end{split}
 \end{equation}

Equations \eqref{eq:EVI-HMHF}--\eqref{eq:EVI-HMHF-integrated} define the \EVI{}
harmonic map heat flow used below, in particular in
\cref{lem:EVI-finite-difference}.  This definition agrees with the
semigroup solution obtained from Mayer's resolvent construction, namely the
limit of the implicit Euler minimizing scheme; see
\cite[Theorems~1.13 and~2.5]{Mayer1998}.

The energy is nonincreasing along this flow: by
\cite[Corollary~2.6]{Mayer1998},
\begin{equation}\label{eq:EVI-energy-monotonicity}
 \FlowE(u^t)\le \FlowE(u^s)\le \FlowE(u_0),
 \qquad 0\le s\le t.
\end{equation}

For an absolutely continuous curve $\tau\mapsto u^\tau$ in
$(L^2(\Omega,Y),\dist_{L^2})$, its metric derivative is
\[
 |\dot u^\tau|_{\dist_{L^2}}
 :=\lim_{h\to0}\frac{\dist_{L^2}(u^{\tau+h},u^\tau)}{|h|}
\]
whenever the limit exists.  Here the subscript $\dist_{L^2}$ specifies the ambient
metric.  General \CAT(0)
gradient-flow theory gives the semigroup property and, for every $0<t_*<T$,
the finite metric-speed bound
\begin{equation}\label{eq:EVI-L}
 L=L(u,t_*,T):=\esssup_{\tau\in[t_*/2,T+1]}|\dot u^\tau|_{\dist_{L^2}}<\infty,
 \qquad
 \dist_{L^2}(u^t,u^s)\le L|t-s|.
\end{equation}
See Mayer \cite[Theorems~2.9 and~2.17]{Mayer1998} and Guo \cite{Guo2021}.

The bounded initial image in \cref{main-thm} propagates without any boundedness assumption on the
interior values of the chosen representative $\psi$.  Indeed, let
$\mathcal B_Y:=\overline{B}_M(P_0)$ contain \(u_0(x)\) for \(\m\)-almost every \(x\), and
let $\pi_{\mathcal B_Y}:Y\to\mathcal B_Y$ be the metric projection.  Since
$u_0\in\KS^{1,2}_\psi(\Omega,Y)$, the maps $u_0$ and $\psi$ have the same
trace, so the fixed-trace class may equivalently be described using the trace
of $u_0$.  The metric projection onto a closed convex subset of a complete
$\CAT(0)$ space is $1$-Lipschitz
\cite[Chapter~II.2]{BridsonHaefliger1999}.  Thus
$\pi_{\mathcal B_Y}$ fixes $u_0$, post-composition by
$\pi_{\mathcal B_Y}$ preserves that trace class and satisfies
\(\MapE(\pi_{\mathcal B_Y}\circ w)\le\MapE(w)\).  If \(w\) takes values in
$\mathcal B_Y$, projection also decreases the $L^2$ distance from \(w\).
Therefore, by induction, every resolvent iterate, and then the \EVI{} limit,
remains in $\mathcal B_Y$.  Here a resolvent iterate is one step of Mayer's
implicit minimization scheme.  In particular,
\begin{equation}\label{eq:bounded-image-ball}
 u(x,t)\in\overline{B}_M(P_0)
\quad\text{for a.e. }x\text{ and every }t>0.
\end{equation}

For a space--time map \(u\), let \(e_r^{X\times\mathbb R}[u]\) denote the
scale-\(r\) squared density computed with the product metric and
product measure on \(X\times\mathbb R\).
This is the ordinary product-metric scale, not the parabolic scale of
\(Q_r\).

\begin{proposition}[Local space--time Sobolev regularity of the heat flow]
\label{prop:space-time-KS-regularity}
Let \(u\) be the harmonic map heat flow generated by
\(\FlowE=(n+2)\MapE/2\), and let
\(Q'=\Omega'\times(t_0,t_1)\Subset\Omega\times(0,\infty)\).
Then the map \((x,t)\mapsto u(x,t)\) belongs to
\(\KS^{1,2}(Q',Y)\) with respect to the product metric-measure structure on
\(X\times\mathbb R\).  For all sufficiently small \(r\),
\begin{equation}
 \int_{Q'} e_r^{X\times\mathbb R}[u](x,t)\,\dd\m(x)\,\dd t
 \le C L_{t_0,t_1}^2(t_1-t_0)
   +C\int_{t_0/2}^{2t_1}\MapE(u^s)\,\dd s,
\end{equation}
where \(L_{t_0,t_1}\) is the metric Lipschitz constant of
 \(t\mapsto u^t\) on \([t_0/2,2t_1]\) in \((L^2(\Omega,Y),\dist_{L^2})\).  The constant \(C\) depends only on the
local \RCD{} data and on the distance of \(Q'\) from the boundary of the
chosen localization cylinder.
\end{proposition}

\begin{proof}
Choose \(Q=\Omega_0\times(t_0/2,2t_1)\Subset\Omega\times(0,\infty)\) with
\(Q'\Subset Q\).  For \((x,t),(y,s)\in Q\) close enough,
\[
 \dist_Y^2(u(x,t),u(y,s))
 \le 2\dist_Y^2(u(x,t),u(x,s))
   +2\dist_Y^2(u(x,s),u(y,s)).
\]
The first term is controlled after integration by
\[
 \dist_{L^2}^2(u^t,u^s)\le L_{t_0,t_1}^2|t-s|^2.
\]
Inclusions between balls in the product metric and products of spatial balls
with time intervals compare the corresponding averages, with
constants depending only on local doubling.  For the spatial term, apply
\cref{prop:finite-scale-KS-domination}, specifically the integrated estimate
 \eqref{eq:finite-scale-KS-domination}, to each slice \(u^s\), and then use
Fubini's theorem to obtain
\[
 \int_{Q'} e_r^{X\times\mathbb R}[u]
 \le C L_{t_0,t_1}^2(t_1-t_0)
   +C\int_{t_0/2}^{2t_1}\MapE(u^s)\,\dd s.
\]
The right-hand side is finite by the energy monotonicity
\eqref{eq:EVI-energy-monotonicity} from
\cref{subsec:semigroup-HMHF}.  Taking
the liminf as $r\to0^+$ proves the asserted
space--time Sobolev regularity.
\end{proof}

\section{The distance subsolution on an \RCD{} source}\label{sec:two-flow}

This section establishes the scalar parabolic inequalities used in the later
regularity argument.  We first compute the Sobolev energy variation
under paired geodesic interpolations, retaining a nonnegative remainder that
reduces to the energy density when one map is constant.  We then combine this
variation with the \EVI{} finite-difference scheme of
\cite[Section~3]{ZhangZhu2026} to obtain the two-flow inequality in
\cref{thm:RCD-two-flow} and its fixed-target consequences in
\cref{cor:fixed-target}.  Because the variations are compactly supported, the
two flows may have different fixed boundary traces.

\subsection{The energy variation and its nonnegative remainder}

Let $u,v\in\KS^{1,2}(\Omega,Y)\cap L^\infty(\Omega,Y)$, let
$0\le\phi\le1$ be Lipschitz with compact support, and set
\begin{equation}
 u_\phi(x):=(1-\phi(x))u(x)+\phi(x)v(x),\qquad
 v_\phi(x):=\phi(x)u(x)+(1-\phi(x))v(x).
\end{equation}
The interpolations are taken along the unique target geodesics.  Put
$w:=\dist_Y^2(u,v)$.

\begin{lemma}[Geodesic variations]\label{lem:variation-admissible}
The maps $u_\phi$ and $v_\phi$ lie in
$\KS^{1,2}(\Omega,Y)\cap L^\infty(\Omega,Y)$.  Moreover,
$w\in W^{1,2}(\Omega,\dist,\m)\cap L^\infty(\Omega)$.
If $u$ and $v$ belong to possibly different fixed-trace classes, then
$u_\phi$ has the trace of $u$, whereas $v_\phi$ has the trace of $v$.
\end{lemma}

\begin{proof}
Part~(a) of \cref{prop:KS-properties}, specifically
\eqref{eq:distance-upper-gradient}, and the scalar chain rule give the
assertion for $w$.  Part~(b) of the same proposition gives the Sobolev
regularity of \(u_\phi\) and \(v_\phi\); their images are bounded by
\(\CAT(0)\) convexity.  Since $\phi$ is compactly supported, the two
interpolations agree with $u$ and $v$, respectively, near the boundary,
which proves the trace statement.
\end{proof}

For the finite-scale calculation set
\begin{equation}
 \kappa_r(x,y):=
 \frac{\mathbf 1_{B_r(x)}(y)}{r^2\m(B_r(x))},
\end{equation}
so that
\[
 e_r[u](x)=\int\kappa_r(x,y)\dist_Y^2(u(x),u(y))\,\dd\m(y).
\]
By the local convention following \eqref{eq:KS-energy-definition}, the scale
energy \(\MapE_r(u)\) is the integral of this density over \(\Omega\).
No symmetry of the kernel is required below.  In terms of \(\kappa_r\), the
polarization from \cref{prop:KS-properties}\,(c) reads
\begin{equation}
 B_r(f,g)=\frac12\iint\kappa_r(x,y)
 (f(x)-f(y))(g(x)-g(y))\,\dd\m(y)\,\dd\m(x).
\end{equation}
By the scalar energy-density convergence
\eqref{eq:GT-density-convergence} and polarization,
\begin{equation}\label{eq:scalar-polarization-limit}
 B_r(f,g)\longrightarrow
 \frac{1}{2(n+2)}\int\ip{\nabla f}{\nabla g}\,\dd\m
\end{equation}
whenever one factor is compactly supported in the interior localization.

Define
\begin{align}
 R_{r}^{u,v}(x)
 &:={\int\kappa_r(x,y)
 \Bigl(\dist_Y(u(x),u(y))-\dist_Y(v(x),v(y))\Bigr)^2\,\dd\m(y)},
 \label{eq:finite-scale-remainder-definition}\\
 R_{u,v}(x)&:=\liminf_{\substack{r\downarrow0\\ r\in\mathbb Q}}R_{r}^{u,v}(x).
 \label{eq:remainder-definition}
\end{align}
Restricting to rational radii makes the pointwise liminf measurable.
Since
\[
 0\le R_r^{u,v}\le2e_r[u]+2e_r[v],
\]
the convergence \eqref{eq:GT-density-convergence} from
\cref{subsec:KS-energy} gives
\begin{equation}\label{eq:remainder-density-bound}
 0\le R_{u,v}\le2e_u+2e_v
 \qquad\m\text{-a.e.}
\end{equation}
and hence $R_{u,v}\in L^1_{\loc}$.  In the fixed-target case,
\begin{equation}\label{eq:remainder-fixed-P}
 R_{u,P}=e_u\qquad\m\text{-a.e.}
\end{equation}
because $R_r^{u,P}=e_r[u]$.

\begin{theorem}[Dirichlet energy variation on \RCD{} spaces]\label{thm:KS-variation}
For the geodesic variations in \cref{lem:variation-admissible}, with
the remainders defined by
\eqref{eq:finite-scale-remainder-definition}--\eqref{eq:remainder-definition},
the bound \eqref{eq:remainder-density-bound} holds; in the fixed-target case,
one has \eqref{eq:remainder-fixed-P}.  Moreover,
\begin{equation}\label{eq:KS-variation}
\begin{split}
 \MapE(u_\phi)+\MapE(v_\phi)-\MapE(u)-\MapE(v)
 \le{}&-\frac{1}{n+2}
 \int_\Omega\ip{\nabla\phi}{\nabla((1-2\phi)w)}\,\dd\m\\
 &-2\int_\Omega(\phi-\phi^2)R_{u,v}\,\dd\m.
\end{split}
\end{equation}
\end{theorem}

\begin{proof}
Fix $x,y$ and write
\[
 p=\phi(x),\quad q=\phi(y),\quad
 a=\dist_Y(u(x),u(y)),\quad b=\dist_Y(v(x),v(y)).
\]
The \CAT(0) four-point calculation used in
\cite[proof of Lemma~3.2]{ZhangZhu2026}, based on the inequalities
\eqref{eq:CAT0-strong-convexity} and \eqref{eq:Reshetnyak-four-point}
recalled in \cref{subsec:CAT0-barycenters}, gives
\begin{align}
 &\dist_Y^2(u_\phi(x),u_\phi(y))
 +\dist_Y^2(v_\phi(x),v_\phi(y))-a^2-b^2 \notag\\
 &\quad\le
 -(p-q)\bigl[(1-2p)w(x)-(1-2q)w(y)\bigr] \notag\\
 &\qquad-2p(1-p)(b-a)^2+|p-q|(b-a)^2.
 \label{eq:finite-scale-CAT-variation}
\end{align}
Indeed, apply the same strong-convexity inequality successively to
the two geodesics joining $u(x)$ to $v(x)$ and $u(y)$ to $v(y)$, with
parameters $p$ and $q$, and collect the mixed terms.  The last term is
precisely the error caused by the unequal parameters.
The last term is an $o(1)$ error after integration: since
$|p-q|\le\Lip(\phi)r$ on the support of $\kappa_r$,
\begin{equation}
\begin{aligned}
 &\iint\kappa_r|\phi(x)-\phi(y)|(b-a)^2\,
     \dd\m(y)\,\dd\m(x)\\
 &\qquad\le2\Lip(\phi)r
   \left(\int e_r[u]+\int e_r[v]\right)\longrightarrow0.
\end{aligned}
\end{equation}
Integrating \eqref{eq:finite-scale-CAT-variation} gives
\begin{align}
 &\MapE_r(u_\phi)+\MapE_r(v_\phi)-\MapE_r(u)-\MapE_r(v)\notag\\
 &\quad\le-2B_r\bigl(\phi,(1-2\phi)w\bigr)
 -2\int(\phi-\phi^2)R_r^{u,v}\,\dd\m+o(1),
\end{align}
where \(R_r^{u,v}\) is the finite-scale remainder in the theorem statement.
Let $r\downarrow0$ through rational radii.  The four energy terms converge by
\eqref{eq:KS-energy-definition} from \cref{subsec:KS-energy}; the scalar term
converges by Part~(c) of \cref{prop:KS-properties}, equivalently
\eqref{eq:scalar-polarization-limit}; and Fatou's lemma gives
\[
 \limsup_{r\downarrow0}
 \left[-\int(\phi-\phi^2)R_{r}^{u,v}\,\dd\m\right]
 \le-\int(\phi-\phi^2)R_{u,v}\,\dd\m.
\]
This proves the theorem.
\end{proof}

\subsection{The two-flow parabolic inequality}

\begin{lemma}[\EVI{} finite-difference inequality]\label{lem:EVI-finite-difference}
Let $u^t$ and $v^t$ be \EVI{} harmonic map heat flows, possibly in different
fixed-trace classes, and put $w^t=\dist_Y^2(u^t,v^t)$.  For $0<h<t$ and
$0\le\phi\le1$ compactly supported in the interior,
\begin{align}
 \int_\Omega\phi(w^t-w^{t-h})\,\dd\m
 \le{}&\dist_{L^2}^2(u^t,u^{t-h})+\dist_{L^2}^2(v^t,v^{t-h})\notag\\
 &+(n+2)h\bigl[\MapE((u^t)_\phi)-\MapE(u^t)
 +\MapE((v^t)_\phi)-\MapE(v^t)\bigr].
\end{align}
\end{lemma}

\begin{proof}
Write \(u^t_{\phi}:=(u^t)_\phi\) and
\(v^t_{\phi}:=(v^t)_\phi\).  The pointwise \(\CAT(0)\) quadrilateral
inequality gives
\begin{align}
 \int_\Omega\phi(w^t-w^{t-h})\,\dd\m
 \le{}&\dist_{L^2}^2(u^t,u^{t-h})
       +\dist_{L^2}^2(u^t_{\phi},u^t)
       -\dist_{L^2}^2(u^t_{\phi},u^{t-h})\notag\\
 &+\dist_{L^2}^2(v^t,v^{t-h})
       +\dist_{L^2}^2(v^t_{\phi},v^t)
       -\dist_{L^2}^2(v^t_{\phi},v^{t-h}).
 \label{eq:EVI-quadrilateral-step}
\end{align}
Indeed, apply \eqref{eq:CAT0-strong-convexity} from
\cref{subsec:CAT0-barycenters} to the point \(u^t_{\phi}(x)\) on the
geodesic from \(u^t(x)\) to \(v^t(x)\), repeat the argument with
\(v^t_{\phi}(x)\), and add the two inequalities.  The remaining cross terms
are controlled by \eqref{eq:Reshetnyak-four-point} from
\cref{subsec:CAT0-barycenters};
see also \cite[Lemma~3.3]{ZhangZhu2026}.

Integrating \eqref{eq:EVI-HMHF} from
\cref{subsec:semigroup-HMHF} from \(t-h\) to \(t\) for the flow \(u\), with
the fixed competitor \(u^t_{\phi}\), and using the energy monotonicity
\eqref{eq:EVI-energy-monotonicity} from the same subsection gives
\begin{equation}
 \dist_{L^2}^2(u^t_{\phi},u^t)
 -\dist_{L^2}^2(u^t_{\phi},u^{t-h})
 \le 2h\bigl[\FlowE(u^t_{\phi})-\FlowE(u^t)\bigr].
 \label{eq:EVI-u-integrated-step}
\end{equation}
The estimate \eqref{eq:EVI-u-integrated-step} also holds with \(u\) replaced
by \(v\).  By
\cref{lem:variation-admissible}, the two competitors belong to their
respective fixed-trace classes; equality of the traces is not needed.
Substituting these two estimates into \eqref{eq:EVI-quadrilateral-step} and
using the normalization \eqref{eq:flow-energy-normalization} from
\cref{subsec:main-result} proves the claim.
\end{proof}

\begin{theorem}[Strong two-flow parabolic inequality]\label{thm:RCD-two-flow}
Let $u^t,v^t$ be bounded semigroup harmonic map heat flows, possibly in
different fixed-trace classes.  On every positive-time interior cylinder,
\begin{equation}\label{eq:RCD-two-flow-PDE}
 (\boldsymbol\Delta-\partial_t)\dist_Y^2(u,v)
 \ge2(n+2)R_{u^t,v^t}
\end{equation}
in the distributional sense.  In particular,
$(\boldsymbol\Delta-\partial_t)\dist_Y^2(u,v)\ge0$.
\end{theorem}

\begin{proof}
Fix \(0<t_1<t_2<\infty\) inside the positive-time interval under
consideration.
By \cref{prop:space-time-KS-regularity,prop:KS-properties},
$w=\dist_Y^2(u,v)$ belongs locally to the spatial Sobolev class required in
the weak formulation.  Moreover, the $\dist_{L^2}$ metric-speed bounds give local
$L^1$ continuity in time.  The function
$R_{u^t,v^t}(x)$ is measurable and locally integrable by
\cref{thm:KS-variation}, specifically the remainder bound
\eqref{eq:remainder-density-bound}, and the energy monotonicity recalled in
\cref{subsec:semigroup-HMHF}.  Let \(L_u\) and \(L_v\) be the metric-speed
bounds \eqref{eq:EVI-L} from that subsection on a slightly larger
positive-time interval.

Take nonnegative \(\zeta\in\Lip_c(\Omega)\) and
\(\theta\in C_c^1((t_1,t_2))\).  For all sufficiently small \(h>0\), put
\(\phi_h=\sqrt h\,\zeta\).  Combining
\cref{lem:EVI-finite-difference,thm:KS-variation} and dividing by
\(h\sqrt h\) gives, for almost every \(t\),
\begin{equation}\label{eq:two-flow-scaled-difference}
\begin{split}
 \frac1h\int_\Omega\zeta(w^t-w^{t-h})\,\dd\m
 \le{}&\frac{\dist_{L^2}^2(u^t,u^{t-h})
              +\dist_{L^2}^2(v^t,v^{t-h})}{h^{3/2}}\\
 &-\int_\Omega
   \ip{\nabla\zeta}{\nabla\bigl((1-2\sqrt h\,\zeta)w^t\bigr)}\,\dd\m\\
 &-2(n+2)\int_\Omega
   (\zeta-\sqrt h\,\zeta^2)R_{u^t,v^t}\,\dd\m.
\end{split}
\end{equation}
The metric-speed estimate gives
\begin{equation}\label{eq:two-flow-speed-error}
 \frac{\dist_{L^2}^2(u^t,u^{t-h})
              +\dist_{L^2}^2(v^t,v^{t-h})}{h^{3/2}}
 \le (L_u^2+L_v^2)\sqrt h\longrightarrow0.
\end{equation}
After multiplying \eqref{eq:two-flow-scaled-difference} by \(\theta(t)\)
and integrating in time, the remaining \(\sqrt h\)-terms vanish by the local
Sobolev bound for \(w\) and the local integrability of \(R_{u^t,v^t}\).  Hence
\begin{align}
 &-\iint\theta\ip{\nabla\zeta}{\nabla\bigl((1-2\sqrt h\,\zeta)w\bigr)}
       \,\dd\m\,\dd t
 \longrightarrow-\iint\theta\ip{\nabla\zeta}{\nabla w}\,\dd\m\,\dd t,
 \label{eq:two-flow-spatial-limit}\\
 &-2(n+2)\iint\theta(\zeta-\sqrt h\,\zeta^2)R_{u^t,v^t}
       \,\dd\m\,\dd t
 \longrightarrow-2(n+2)\iint\theta\zeta R_{u^t,v^t}\,\dd\m\,\dd t.
 \label{eq:two-flow-remainder-limit}
\end{align}
The backward difference satisfies
\[
 \frac1h\int_{t_1}^{t_2}\theta(t)
 \int\zeta(w^t-w^{t-h})\,\dd\m\,\dd t
 \longrightarrow-\iint w\zeta\theta'\,\dd\m\,\dd t.
\]
Passing to the limit in \eqref{eq:two-flow-scaled-difference} using
\eqref{eq:two-flow-speed-error}, \eqref{eq:two-flow-spatial-limit},
\eqref{eq:two-flow-remainder-limit}, and the preceding backward-difference
limit gives
\[
 -\iint\theta\ip{\nabla w}{\nabla\zeta}\,\dd\m\,\dd t
 +\iint w\zeta\theta'\,\dd\m\,\dd t
 \ge2(n+2)\iint\theta\zeta R_{u^t,v^t}\,\dd\m\,\dd t,
\]
which proves the asserted inequality for tests of the form
\(\zeta(x)\theta(t)\).  To obtain the full
space--time formulation, first mollify a compactly supported Lipschitz test in
time and then approximate its time dependence by piecewise-linear
interpolation.  The resulting finite sums of separated tests converge in all
terms of the weak formulation: the tests converge uniformly on their common
compact support, their spatial gradients converge in \(L^2\), and their time
derivatives converge in \(L^1\).  The spatial Sobolev bound for \(w\) and the local
integrability of \(R_{u^t,v^t}\) established above allow passage to the limit.
\end{proof}

\begin{corollary}[Fixed-target inequalities]\label{cor:fixed-target}
For every $P\in Y$,
\begin{align}
 (\boldsymbol\Delta-\partial_t)\dist_Y^2(u,P)&\ge2(n+2)e_{u^t},
 \label{eq:RCD-square-P}\\
 (\boldsymbol\Delta-\partial_t)\dist_Y(u,P)&\ge0
 \label{eq:RCD-distance-P}
\end{align}
in distributions on positive-time interior cylinders.
\end{corollary}

\begin{proof}
The constant map $v^t\equiv P$ is the stationary \EVI{} flow in the fixed-trace
class with constant boundary datum \(P\).  Applying
\cref{thm:RCD-two-flow} and using the fixed-target identity from
\cref{thm:KS-variation}, specifically \eqref{eq:remainder-fixed-P}, gives the
first inequality; compare \cite[Corollary~3.5]{ZhangZhu2026}.  The two trace
classes need not coincide.

Put $r_P=\dist_Y(u,P)$ and $g_\delta=(r_P^2+\delta)^{1/2}$.  By
\cref{prop:KS-properties}\,(a), specifically
\eqref{eq:fixed-P-upper-gradient}, $|\mathrm D r_P|^2\le(n+2)e_{u^t}$.  The
measure-valued \RCD{} chain rule
\cite[Proposition~4.11]{Gigli2015} gives
\[
 (\boldsymbol\Delta-\partial_t)g_\delta
 \ge \frac{(n+2)e_{u^t}}{g_\delta}
      -\frac{r_P^2|\mathrm D r_P|^2}{g_\delta^3}
 \ge \frac{(n+2)\delta e_{u^t}}{g_\delta^3}\ge0.
\]
Letting $\delta\downarrow0$ in the weak formulation proves
the distance inequality.
\end{proof}

Mayer \cite[Theorem~3.4]{Mayer1998} proved that, on a relatively compact
Riemannian domain with fixed boundary values, the harmonic map heat flow
converges in $L^2$ to the unique energy-minimizing harmonic map.  Guo
\cite[Theorem~1.6]{Guo2021} extended this conclusion to singular metric
sources satisfying his property~B.  The required coercivity, lower
semicontinuity, and compactness hold in the present finite-dimensional
$\RCD(K,N)$ setting by the PI (doubling and Poincar\'e)  structure and the
Sobolev theory in
\cref{subsec:KS-energy,subsec:semigroup-HMHF}.  For a relatively compact
fixed-boundary domain with nonempty exterior, the required zero-boundary
Poincar\'e estimate follows by extending a function in $W^{1,2}_0(\Omega)$ by
zero and applying the Poincar\'e inequality on a containing ball.

For the harmonic limit $u_\infty$, the two-flow inequality shows that
$\dist_Y^2(u,u_\infty)$ is a nonnegative heat subsolution.  The scalar local
boundedness estimate then gives local uniform convergence from the $L^2$
convergence.

\begin{corollary}[Local uniform convergence to the harmonic map]
\label{cor:local-uniform-to-hm}
Let $\psi\in\KS^{1,2}(\Omega,Y)$, let
$u_0\in\KS^{1,2}_\psi(\Omega,Y)$ have bounded essential image in $Y$, and
let $u^t$ be the corresponding harmonic map heat flow.  Assume that
$\Omega$ is relatively compact in $X$ and
$X\setminus\overline\Omega\ne\varnothing$.
Then the fixed-trace class contains a unique energy-minimizing harmonic map
$u_\infty\in\KS^{1,2}_\psi(\Omega,Y)$, and, for every ball
$B_{2R}(x_0)\Subset\Omega$,
\begin{equation}\label{eq:local-uniform-to-hm}
 \lim_{t\to\infty}\sup_{x\in B_R(x_0)}
 \dist_Y\bigl(u^t(x),u_\infty(x)\bigr)=0.
\end{equation}
Here the supremum is taken using the continuous representatives of $u^t$ and
$u_\infty$.
\end{corollary}

\begin{proof}
By \cite[Theorem~3.4]{Mayer1998} and
\cite[Theorem~1.6]{Guo2021}, there is a unique minimizer $u_\infty$ and
\begin{equation}\label{eq:L2-convergence-to-hm}
 \dist_{L^2}(u^t,u_\infty)\longrightarrow0
 \qquad\text{as }t\to\infty.
\end{equation}
The projection argument preceding \eqref{eq:bounded-image-ball} shows that
$u_\infty$ lies in the same bounded target ball as $u_0$.  Since a minimizer
is a stationary \EVI{} flow, \cref{thm:RCD-two-flow} shows that
$q(x,t):=\dist_Y^2(u(x,t),u_\infty(x))$ is also a nonnegative heat subsolution.

Take the continuous representatives provided by
\cref{thm:RCD-holder-map-flow} and the harmonic-map regularity theory
\cite{Gigli2023,MondinoSemola2026}.  For $B_{2R}(x_0)\Subset\Omega$ and
$T>4R^2$, apply \eqref{eq:expanded-local-boundedness} on
$Q_{2R}^-((x_0,T))$.  Continuity up to the top time gives
\begin{align*}
 \sup_{x\in  B_R(x_0)}q(x,T)
 &\le \frac{C}{R^2\m(B_{2R}(x_0))}
 \int_{T-4R^2}^{T}
 \int_{B_{2R}(x_0)}q(x,t)\,\dd\m\,\dd t\\
 &\le \frac{4C}{\m(B_{2R}(x_0))}
 \max_{s\in[T-4R^2,T]}
 \dist_{L^2}^2(u^s,u_\infty).
\end{align*}
The right-hand side tends to zero by \eqref{eq:L2-convergence-to-hm}, which
proves \eqref{eq:local-uniform-to-hm}.
\end{proof}

\section{Positive-time regularity}\label{sec:positive-time}

Local H\"older regularity follows from the fixed-target square inequality and
the $L^2$ metric-speed bound, while pointwise Lipschitz continuity in time
follows from the two-flow inequality applied to time translates.  These are
the \RCD{} counterparts of the arguments in
\cite[Sections~4--5]{ZhangZhu2026}; the Campanato step is independent of
pointwise time regularity.

\subsection{Local H\"older continuity}

\begin{theorem}[Local H\"older representative]\label{thm:RCD-holder-map-flow}
The bounded harmonic map heat flow has a locally H\"older continuous
representative on $\Omega\times(0,\infty)$.  On every compact positive-time
cylinder there are $\alpha\in(0,1)$ and $C<\infty$ such that
\begin{equation}\label{eq:RCD-holder}
 \dist_Y(u(x,t),u(y,s))
 \le C\bigl(\dist(x,y)+|t-s|^{1/2}\bigr)^\alpha.
\end{equation}
\end{theorem}

\begin{proof}
Fix compact cylinders
\(Q_0\Subset Q_1\Subset\Omega\times(0,\infty)\), and choose \(r_0>0\) so
that \(Q_{8r}(z_0)\Subset Q_1\) whenever
\(z_0\in Q_0\) and \(0<r<r_0\).  Write \(e=e_{u^t}\).  The positive-time
semigroup estimates in \cref{subsec:semigroup-HMHF}, specifically
\eqref{eq:EVI-L} and \eqref{eq:bounded-image-ball}, provide a metric-speed
bound \(L\) on the time projection of \(Q_1\) and give
\(\dist_Y(u,P_0)\le M\) on \(Q_1\).

For \(P\in\overline{B}_{2M}(P_0)\), set
\begin{equation}\label{eq:VP-definitions}
\begin{split}
 A_P(r;z_0)&:=\esssup_{Q_r(z_0)}\dist_Y^2(u,P),\\
 B_P(r;z_0)&:=\fint_{Q_r(z_0)}\dist_Y^2(u,P)\,\dd\m\,\dd t.
\end{split}
\end{equation}
For the quantities in \eqref{eq:VP-definitions}, there exist
\(\delta\in(0,1)\) and \(C<\infty\), depending only on the fixed localization
data, such that
\begin{align}
 A_P(r;z_0)
 &\le(1-\delta)A_P(4r;z_0)+\delta B_P(r;z_0)+CMLr^2,
 \label{eq:holder-oscillation-step}\\
 r^2\fint_{Q_{r/2}(z_0)}e\,\dd\m\,\dd t
 &\le C\bigl[A_P(4r;z_0)-A_P(r;z_0)+MLr^2\bigr],
 \label{eq:holder-energy-step}\\
 \fint_{Q_r(z_0)}\dist_Y^2(u,\bar u_{Q_r(z_0)})\,\dd\m\,\dd t
 &\le Cr^2\fint_{Q_{2r}(z_0)}e\,\dd\m\,\dd t+CMLr^2.
 \label{eq:holder-poincare-step}
\end{align}
To obtain \eqref{eq:holder-oscillation-step}, apply Part~(ii) of
\cref{prop:scalar-parabolic-estimates}, namely
\eqref{eq:expanded-weak-harnack}, to
\(A_P(4r;z_0)-\dist_Y^2(u,P)\), which is a nonnegative heat subsolution by
\cref{cor:fixed-target},
specifically \eqref{eq:RCD-square-P}.  Estimate
\eqref{eq:holder-energy-step} follows by testing
\eqref{eq:RCD-square-P} with the
cutoff from \cref{prop:RCD-cutoffs}.  Finally,
\eqref{eq:holder-poincare-step} follows from
\cref{prop:KS-properties}\,(d), specifically \eqref{eq:metric-Poincare}, and
the estimate \(\dist_{L^2}(u^t,u^s)\le L|t-s|\).
These are the \RCD{} forms of \cite[Lemmas~4.2--4.4]{ZhangZhu2026}.

We now derive the contraction estimate.  Fix \(\varepsilon=1/64\), and choose
\(m\in\mathbb N\) so that \((1-\delta)^m\le\varepsilon\).  For
\(0<\rho<\varepsilon^m r\), put
\(P_1:=\bar u_{Q_{r/4}(z_0)}\) and \(P:=\bar u_{Q_\rho(z_0)}\).
Both barycenters belong to \(\overline{B}_M(P_0)\).  The energy and Poincar\'e
estimates \eqref{eq:holder-energy-step}--\eqref{eq:holder-poincare-step},
local doubling, and the minimizing property of the barycenter give
\begin{align*}
 B_{P_1}(r';z_0)
 &\le C(m)\fint_{Q_{r/4}(z_0)}\dist_Y^2(u,P_1)\,\dd\m\,\dd t\\
 &\le Cr^2\fint_{Q_{r/2}(z_0)}e\,\dd\m\,\dd t+Cr^2\\
 &\le C_0\bigl[A_P(4r;z_0)-A_P(r;z_0)\bigr]+C_1r^2
\end{align*}
whenever \(\varepsilon^m r<r'\le r/4\).  Here the first line uses local
doubling, the second is \eqref{eq:holder-poincare-step} at scale \(r/4\),
and the last is \eqref{eq:holder-energy-step}.  Hence
\begin{equation}\label{eq:holder-average-decay}
 B_{P_1}(r';z_0)
 \le C_0\bigl[A_P(4r;z_0)-A_P(r;z_0)\bigr]+C_1r^2,
 \qquad \varepsilon^m r<r'\le r/4.
\end{equation}
The barycenter \(P\) lies in the closed convex hull of the essential image
of \(u\) on \(Q_\rho(z_0)\subset Q_{\varepsilon^m r}(z_0)\).  Since balls
are convex in a \(\CAT(0)\) space,
\(\dist_Y(P,P_1)\le A_{P_1}(\varepsilon^m r;z_0)^{1/2}\), and the triangle
inequality gives
\begin{equation}\label{eq:holder-barycenter-comparison}
 A_P(\varepsilon^m r;z_0)
 \le4A_{P_1}(\varepsilon^m r;z_0),
 \qquad A_{P_1}(r;z_0)\le4A_P(r;z_0).
\end{equation}
The second inequality follows in the same way because
\(P_1\) lies in the closed convex hull of the essential image on
\(Q_{r/4}(z_0)\subset Q_r(z_0)\).  The iterated form of
\eqref{eq:holder-oscillation-step} gives, for some
\(r'\in(\varepsilon^m r,r/4]\),
\[
 A_{P_1}(\varepsilon^m r;z_0)
 \le\varepsilon A_{P_1}(r;z_0)+B_{P_1}(r';z_0)+Cr^2.
\]
Using \eqref{eq:holder-average-decay},
\eqref{eq:holder-barycenter-comparison}, and \(\varepsilon=1/64\), we obtain
\begin{equation}\label{eq:holder-precontraction}
 A_P(\varepsilon^m r;z_0)
 \le\frac14A_P(r;z_0)
   +C_0\bigl[A_P(4r;z_0)-A_P(r;z_0)\bigr]+C_1r^2.
\end{equation}
By monotonicity,
\(C_0A_P(\varepsilon^m r;z_0)\le C_0A_P(r;z_0)\).  Adding this inequality
to \eqref{eq:holder-precontraction} gives
\[
 (1+C_0)A_P(\varepsilon^m r;z_0)
 \le(C_0+1/4)A_P(4r;z_0)+C_1r^2.
\]
Consequently,
\begin{equation}\label{eq:holder-contraction}
 A_P(\varepsilon^m r;z_0)
 \le\beta A_P(4r;z_0)+C_2r^2,
 \qquad
 \beta:=\frac{C_0+1/4}{C_0+1}<1.
\end{equation}
Set \(\lambda:=\varepsilon^m/4\).  The discrete Campanato iteration in
\cite[proof of Theorem~4.1]{ZhangZhu2026}, applied successively to
\eqref{eq:holder-contraction} with \(R=4r\), gives
\[
 A_P(\lambda^kR;z_0)
 \le\beta^kA_P(R;z_0)
   +CR^2\sum_{j=0}^{k-1}\beta^{k-1-j}\lambda^{2j}.
\]
Consequently, for some \(\gamma\in(0,2)\), uniformly in \(z_0\in Q_0\),
\begin{equation}\label{eq:holder-sup-decay}
 A_{\bar u_{Q_\rho(z_0)}}(\rho;z_0)\le C\rho^\gamma,
 \qquad 0<\rho<r_0.
\end{equation}
By \eqref{eq:holder-sup-decay}, the mean-square Campanato excess satisfies
the same decay.
The metric Campanato embedding, equivalently the barycenter-chain argument in
\cite[proof of Theorem~4.1]{ZhangZhu2026}, now gives
\[
 \begin{aligned}
 \dist_Y(u(z),u(z'))
 &\le C\bigl(\dist(x,x')+|t-t'|^{1/2}\bigr)^{\gamma/2},\\
 &\hspace{25mm}z=(x,t),\quad z'=(x',t')\in Q_0.
 \end{aligned}
\]
Since \(Q_0\) is arbitrary, this proves the theorem after renaming
\(\gamma/2\) as \(\alpha\).
\end{proof}

\subsection{Lipschitz continuity in time}

\begin{theorem}[Pointwise time-Lipschitz representative]
\label{thm:RCD-time-Lipschitz}
Let $B_{16R}(\bar x)\Subset\Omega$ and $0<t_*<T$.  On the representative in
\cref{thm:RCD-holder-map-flow}, there is a constant $c_L<\infty$ such that
\begin{equation}\label{eq:pointwise-time-Lip}
 \dist_Y(u(x,t),u(x,s))\le c_L|t-s|
\end{equation}
for every $x\in B_{2R}(\bar x)$ and $s,t\in[t_*,T]$.  The constant depends
only on the local \RCD{} data, $R,t_*,T$, and the metric-speed bound $L$ in
\cref{subsec:semigroup-HMHF}, specifically \eqref{eq:EVI-L}; more precisely,
$c_L\le C_{\mathrm{loc}}L$, where
$C_{\mathrm{loc}}$ depends on the fixed localization radii, the local
doubling--Poincar\'e constants, and the corresponding volume normalization.
\end{theorem}

\begin{proof}
For $h>0$ set \(w_h(x,t):=\dist_Y^2(u(x,t+h),u(x,t))\).
The semigroup property makes $u(\cdot,t+h)$ another harmonic map heat flow.
The time-translation step in \cite[proof of Theorem~5.1]{ZhangZhu2026},
with \cref{thm:RCD-two-flow} in place of its smooth-source counterpart, gives
$(\boldsymbol\Delta-\partial_t)w_h\ge0$.  Cover
$B_{2R}(\bar x)\times[t_*,T-h]$ by finitely many backward cylinders of one
fixed radius whose doubled cylinders lie in
$B_{4R}(\bar x)\times(t_*/2,T+1)$.  This radius, and hence the constants in
the scalar estimate below, depend only on the separation of
\(B_{2R}(\bar x)\times[t_*,T]\) from the boundary of the enlarged cylinder,
not on \(h\).  On each enlarged cylinder,
\[
 \int_\Omega w_h(\cdot,t)\,\dd\m
 =\dist_{L^2}^2(u^{t+h},u^t)\le L^2h^2.
\]
Part~(i) of \cref{prop:scalar-parabolic-estimates}, specifically
\eqref{eq:expanded-local-boundedness}, therefore yields
\[
 \esssup_{B_{2R}(\bar x)\times[t_*,T-h]}w_h\le C L^2h^2.
\]
Since $u$ is continuous by \cref{thm:RCD-holder-map-flow}, the essential
supremum is the ordinary supremum.  Taking square roots proves
\eqref{eq:pointwise-time-Lip} with $c_L=C^{1/2}L$.
The estimate first holds for $0<h\le h_0$, where $h_0>0$ is
fixed by the temporal separation from the enlarged cylinder.  For arbitrary
$s,t\in[t_*,T]$, subdivide the interval between them into finitely many
subintervals of length at most $h_0$ and sum the resulting estimates.  The
same constant $c_L$ is retained because the lengths of the subintervals add
to $|t-s|$.
\end{proof}

\section{Fixed-target inequalities and heat-kernel good sets}\label{sec:fixed-target-good-set}

The fixed-time distance and spatial-scale estimates below adapt
\cite[Section~5]{ZhangZhu2026} to an \RCD{} source.  The additional issues are the
choice of exceptional sets simultaneously in the target variable, recovery
of the energy density from the heat kernel, and localization by
Gaussian tails.

\subsection{Fixed-target scalar inequalities}

The representative from \cref{thm:RCD-holder-map-flow} is continuous on the
separable space \(\Omega\times(0,\infty)\), so its range is separable.  Let
\(Y_u\) be the closed convex hull of this range.  Then \(Y_u\) is separable:
starting from a countable dense subset of the range, take successively all
dyadic points on geodesic segments between points already obtained and then
take the closure.  This is the usual essentially separably valued convention
for metric-valued $L^2$ maps; see \cite{KorevaarSchoen1993,Ohta2004}.  Fix a
countable dense family \(\{P_j\}_{j\ge1}\subset Y_u\).

\begin{proposition}[Fixed-time scalar distance inequality]\label{prop:simultaneous-P}
Let \(U\Subset\Omega\) and \(I\Subset(0,\infty)\), and suppose that
the estimate from \cref{thm:RCD-time-Lipschitz}, specifically
\eqref{eq:pointwise-time-Lip}, holds on \(U\times I\) with constant \(c_L\).
There is a full-measure set
\(T_{\mathrm{dist}}=T_{\mathrm{dist}}(U,I)\subset I\) such that for every
\(t\in T_{\mathrm{dist}}\) and every \(P\in Y_u\),
\begin{equation}
 \boldsymbol\Delta\dist_Y(u^t,P)\ge-c_L\m
\end{equation}
on \(U\).  The exceptional set is independent of \(P\).
\end{proposition}

\begin{proof}
Fix a compact ball \(B\Subset U\), one
target point \(P\in Y_u\), and
write \(r_P(x,t)=\dist_Y(u(x,t),P)\).  Testing the fixed-target distance
inequality in \cref{cor:fixed-target}, namely \eqref{eq:RCD-distance-P}, with a product
\(\eta(x)\chi(t)\), where \(\eta\ge0\) is Lipschitz and compactly supported
in \(B\), and using the time-Lipschitz estimate
\eqref{eq:pointwise-time-Lip} from
\cref{thm:RCD-time-Lipschitz} gives
 \[
 \int\chi(t)\left[-\int_B\ip{\nabla r_P(\cdot,t)}{\nabla\eta}\,\dd\m
 -\int_B\eta\,\partial_t r_P(\cdot,t)\,\dd\m\right]\dd t\ge0.
 \]
Here \(\partial_t r_P\) exists almost everywhere and
\(|\partial_t r_P|\le c_L\).  The fundamental lemma in the time variable
therefore gives, for almost every \(t\),
\[
 -\int_B\ip{\nabla r_P(\cdot,t)}{\nabla\eta}\,\dd\m
 \ge\int_B\eta\,\partial_t r_P(\cdot,t)\,\dd\m
 \ge-c_L\int_B\eta\,\dd\m.
\]
Taking countable dense families of
spatial tests and of target points \(P_j\) yields one full-measure set of times
on which
\[
 -\int_B\ip{\nabla r_{P_j}(\cdot,t)}{\nabla\eta}\,\dd\m
 \ge -c_L\int_B\eta\,\dd\m
\]
for every \(j\) and every nonnegative test \(\eta\).  Equivalently, the
functional
\[
 \eta\longmapsto
 -\int_B\ip{\nabla r_{P_j}(\cdot,t)}{\nabla\eta}\,\dd\m
 +c_L\int_B\eta\,\dd\m
\]
is nonnegative on nonnegative tests.  On each compact set, this functional is
bounded in the supremum norm: if \(0\le\eta\le1\) is supported there, choose
one nonnegative Lipschitz cutoff \(\chi\) that equals \(1\) on that set and
use positivity to bound the functional at \(\eta\) by its value at \(\chi\).
The Riesz--Markov theorem therefore constructs the measure-valued Laplacian
and gives \(\boldsymbol\Delta r_{P_j}\ge-c_L\m\).

Now let \(P_j\to P\).  The functions \(r_{P_j}\) converge uniformly to
\(r_P\) on the bounded image cylinder and satisfy
\(|\mathrm D r_{P_j}|^2\le(n+2)e_{u^t}\).  They are therefore uniformly bounded in
 \(W^{1,2}(B)\).  Strong \(L^2\)-convergence, the closure and minimality
properties \cite[Lemma~4.3(b) and Lemma~4.4]{AmbrosioGigliSavare2014}, and
weak compactness of gradients pass the preceding inequality to \(r_P\).  A countable
exhaustion of \(U\) proves the statement simultaneously for all
\(P\in Y_u\).
\end{proof}

\subsection{Two-point estimates}\label{subsec:two-point-estimates}

Throughout this subsection, write
\(\Z:=X\times X\) and \(\m_{\Z}:=\m\otimes\m\), and let \(\DeltaZ\) denote
the measure-valued Laplacian on \(\Z\).
For $x,y\in\Omega$ and $t>0$, set
\begin{equation}\label{eq:F-definition}
 F(x,y,t):=\dist_Y(u(x,t),u(y,t)).
\end{equation}

\begin{proposition}[Two-point spatial bound]\label{prop:F-two-point}
For every $t\in T_{\mathrm{dist}}$,
\begin{equation}\label{eq:F-slice-lower}
 \DeltaZ F(\cdot,\cdot,t)
 \ge-2c_L\,\m_{\Z}.
\end{equation}
Here the measure inequality is asserted on \(U\times U\), with
\(U,I,T_{\mathrm{dist}}\) as in \cref{prop:simultaneous-P}.
Moreover,
\begin{equation}\label{eq:F-time-Lip}
 |F(x,y,t)-F(x,y,s)|\le2c_L|t-s|,
\end{equation}
for \(x,y\in U\) and \(s,t\in I\).
\end{proposition}

\begin{proof}
Fix \(t\in T_{\mathrm{dist}}\), relatively compact balls
\(B_1,B_2\Subset U\), and a nonnegative test
\(\Phi\in\Lip_c(B_1\times B_2)\).  Here
\(\boldsymbol\Delta_x\) and \(\boldsymbol\Delta_y\) denote the
measure-valued slice Laplacians in the first and second variables.  For a.e.
\(y\), the section
\(x\mapsto\Phi(x,y)\) is a nonnegative compactly supported Lipschitz test on
\(B_1\).  Since
\(u(y,t)\in Y_u\), \cref{prop:simultaneous-P} gives
\[
 \int_{B_2}\left[\int_{B_1}\Phi(x,y)\,\dd(\boldsymbol\Delta_xF(\cdot,y,t))(x)\right]\,\dd\m(y)
 \ge -c_L\iint\Phi\,\dd\m\,\dd\m.
\]
Repeating this calculation in the second variable, now testing with
\(y\mapsto\Phi(x,y)\) and fixed target \(u(x,t)\), gives the identical lower
bound for \(\boldsymbol\Delta_yF\).  The tensorized Cheeger calculus
\cite[Theorems~6.17--6.18]{AmbrosioGigliSavareDuke2014} identifies the
product gradient with the two partial gradients and justifies Fubini; its
interval-factor analogue also appears in the first author's joint work with
Gigli \cite[Theorem~3.7]{GigliHan2018}.  Adding the two inequalities therefore
gives the asserted slice bound.  The Sobolev regularity of \(F\) on the product
follows from \(|\mathrm D_x F|^2\le(n+2)e_{u^t}(x)\) and
\(|\mathrm D_y F|^2\le(n+2)e_{u^t}(y)\), which are the fixed-target upper-gradient
estimates in \eqref{eq:fixed-P-upper-gradient} of
\cref{prop:KS-properties}.
Finally, \cref{thm:RCD-time-Lipschitz} gives
\[
 |F(x,y,t)-F(x,y,s)|
 \le \dist_Y(u(x,t),u(x,s))+\dist_Y(u(y,t),u(y,s))
 \le2c_L|t-s|,
\]
which proves the time-Lipschitz estimate.
\end{proof}

\medskip

For a continuous map $v$ and $r>0$, define
\begin{equation}\label{eq:lip-r-definition}
 \lip_rv(x):=\sup_{0<s<r}\ \sup_{y\in B_s(x)}
 \frac{\dist_Y(v(x),v(y))}{s}.
\end{equation}

\begin{proposition}[\RCD{} spatial-scale energy estimate]\label{prop:lip-r-energy}
Let $B_{16R}(\bar x)\Subset\Omega$ and $0<t_*<T$.  For almost every
$t\in(t_*,T)$ and every $0<r<R/4$,
\begin{equation}\label{eq:lip-r-energy}
 \int_{B_R(\bar x)}[\lip_ru^t]^2\,\dd\m
 \le C\int_{B_{4R}(\bar x)}e_{u^t}\,\dd\m
 +C c_L^2R^2\m(B_R(\bar x)).
\end{equation}
Here \(c_L\) is the pointwise time-Lipschitz constant in
\cref{thm:RCD-time-Lipschitz}, specifically
\eqref{eq:pointwise-time-Lip}, on the corresponding enlarged cylinder.
\end{proposition}

\begin{proof}
Fix a time $t$ for which the fixed-target inequality in
\cref{prop:simultaneous-P} and the slice Sobolev estimates hold, and fix a
point $x$.  Throughout
the proof, scalar distance functions are evaluated using the continuous
representative from \cref{thm:RCD-holder-map-flow}.  Therefore the
essential supremum in Part~(iii) of
\cref{prop:scalar-parabolic-estimates} agrees with the
ordinary suprema on the smaller balls appearing below.  Put
$h_x(y)=\dist_Y(u(x,t),u(y,t))$.  By
\cref{prop:simultaneous-P},
$\boldsymbol\Delta h_x\ge-c_L\m$.  Hence Part~(iii) of
\cref{prop:scalar-parabolic-estimates}, specifically
\eqref{eq:elliptic-local-boundedness}, gives
\begin{equation}\label{eq:hx-local-bound}
 \sup_{B_s(x)}h_x
 \le C\left(\fint_{B_{2s}(x)}h_x\,\dd\m+c_Ls^2\right).
\end{equation}
With all balls contained in the fixed ambient localization ball, define the
truncated Hardy--Littlewood maximal operator by
\(\mathcal M_R g(x):=\sup_{0<\rho<R}\fint_{B_\rho(x)}|g|\,\dd\m\).
On a doubling Poincar\'e space, the Haj\l asz--Koskela pointwise Poincar\'e
estimate \cite[Theorem~3.2]{HajlaszKoskela2000} implies, for almost every $x$,
\[
 \fint_{B_{2s}(x)}h_x\,\dd\m
 \le Cs\bigl(\mathcal M_R(\sqrt{(n+2)e_{u^t}})(x)
 +\mathcal M_R(\mathcal M_R(\sqrt{(n+2)e_{u^t}}))(x)\bigr),
\]
because \(|\mathrm D h_x|^2\le(n+2)e_{u^t}\) by
\cref{prop:KS-properties}\,(a).  Divide
\eqref{eq:hx-local-bound} by $s$ and take the supremum over $s<r$:
\[
 \lip_ru^t(x)\le C\bigl(\mathcal M_R(\sqrt{e_{u^t}})(x)
 +\mathcal M_R(\mathcal M_R(\sqrt{e_{u^t}}))(x)+c_LR\bigr).
\]
The factor \(n+2\) is absorbed into the local constant in this estimate.
The localized Hardy--Littlewood maximal operator is bounded on $L^2$ in a
locally doubling space by
\cite[Theorem~14.13]{HajlaszKoskela2000}.  Squaring and integrating gives
the asserted energy bound; cf. \cite[Proposition~5.4]{ZhangZhu2026}.
\end{proof}

\subsection{Heat-kernel differentiation and full-measure good sets}

Here a \emph{good set} means a Borel subset of full measure, with respect to
the ambient spatial or space--time product measure, on which all pointwise
heat-kernel limits needed later hold simultaneously.  Thus the definition
removes one common null set and imposes no additional regularity
condition.

Let \(U_0\Subset U\Subset\Omega\) and \(I\Subset(0,\infty)\) be arbitrary,
and let $p_s(x,y)$ denote the \RCD{} heat kernel.  Put
$\mu_s^x:=\Hheat_s\delta_x=p_s(x,\cdot)\m$.  Extend $u^t$ by the point $P_0$
from \cref{subsec:semigroup-HMHF}, specifically
\eqref{eq:bounded-image-ball}, outside \(U\), and suppress this extension
from the notation.  The following lemma shows that all subsequent limits at points
of \(U_0\) are independent of this bounded extension.  Related heat-kernel
localization and short-time concentration estimates occur in
\cite[Sections~2.1, 2.2, and~2.6]{HanZhu2026}.

\begin{lemma}[Off-diagonal heat-kernel tails]\label{lem:RCD-tail}
Let $U\subset X$ be open and $x_0\in U$.  If
$r_0=\dist(x_0,X\setminus U)>0$, then for $j=0,1,2$,
\begin{equation}
 \frac{1}{s}\int_{X\setminus U}(1+\dist^j(x,x_0))
 p_s(x_0,x)\,\dd\m(x)\longrightarrow0.
\end{equation}
Let $V\subset X$ be open and $y_0\in V$, and set
\(r_1:=\dist(y_0,X\setminus V)\) and \(r_*:=\min\{r_0,r_1\}\).
If $r_1>0$ and $\Pi_s$ is any coupling of $\mu_s^{x_0}$ and
$\mu_s^{y_0}$, then
\[
 \Pi_s\bigl((U\times V)^c\bigr)
 \le\mu_s^{x_0}(U^c)+\mu_s^{y_0}(V^c)=o(s).
\]
Consequently every uniformly bounded localization error supported outside
$U\times V$ has $\Pi_s$-integral $o(s)$.
\end{lemma}

\begin{proof}
The Gaussian upper bound \eqref{eq:Gaussian} from
\cref{sec:prelim}, local doubling, and an annular
decomposition give the assertion.  For the first marginal the factor
$e^{-r_0^2/(Cs)}$, and for the second the factor $e^{-r_1^2/(Cs)}$, dominates
every inverse power of $s$; equivalently both errors are bounded by a
polynomial factor times $e^{-r_*^2/(Cs)}$.  The coupling assertion follows
from the two marginal estimates; see \cite{JiangLiZhang2016} and
\cite[Lemma~2.53]{MondinoSemolaWeakLaplacian2025}.
\end{proof}

\begin{lemma}[Heat-kernel recovery for continuous \KS{} maps]
\label{lem:continuous-KS-heat-recovery}
Let $W\Subset X$ and let
$w\in\KS^{1,2}_{\loc}(W,Y)\cap C(W,Y)$.  At $\m$-almost every $x\in W$,
\begin{equation}
 \lim_{s\downarrow0}\frac1s
 \Pheat_s\!\left[\dist_Y^2(w(\cdot),w(x))\right](x)
 =2(n+2)e_w(x).
\end{equation}
The heat flow may be applied to any bounded extension agreeing with $w$ near
$x$.
\end{lemma}

\begin{proof}
Choose $x$ that is $n$-regular, meaning that its unique measured tangent is
Euclidean \(\mathbb R^n\), is an approximate metric differentiability point
of $w$, and is a Lebesgue point of $e_w$.  Put $r=\sqrt{s}$.  Under the
rescaling $r^{-1}\dist$, the pointed measured space converges to
$\mathbb R^n$.  The approximate metric differentiability and energy-density
convergence in \cite[Section~3.1 and Theorem~3.13]{GigliTyulenev2021} give
the local $L^2$ convergence
\[
 r^{-1}\dist_Y(w(\cdot),w(x))\longrightarrow\md_xw;
\]
equivalently, the squares converge locally in $L^1$.  Continuity of $w$
identifies the representative and its value at $x$ and permits localization.
Let \(\Pheat_s^{\mathbb R^n}\) denote the Euclidean heat semigroup.
The stability of heat flows under pointed measured convergence
\cite[Theorems~5.7 and~6.11]{GigliMondinoSavare2015} handles the local part,
while the Gaussian bounds and \cref{lem:RCD-tail} make the contribution from
the complement of every fixed rescaled ball negligible; see also
\cite[Sections~2.1 and~2.6]{HanZhu2026}.  The heat-kernel recovery argument in
\cite[Proposition~3.3]{MondinoSemola2026} depends only on these facts and
therefore applies to the present continuous metric-valued Sobolev map.
Consequently,
\[
 \lim_{s\downarrow0}\frac1s
 \Pheat_s\!\left[\dist_Y^2(w(\cdot),w(x))\right](x)
 =\Pheat^{\mathbb R^n}_1[(\md_xw)^2](0).
\]
Since $\md_xw$ is one-homogeneous, the Euclidean Gaussian calculation gives
\[
 \Pheat^{\mathbb R^n}_1[(\md_xw)^2](0)
 =2(n+2)\fint_{B_1(0)}(\md_xw(z))^2\,\dd\mathcal L^n(z)
 =2(n+2)e_w(x).
\]
The independence of the extension follows from \cref{lem:RCD-tail}.
\end{proof}

\begin{proposition}[Time-shifted energy identity]\label{prop:time-shift-energy}
On every compact positive-time cylinder there is a Borel set of full
\(\m\otimes\mathcal L^1\)-measure such that, at each of its points
\((x_0,t_0)\),
\begin{equation}
 \lim_{s\downarrow0}\frac{1}{s}\int_Xp_s(x_0,y)
 \dist_Y^2(u(y,t_0-s),u(x_0,t_0))\,\dd\m(y)
 =2(n+2)e_{u^{t_0}}(x_0).
\end{equation}
\end{proposition}

\begin{proof}
For almost every \(t_0\), the slice \(u^{t_0}\) belongs locally to
$\KS^{1,2}$, and \cref{thm:RCD-holder-map-flow} supplies its
continuous representative.  Applying
\cref{lem:continuous-KS-heat-recovery} to this slice and then using Fubini's theorem
gives, for almost every \((x_0,t_0)\),
\begin{equation}\label{eq:heat-energy-recovery-prelim}
 \lim_{s\downarrow0}\frac{1}{s}\int_Xp_s(x_0,y)
 \dist_Y^2(u(y,t_0),u(x_0,t_0))\,\dd\m(y)
 =2(n+2)e_{u^{t_0}}(x_0).
\end{equation}
The localization is independent of the bounded extension by
\cref{lem:RCD-tail}.

Starting from \eqref{eq:heat-energy-recovery-prelim}, work on a ball whose
closure is contained in the open set on which the chosen extension agrees
with \(u\).  Then
\cref{thm:RCD-time-Lipschitz}, specifically
\eqref{eq:pointwise-time-Lip}, and
$|a^2-b^2|\le(a-b)^2+2b|a-b|$ bound the difference between the time-shifted
and fixed-time integrands by
\[
 (c_Ls)^2+2c_Ls\dist_Y(u(y,t_0),u(x_0,t_0)).
\]
After division by $s$, the first term vanishes and the heat average of the last
distance tends to zero by the H\"older continuity of $u^{t_0}$.  The exterior
part is $o(1)$ by \cref{lem:RCD-tail}.  Taking a Borel representative of the
resulting full-measure set completes the proof.
\end{proof}

Define the space--time semigroup
\begin{equation}
 (\mathscr S_\tau f)(x,t):=\Pheat_\tau[f(\cdot,t-\tau)](x).
\end{equation}
For fixed $\tau$, this operator applies the heat flow only in the spatial
variable and translates the time variable from $t$ to $t-\tau$; it does not
mollify the function in time.  The average over $0<\tau<s$ in the next
proposition is used solely as a differentiation basis for $L^1$ functions.

\begin{proposition}[Parabolic semigroup differentiation]\label{prop:parabolic-differentiation}
For $f\in L^1_{\loc}(X\times\R)$,
\begin{equation}\label{eq:parabolic-differentiation}
 \frac{1}{s}\int_0^s\mathscr S_\tau f(x,t)\,\dd\tau
 \longrightarrow f(x,t)
\end{equation}
for almost every $(x,t)$ on every compactly contained space--time cylinder.
\end{proposition}

\begin{proof}
Fix a compact cylinder \(Q'\Subset Q\) and a cutoff supported in \(Q\), and
extend the localized function by zero.  The operators
\(\mathscr S_\tau\) form a strongly continuous positive contraction semigroup
on the resulting global \(L^1\)-space: they preserve nonnegative functions,
satisfy \(\|\mathscr S_\tau f\|_{L^1}\le\|f\|_{L^1}\), and obey
\(\mathscr S_\tau f\to f\) in \(L^1\) as \(\tau\downarrow0\).  The
Hopf--Dunford--Schwartz maximal
theorem of Dunford--Schwartz \cite{DunfordSchwartz1958} and approximation by
bounded compactly supported functions give
the asserted differentiation limit almost everywhere on \(Q'\).  If two
localizations agree on a neighborhood of \(Q'\), then there is \(\eta>0\)
such that every \((y,s)\) in the support of their difference satisfies
\(\dist(x,y)+|t-s|^{1/2}\ge\eta\) for every \((x,t)\in Q'\).  The Gaussian heat-kernel tail
shows that its averaged contribution tends to zero.  Hence the resulting good
set and limit are independent of the chosen cutoff.
\end{proof}

\begin{lemma}[Localized Duhamel inequality]
\label{lem:Duhamel-boundary-tail}
Suppose
$q\in C(U\times I)\cap L^2_{\loc}(I;W^{1,2}_{\loc}(U))\cap L^\infty_{\loc}$
and \(a\in L^1_{\loc}(U\times I)\), and suppose
\((\boldsymbol\Delta-\partial_t)q\le a\) weakly.  If $B_{4r}(x)\Subset U$ and
$(x,t)$ is separated from the time
boundary, then
\begin{equation}\label{eq:Duhamel-conclusion}
 \Pheat_s[q(\cdot,t-s)](x)-q(x,t)
 \le\int_0^s\Pheat_\tau[a(\cdot,t-\tau)](x)\,\dd\tau+o(s).
\end{equation}
The error is locally uniform under uniform bounds for $q$ and its local
$L^2(I;W^{1,2})$ norm.
The heat term is computed using any bounded extension of
\(q(\cdot,t-s)\) that agrees with \(q\) near \(x\); two such choices change
the left-hand side by \(o(s)\).
\end{lemma}

\begin{proof}
By \cref{prop:RCD-cutoffs}, choose a compactly supported \RCD{} test function
$\chi$, equal to one on
$B_{2r}(x)$ and supported in $B_{3r}(x)$, with $|\mathrm D\chi|\le C/r$ and
$\boldsymbol\Delta\chi=g_\chi\m$, $|g_\chi|\le C/r^2$.  For the time
integration at the stated regularity,
fix a compact subinterval $I'\Subset I$ containing $[t-s,t]$.  For
$0<\delta<\dist_{\mathbb R}(I',\partial I)$, use the forward temporal Steklov averages
\[
 q_\delta(x,r):=\frac1\delta\int_0^\delta q(x,r+\sigma)\,\dd\sigma,
 \qquad
 a_\delta(x,r):=\frac1\delta\int_0^\delta a(x,r+\sigma)\,\dd\sigma,
 \qquad r\in I'.
\]
The same subscript denotes this forward average when applied to any other
locally integrable space--time function below.
Testing the averaged weak inequality against a
nonnegative backward heat solution with the spatial cutoff $\chi$, and then
integrating from $t-s$ to $t$, gives
\[
 \Pheat_s[\chi q_\delta(\cdot,t-s)](x)
 -q_\delta(x,t)
 \le\int_0^s\Pheat_\tau[\chi a_\delta(\cdot,t-\tau)](x)\,\dd\tau
 +\mathcal R_{\delta,s}.
\]
Here
\[
 \mathcal R_{\delta,s}
 :=\int_0^s\Pheat_\tau\!\left[
 \left(qg_\chi+2\ip{\nabla\chi}{\nabla q}\right)_\delta
 (\cdot,t-\tau)\right](x)\,\dd\tau
\]
is the contribution of the spatial Leibniz commutator.
The inequality is first obtained with bounded heat-kernel approximations as
test functions.  These approximations belong to the local energy class and
are therefore admissible in the weak formulation; monotone approximation
then recovers the kernel.  Letting
$\delta\downarrow0$ is justified by the continuity of $q$, the local
$L^1$ assumption on $a$, and the local $L^2(I;W^{1,2})$ bound.  Passing to
the limit gives the preceding localized variation-of-constants inequality.
Its remainder is the Leibniz commutator
\(qg_\chi+2\ip{\nabla\chi}{\nabla q}\), supported in the annulus
$A=B_{3r}(x)\setminus B_{2r}(x)$.  The zeroth-order
term is $o(s)$ by \cref{lem:RCD-tail}; compare
\cite[Section~2.6]{HanZhu2026}.  For the gradient term,
Cauchy--Schwarz and the annular Gaussian bound give, for some $\nu>0$,
\begin{align*}
&\int_0^s\!\int_A p_\tau(x,z)|\mathrm D\chi|\,|\mathrm D q(z,t-\tau)|\,\dd\m(z)\,\dd\tau\\
&\quad\le C\left(\int_0^s \tau^{-\nu}e^{-cr^2/\tau}\,\dd\tau\right)^{1/2}
 \left(\int_0^s\!\int_A |\mathrm D q|^2\,\dd\m\,\dd\tau\right)^{1/2}=o(s).
\end{align*}
The exponential factor absorbs every power of $\tau^{-1}$, while the
second factor is finite by the local energy assumption.  Removing the cutoff
therefore gives the asserted Duhamel inequality.
\end{proof}

For $P\in Y$ set
\begin{equation}\label{eq:wP-definition}
\begin{split}
 w_{P,s}(x_0,y,t_0):={}&-\dist_Y^2(P,u(y,t_0-s))
 +\dist_Y^2(P,u(x_0,t_0))\\
 &+\dist_Y^2(u(y,t_0-s),u(x_0,t_0)).
\end{split}
\end{equation}

\begin{proposition}[Simultaneous heat-kernel inequality]
\label{prop:RCD-paired-mean-value}
On \(Q_0:=U_0\times I\), let \(G_{\mathrm{hk}}\) be the intersection of the following
Borel full-measure sets:
\begin{enumerate}[label=(\roman*),leftmargin=2em]
\item Lebesgue points of \(e\) at which the localized differentiation formula
in \cref{prop:parabolic-differentiation} holds;
\item points at which the identity in \cref{prop:time-shift-energy} holds.
\end{enumerate}
Then \(G_{\mathrm{hk}}\) is independent of the target point, and for every
\((x_0,t_0)\in G_{\mathrm{hk}}\) and every \(P\) in the closed convex hull of the bounded
image,
\begin{equation}\label{eq:RCD-wP-mean-value}
 \limsup_{s\downarrow0}\frac{1}{s}
 \int_Xp_s(x_0,y)w_{P,s}(x_0,y,t_0)\,\dd\m(y)\le0.
\end{equation}
Consequently,
\begin{equation}\label{eq:E-paired-good-set}
 E_{\mathrm{pair}}:=\left\{(x,y,t):(x,t)\in G_{\mathrm{hk}},\ (y,t)\in G_{\mathrm{hk}}\right\}
\end{equation}
has full $\m\otimes\m\otimes\mathcal L^1$ measure in
\(U_0\times U_0\times I\).
\end{proposition}

\begin{proof}
Fix \((x_0,t_0)\in G_{\mathrm{hk}}\) and an arbitrary \(P\) in the closed convex hull of the
bounded image.  Apply the square inequality in \cref{cor:fixed-target},
namely \eqref{eq:RCD-square-P}, to
\(q_P(x,t):=-\dist_Y^2(P,u(x,t))+\dist_Y^2(P,u(x_0,t_0))\).
The localized Duhamel inequality
\eqref{eq:Duhamel-conclusion} of \cref{lem:Duhamel-boundary-tail} gives
\begin{equation}
 \limsup_{s\downarrow0}\frac{1}{s}\int_Xp_s(x_0,y)q_P(y,t_0-s)\,\dd\m(y)
 \le-2(n+2)e_{u^{t_0}}(x_0),
\end{equation}
because \((x_0,t_0)\) is a common semigroup differentiation point for
\(e\) given by \eqref{eq:parabolic-differentiation} of
\cref{prop:parabolic-differentiation}, without an
approximation \(P_j\to P\).  Combining this with the
identity in \cref{prop:time-shift-energy} proves
\eqref{eq:RCD-wP-mean-value}.

The differentiation set in item~(i) is independent of \(P\), because the
right-hand side of the Duhamel estimate is always \(-2(n+2)e\).  Its
cutoff-tail constants are uniform when \(P\) ranges over the closed convex
hull of the bounded image, since both the image and these target points lie
in a fixed bounded ball.  Item~(ii) is already target-independent.  Thus the
same \(G_{\mathrm{hk}}\) works simultaneously for all \(P\), and Fubini's theorem gives
the claimed full-measure product set.
\end{proof}

\section{Elliptic perturbations on \RCD{} spaces}\label{sec:contact-moving}

This section constructs the perturbations needed to move a contact point into
a prescribed full-measure set.  We first establish quadratic barrier bounds
on product spaces, then combine the localized elliptic ABP estimate with a
time-slicing argument.

\subsection{Quadratic barriers and product Laplacians}

The measure-valued calculus and Laplacian comparison are used in the form
established by Gigli \cite{Gigli2015}; see also the systematic treatment of
weak Laplacian bounds by Mondino--Semola
\cite{MondinoSemolaWeakLaplacian2025}.  For bounded lower semicontinuous functions on
\(\RCD(K,N)\) spaces, \(K\in\mathbb R\), \(1<N<\infty\),
distributional, comparison, viscosity, and heat-flow upper bounds for the
Laplacian are equivalent by Gigli--Mondino--Semola
\cite[Theorem~1.1]{GigliMondinoSemola2024}, provided that the prescribed upper bound is
continuous.  All applications below have a constant upper bound, except for
the passage in \cref{prop:heat-test-weak-formulation}, where temporal
mollification produces a continuous right-hand side before the equivalence is
used.

For \(K\in\mathbb R\), \(1<N<\infty\), put
\(K_-:=\max\{-K,0\}\).  For \(R>0\), set
\begin{equation}
  \mathsf c_{K,N}(R):=
  \begin{cases}
    N, & K\ge0,\\[1mm]
    1+\sqrt{(N-1)K_-}\,R
    \coth\!\left(\sqrt{\dfrac{K_-}{N-1}}\,R\right),&K<0,
  \end{cases}
\end{equation}
The value at \(K=0\) is the continuous limit of the second line.

\begin{proposition}[Squared-distance Laplacian comparison]
\label{prop:squared-distance-comparison}
Let \((X,\dist,\m)\) be \(\RCD(K,N)\), \(1<N<\infty\), and put
\(\dist_p(x):=\dist(x,p)\).  If \(p\in X\) and \(U\subset B_R(p)\), then
\begin{equation}\label{eq:square-lap-comp}
  \boldsymbol\Delta\!\left(\frac12\dist_p^2\right)
  \le \mathsf c_{K,N}(R)\,\m
  \qquad\text{on }U
\end{equation}
in the measure-valued sense.  The same local upper bound holds in the equivalent
comparison, viscosity, and heat-flow senses of Gigli--Mondino--Semola
\cite[Theorem~1.1]{GigliMondinoSemola2024}.
\end{proposition}

\begin{proof}
The Laplacian comparison theorem in the measure-valued formulation gives the
distributional upper bound for \(\dist_p\), including its singularity at
\(p\); see \cite[Theorem~3.7]{MondinoSemola2026}.  Applying the chain rule
\cite[Proposition~4.11]{Gigli2015} to \(r\mapsto r^2/2\) gives
\(\boldsymbol\Delta(\dist_p^2/2)
 =\dist_p\boldsymbol\Delta\dist_p+|\mathrm D\dist_p|^2\m\), while
\(|\mathrm D\dist_p|=1\) \(\m\)-a.e.\ off \(p\).  Bounding
\(\dist_p\le R\) on \(U\) yields the claimed comparison.  The other
formulations follow from
Gigli--Mondino--Semola \cite[Theorem~1.1]{GigliMondinoSemola2024}.
\end{proof}

\medskip

\noindent\textbf{Tensorized \RCD{} calculus.}
Let \((X_i,\dist_i,\m_i)\) be \(\RCD(K_i,N_i)\), \(i=1,2\), with
\(1<N_i<\infty\), and put \(\Z:=X_1\times X_2\) and
\(\m_{\Z}:=\m_1\otimes\m_2\).  Equip \(\Z\) with the product distance
defined by
\begin{equation}
  \dist_{\Z}^2((x_1,x_2),(y_1,y_2))
  :=\dist_1^2(x_1,y_1)+\dist_2^2(x_2,y_2).
\end{equation}
We write \(\DeltaZ\) for the measure-valued Laplacian on \(\Z\), and
\(\boldsymbol\Delta_i\) for the
measure-valued Laplacian on the $i$th factor.  For functions \(\psi_i\) on
\(X_i\), set \((\psi_1\oplus\psi_2)(x_1,x_2):=\psi_1(x_1)+\psi_2(x_2)\).
Set \(K_{12}:=\min\{K_1,K_2\}\) and \(N_{12}:=N_1+N_2\).  Then
\begin{equation}
  (\Z,\dist_{\Z},\m_{\Z})
  \text{ is an }\RCD(K_{12},N_{12})\text{ space}.
\end{equation}
For separated functions, the measure-valued product Laplacian is
\begin{equation}\label{eq:product-lap}
  \DeltaZ(\psi_1\oplus\psi_2)
  =(\boldsymbol\Delta_1\psi_1)\otimes\m_2
   +\m_1\otimes(\boldsymbol\Delta_2\psi_2).
\end{equation}
Compactly supported tests are approximated in \(W^{1,2}\) by finite sums of
functions \(\varphi_1(x_1)\varphi_2(x_2)\), and the minimal weak upper gradient decomposes into its two
partial components.  Distributional identities may therefore be verified on
these product functions and extended by density and Fubini's theorem.
The \RCD{} tensorization and the displayed parameters follow from
\cite[Theorem~6.11]{AmbrosioGigliSavareDuke2014} and the
curvature--dimension tensorization in \cite{SturmActaII}.  The product
gradient decomposition and \eqref{eq:product-lap} follow from
\cite[Theorems~6.17--6.18]{AmbrosioGigliSavareDuke2014}.  For the analogous
slice calculus when one factor is an interval, see
\cite[Theorem~3.7]{GigliHan2018}, from the first author's joint work with
Gigli.

\begin{lemma}[Sliced upper Laplacian bounds on products]
\label{lem:sliced-product-lap}
Let \(U_i\Subset X_i\) be bounded open sets and let
\(G\in W^{1,2}_{\loc}(U_1\times U_2)\cap L^\infty_{\loc}(U_1\times U_2)\).
Assume that for a.e. \(x_2\in U_2\), the slice \(G(\cdot,x_2)\)
has a measure-valued Laplacian on \(U_1\) with
\(\boldsymbol\Delta_{1}G(\cdot,x_2)\le C_1\m_1\),
and that for a.e. \(x_1\in U_1\), the slice \(G(x_1,\cdot)\) has a
measure-valued Laplacian on \(U_2\) with
\(\boldsymbol\Delta_{2}G(x_1,\cdot)\le C_2\m_2\).
Then \(G\) has the product distributional upper bound
\begin{equation}\label{eq:sliced-product-lap-bound}
 \DeltaZ G
 \le (C_1+C_2)\m_{\Z}
 \qquad\text{on }U_1\times U_2.
\end{equation}
\end{lemma}

\begin{proof}
Let \(0\le\varphi\in\Lip_c(U_1\times U_2)\).  For a.e. \(x_2\), the section
\(x_1\mapsto\varphi(x_1,x_2)\) is a nonnegative compactly supported
Lipschitz test in \(U_1\), and for a.e. \(x_1\) the analogous statement holds
for \(x_2\mapsto\varphi(x_1,x_2)\) in \(U_2\).  Applying the two sliced inequalities and integrating in the other
variable gives
\begin{align*}
 -\int_{U_2}\int_{U_1}\ip{\nabla_{1}G}{\nabla_{1}\varphi}\,\dd\m_1\,\dd\m_2
 &\le C_1\int_{U_1\times U_2}\varphi\,\dd(\m_1\otimes\m_2),\\
 -\int_{U_1}\int_{U_2}\ip{\nabla_{2}G}{\nabla_{2}\varphi}\,\dd\m_2\,\dd\m_1
 &\le C_2\int_{U_1\times U_2}\varphi\,\dd(\m_1\otimes\m_2).
\end{align*}
By Ambrosio--Gigli--Savar\'e
\cite[Theorem~6.17]{AmbrosioGigliSavareDuke2014}, the product Cheeger energy
satisfies the tensorization identity for the
minimal weak upper gradient, so the sum of the two left-hand sides is
\(-\int\ip{\nabla G}{\nabla\varphi}\,\dd(\m_1\otimes\m_2)\).  The result is
\eqref{eq:sliced-product-lap-bound} under
\cref{def:measure-valued-laplacian}.
\end{proof}

\medskip

In a squared-distance barrier \(\alpha\dist_i^2(\cdot,p)/2\), we call \(p\) its
\emph{vertex} and \(\alpha\) its \emph{opening}.

\begin{lemma}[Separated quadratic barriers]\label{lem:product-barrier}
Let \(U_i\Subset X_i\) be bounded open sets and let \(D_i\Subset X_i\) be
compact vertex sets.  Define the local radii and constants
\begin{equation}
  R_i:=\sup\left\{\dist_i(x_i,p_i):x_i\in U_i,\ p_i\in D_i\right\},
  \qquad
  C_i:=\mathsf c_{K_i,N_i}(R_i),
  \qquad i=1,2.
\end{equation}
If
\begin{equation}
  \psi_i(x_i)=\sum_{j=1}^M\alpha_j\frac12\dist_i^2(x_i,p_{ij}),
  \qquad p_{ij}\in D_i,\quad \alpha_j\ge0,
\end{equation}
then
\begin{equation}
  \boldsymbol\Delta_i\psi_i
  \le\left(\sum_{j=1}^M\alpha_j\right)C_i\m_i
  \qquad\hbox{on }U_i.
\end{equation}
Consequently,
\begin{equation}
  \DeltaZ(\psi_1\oplus\psi_2)
  \le\left(\sum_{j=1}^M\alpha_j\right)(C_1+C_2)\m_{\Z}
  \qquad\hbox{on }U_1\times U_2.
\end{equation}
\end{lemma}

\begin{proof}
Apply \cref{prop:squared-distance-comparison} to each vertex \(p_{ij}\), add
the resulting measure inequalities with nonnegative coefficients, and then
apply the product-Laplacian identity of
\cite[Theorem~6.18]{AmbrosioGigliSavareDuke2014}.
\end{proof}

\begin{lemma}[Diagonal squared-distance bound on a product]
\label{lem:diagonal-square}
Let \((X,\dist,\m)\) be \(\RCD(K,N)\), and let \(U,V\Subset X\) be
bounded and open.  Put \(\Z:=X\times X\) and
\(\m_{\Z}:=\m\otimes\m\).  Set
\(R_{UV}:=\sup\{\dist(x,y):x\in U,\ y\in V\}\).
Define \(C_{\mathrm{diag}}(K,N,R_{UV}):=2\mathsf c_{K,N}(R_{UV})\).
For \(G(x,y):=\dist^2(x,y)/2\) on \(U\times V\),
\begin{equation}\label{eq:diagonal-square-measure}
  \DeltaZ G
  \le C_{\mathrm{diag}}(K,N,R_{UV})\,\m_{\Z}.
\end{equation}
\end{lemma}

\begin{proof}
For fixed \(y\in V\), \cref{prop:squared-distance-comparison} gives
\(\boldsymbol\Delta_xG(\cdot,y)\le\mathsf c_{K,N}(R_{UV})\m\) on \(U\).
For fixed \(x\in U\), the same comparison in the second variable gives
\(\boldsymbol\Delta_yG(x,\cdot)\le\mathsf c_{K,N}(R_{UV})\m\) on \(V\).
Applying \cref{lem:sliced-product-lap} with these two bounds yields
\eqref{eq:diagonal-square-measure}.
\end{proof}

\begin{lemma}[Power-distance bound on a product]
\label{lem:diagonal-power}
Under the assumptions of \cref{lem:diagonal-square}, let \(p\ge2\) be an
integer.  Then
\begin{equation}\label{eq:diagonal-power-measure}
 \DeltaZ\!\left(\frac{\dist^p(x,y)}p\right)
 \le C_p(K,N,R_{UV})\,\m_{\Z}
 \quad\text{on }U\times V,
\end{equation}
where one may take
\[
 C_p(K,N,R_{UV})
 :=\bigl[C_{\mathrm{diag}}(K,N,R_{UV})+2(p-2)\bigr]R_{UV}^{p-2}.
\]
\end{lemma}

\begin{proof}
Set \(G:=\dist^2/2\).  The product upper-gradient calculus gives
\(|\mathrm DG|^2\le2\dist^2=4G\) almost everywhere.  Apply the
measure-valued chain rule \cite[Proposition~4.11]{Gigli2015} to
\(\Phi(G):=(2G)^{p/2}/p\), using
\(\Phi'(G)=\dist^{p-2}\) and \(\Phi''(G)=(p-2)\dist^{p-4}\), and combine it
with \eqref{eq:diagonal-square-measure} from
\cref{lem:diagonal-square}.  Since \(\dist\le R_{UV}\) on \(U\times V\),
this gives \eqref{eq:diagonal-power-measure}.  When \(p=3\), use instead
\[
 \Phi_\delta(s)
 :=\frac{(2s+\delta)^{3/2}-\delta^{3/2}}3-\sqrt\delta\,s.
\]
Then \(0\le\Phi_\delta'(s)\le\sqrt{2s}\) and
\(\Phi_\delta''(s)=(2s+\delta)^{-1/2}\).  The preceding calculation gives
the stated bound uniformly in \(\delta\).  Moreover,
\(\Phi_\delta(G)\to\dist^3/3\) locally in \(L^2\), while
\(0\le\Phi_\delta'(G)\le\dist\) and
\(\Phi_\delta'(G)\to\dist\) almost everywhere.  The local boundedness of
\(\dist\), the Sobolev chain rule, and dominated convergence therefore give
strong convergence in \(W^{1,2}_{\loc}(U\times V)\).  Testing the uniform
distributional inequality and passing to the limit proves
\eqref{eq:diagonal-power-measure} for \(p=3\).
\end{proof}

For a continuous function with at most quadratic growth, define
\begin{equation}\label{eq:heat-flow-upper-laplacian}
  \HeatDelta f(x)
  :=\limsup_{s\downarrow0}
  \frac{\Pheat_s f(x)-f(x)}s.
\end{equation}
By \cref{lem:RCD-tail}, the limsup in
\eqref{eq:heat-flow-upper-laplacian} is unchanged if \(f\) is replaced by
another function of at most quadratic growth that agrees with it near \(x\).

\subsection{Sliced elliptic contact selection}
\label{subsec:sliced-contact-selection}

For a bounded open set $U\Subset X$, a vertex set $D\subset X$, an
opening $a>0$, and $w\in C(\overline{U})$, write
\[
 \mathcal A_a(D,U,w):=
 \left\{x\in \overline{U}:\
 \begin{aligned}
 &\text{there is }y\in D\text{ such that}\\[-1mm]
 &w(x)+\frac a2\dist^2(x,y)
   =\min_{z\in\overline{U}}\left(w(z)+\frac a2\dist^2(z,y)\right)
 \end{aligned}
 \right\}.
\]
The inclusion
$\mathcal A_a(D_1,U,w)\subset\mathcal A_a(D_2,U,w)$ holds whenever
$D_1\subset D_2$.

We use the localized ABP contact estimate
\cite[Theorem~4.3, equations~(4.26)--(4.36)]{MondinoSemola2026} in the
following form.  For the sharp
ABP estimate on \RCD{} spaces obtained by an optimal-transport argument, see
the first author's work \cite[Theorem~1.4]{Han2026}.

\begin{proposition}[Localized contact estimate, ABP estimate]
\label{prop:MS-localized-contact}
Let $U\Subset X$ be bounded and open, and let
$w\in W^{1,2}_{\loc}(U)\cap C(\overline{U})$ satisfy
$\boldsymbol\Delta w\le L\m$ for some $L>0$.  Let $D\subset X$ be compact
and fix $a>0$.  Assume that $w$ attains a strict minimum in $U$, that is,
\(\min_{\overline{U}}w<\min_{\partial U}w\), and that
$\mathcal A_a(D,U,w)\subset U$.  Then
\begin{equation}
 \m(D)\le C_{\mathrm E}\,\m\bigl(\mathcal A_a(D,U,w)\bigr),
\end{equation}
where, with \(K_-:=\max\{-K,0\}\), one may take
\begin{equation}\label{eq:localized-contact-constant}
 C_{\mathrm E}
 :=\exp\!\left(\frac{K_-\operatorname{osc}_{\overline{U}}w+L}{a}\right)
 <\infty.
\end{equation}
In particular, this constant is independent of the compact vertex set
\(D\).
\end{proposition}

\medskip

The parabolic contact statement follows by applying this elliptic estimate on
individual time slices.  Fix an $\RCD(K,N)$ space $(X,\dist,\m)$, a bounded
open set $U\Subset X$ with $\m(\partial U)=0$, and an interval
$I=[t_-,t_+]$.  Let $u\in C(\overline{U}\times I)$ and suppose that, for some
constant $L_{\mathrm{time}}>0$,
\begin{equation}\label{eq:uniform-time-lipschitz-contact}
 |u(x,t)-u(x,s)|\le L_{\mathrm{time}}|t-s|
 \quad\text{for all }x\in\overline{U}\text{ and }s,t\in I,
\end{equation}
and, for almost every $t\in I$,
\begin{equation}\label{eq:slice-lap-assumption}
 u^t\in W^{1,2}_{\loc}(U),\qquad
 \boldsymbol\Delta u^t\le L\m
 \quad\text{on }U.
\end{equation}
Here $L>0$ is independent of $t$.
Assume that $u$ attains its minimum at
$(x_*,t_*)\in U\times(t_-,t_+]$.  Put $u_*:=u(x_*,t_*)$ and assume that
\begin{equation}\label{eq:sliced-boundary-gap}
 g_*:=\min\left\{
 \inf_{\partial U\times I}u,
 \inf_{\overline{U}\times\{t_-\}}u
 \right\}-u_*>0.
\end{equation}
Let $D\subset X$ be compact and let $a>0$ satisfy $\m(D)>0$ and
\begin{equation}\label{eq:sliced-vertex-smallness}
 \sup_{y\in D}\frac a2\dist^2(x_*,y)\le \frac{g_*}{8}.
\end{equation}
For $y\in D$ and $t\in[t_-,t_*]$, set
\[
 m_y(t):=\min_{z\in\overline{U}}
 \left(u(z,t)+\frac a2\dist^2(z,y)\right),
 \qquad
 M_y(t):=\min_{s\in[t_-,t]}m_y(s),
\]
and define
\[
 D_t:=\left\{y\in D:M_y(t)=m_y(t)\le u_*+\frac{g_*}{2}\right\}.
\]
Thus $D_t$ consists of the vertices whose running minimum is attained on the
current slice at a level not exceeding $u_*+g_*/2$.
Define
\[
 \Gamma^-:=\{(x,t)\in\overline{U}\times[t_-,t_*]:
 x\in\mathcal A_a(D_t,U,u^t)\}.
\]
Whenever $D_t\ne\varnothing$, the level restriction gives
$\mathcal A_a(D_t,U,u^t)\Subset U$ and
$\min_{\overline{U}}u^t<\min_{\partial U}u^t$.  Thus the two localization
hypotheses of \cref{prop:MS-localized-contact} hold on every active slice.

\begin{proposition}[Sliced elliptic contact selection]
\label{prop:sliced-contact-selection}
Under the preceding notation and assumptions, $\Gamma^-$ is Borel and has
positive $\m\otimes\mathcal L^1$-measure.  More precisely,
\begin{equation}
 (\m\otimes\mathcal L^1)(\Gamma^-)
 \ge \frac{3g_*}{8L_{\mathrm{time}}C_{\mathrm E}}\,\m(D)>0,
\end{equation}
where $C_{\mathrm E}$ is the finite constant in
\cref{prop:MS-localized-contact}.
Consequently, $\Gamma^-$ intersects every prescribed full-measure Borel subset
of $U\times(t_-,t_*)$.
\end{proposition}

\begin{proof}
The function $(y,t)\mapsto m_y(t)$ is continuous: by the properness in
\cref{sec:prelim}, \(\overline{U}\) is compact, and \(m_y(t)\) is the minimum
there of a jointly continuous function.  Uniform
continuity on $D\times[t_-,t_*]$ then shows that $(y,t)\mapsto M_y(t)$ is also
continuous.  In particular, every $D_t$ is compact and
\[
 \left\{(y,t)\in D\times[t_-,t_*]:y\in D_t\right\}
\]
is closed, hence Borel.

The contact set is measurable because
\[
 \left\{(x,t,y)\in\overline{U}\times[t_-,t_*]\times D:
 u(x,t)+\frac a2\dist^2(x,y)=M_y(t)
 \le u_*+\frac{g_*}{2}\right\}
\]
is compact and its projection onto $\overline{U}\times[t_-,t_*]$ is exactly
$\Gamma^-$, so $\Gamma^-$ is compact and, in particular, Borel.

By \cref{subsec:sliced-contact-selection}, specifically
\eqref{eq:uniform-time-lipschitz-contact}, the function $m_y$ is
$L_{\mathrm{time}}$-Lipschitz for each fixed $y$, and its running minimum $M_y$ is
nonincreasing and $L_{\mathrm{time}}$-Lipschitz.  Indeed, if \(t_1<t_2\),
any new minimum on $[t_1,t_2]$ is at least
$M_y(t_1)-L_{\mathrm{time}}(t_2-t_1)$.
Define the truncated running minimum
\(M_y^{\mathrm{tr}}(t):=\min\{M_y(t),u_*+g_*/2\}\).
It is absolutely continuous in $t$.  If $M_y(t)<u_*+\frac{g_*}{2}$ but $m_y(t)>M_y(t)$, then
the strict gap and continuity imply that $M_y$ is locally constant at $t$.
Thus, at almost every $t$ where $M_y^{\mathrm{tr}}$ is differentiable, a negative
derivative can occur only when $M_y(t)=m_y(t)\le u_*+\frac{g_*}{2}$, namely when
$y\in D_t$.  Since the Lipschitz constant is $L_{\mathrm{time}}$,
\begin{equation}\label{eq:truncated-running-minimum}
 -(M_y^{\mathrm{tr}})'(t)
 \le L_{\mathrm{time}}\mathbf 1_{D_t}(y)
 \qquad\text{for a.e. }t\in[t_-,t_*].
\end{equation}

The boundary gap \eqref{eq:sliced-boundary-gap} from
\cref{subsec:sliced-contact-selection} gives
$M_y(t_-)=m_y(t_-)\ge u_*+g_*$, and therefore
$M_y^{\mathrm{tr}}(t_-)=u_*+\frac{g_*}{2}$.  Testing the running minimum at $(x_*,t_*)$ and
using \eqref{eq:sliced-vertex-smallness} from
\cref{subsec:sliced-contact-selection} gives
$M_y(t_*)\le u_*+\frac{g_*}{8}$, hence
$M_y^{\mathrm{tr}}(t_*)\le u_*+\frac{g_*}{8}$.  Integrating
\eqref{eq:truncated-running-minimum} yields
\begin{equation}\label{eq:running-minimum-measure}
 \int_{t_-}^{t_*}\mathbf 1_{D_t}(y)\,\dd t
 \ge \frac{3g_*}{8L_{\mathrm{time}}}.
\end{equation}
Fubini's theorem and \eqref{eq:running-minimum-measure} give
\begin{equation}\label{eq:active-vertices-integrated}
 \int_{t_-}^{t_*}\m(D_t)\,\dd t
 \ge \frac{3g_*}{8L_{\mathrm{time}}}\m(D).
\end{equation}

Fix a time $t$ for which the slice assumption
\eqref{eq:slice-lap-assumption} in
\cref{subsec:sliced-contact-selection} holds.  If $D_t$ is
empty, the required slice estimate is trivial.  Otherwise, for every
$y\in D_t$,
\[
 m_y(t)\le u_*+\frac{g_*}{2},
 \qquad
 \min_{z\in\partial U}
 \left(u^t(z)+\frac a2\dist^2(z,y)\right)\ge u_*+g_*.
\]
Hence every contact associated with a vertex in $D_t$ lies in $U$.  Moreover,
the nonnegativity of the quadratic term gives
\[
 \min_{\overline{U}}u^t\le m_y(t)\le u_*+\frac{g_*}{2}
 <\min_{\partial U}u^t,
\]
so $u^t$ itself attains a strict minimum in $U$.  Both localization
hypotheses of \cref{prop:MS-localized-contact} are therefore satisfied.
Furthermore,
\(\operatorname{osc}_{\overline{U}}u^t\le
\operatorname{osc}_{\overline{U}\times I}u\).
Thus \eqref{eq:localized-contact-constant} gives the same finite constant
$C_{\mathrm E}$ for every such $t$, and
\(\m(D_t)\le C_{\mathrm E}\m(\mathcal A_a(D_t,U,u^t))\).
By definition,
\[
 \Gamma_t^-:=\{x\in\overline{U}:(x,t)\in\Gamma^-\}
 =\mathcal A_a(D_t,U,u^t).
\]
Integrating over the full-measure set of times in
\eqref{eq:slice-lap-assumption} from
\cref{subsec:sliced-contact-selection}, and using
\eqref{eq:active-vertices-integrated} and Fubini's theorem, we obtain
\[
 \frac{3g_*}{8L_{\mathrm{time}}}\m(D)
 \le\int_{t_-}^{t_*}\m(D_t)\,\dd t
 \le C_{\mathrm E}\int_{t_-}^{t_*}\m(\Gamma_t^-)\,\dd t
 =C_{\mathrm E}(\m\otimes\mathcal L^1)(\Gamma^-).
\]
This proves the claimed lower bound.  Finally,
$\m(\partial U)=0$ and the two time endpoints have zero
$\mathcal L^1$-measure, so $\Gamma^-\cap(U\times(t_-,t_*))$ still has positive
measure.  Its intersection with any prescribed full-measure Borel subset of
$U\times(t_-,t_*)$ is therefore nonempty.
\end{proof}

\medskip

We next derive the product-space form of
\cref{prop:sliced-contact-selection} used at a Hamilton--Jacobi contact point.

Fix $\RCD(K_i,N_i)$ spaces $(X_i,\dist_i,\m_i)$, open sets
$U_i\subset X_i$, $i=1,2$, and an open interval $J\subset\mathbb R$.  Write
\(\Z:=X_1\times X_2\) and \(\m_{\Z}:=\m_1\otimes\m_2\), and let
\(\DeltaZ\) be the product measure-valued Laplacian; see
\cite[Theorems~6.17--6.18]{AmbrosioGigliSavareDuke2014}.  Suppose that
$h\in C(U_1\times U_2\times J)$ has a local minimum at
$p^\ast=(x_1^\ast,x_2^\ast,t^\ast)$.  Assume that on every compact product
cylinder the maps $t\mapsto h(x_1,x_2,t)$ have Lipschitz constants uniform
in $(x_1,x_2)$ and that, for almost every $t$,
\begin{equation}
 \DeltaZ h^t\le L\,\m_{\Z}
\end{equation}
for some $L>0$ independent of $t$, where
$h^t(\cdot,\cdot):=h(\cdot,\cdot,t)$.  Let
$E\subset U_1\times U_2\times J$ be a Borel set of full
$\m_{\Z}\otimes\mathcal L^1$-measure.

For $\kappa,a>0$ and vertices $\xi_i\in X_i$, put
\[
 \eta_0(t):=\frac\kappa2(t-t^\ast)^2,
 \qquad
 \eta_i(x_i):=\frac\kappa2\dist_i^2(x_i,x_i^\ast)
 +\frac a2\dist_i^2(x_i,\xi_i),\quad i=1,2.
\]
The radius $r_0$ below is a topological localization parameter, not a
parabolic radius: the time interval has length $r_0$, so the spatial and
temporal quadratic boundary gaps are both of order $\kappa r_0^2$.

\begin{proposition}[Normalized sliced \RCD{} perturbation]
\label{prop:normalized-rcd-perturbation}
For every $\delta>0$ there exists a sufficiently small localization radius
$r_0>0$, numbers $\kappa>0$ and $0<a\le\kappa$, vertices
$\xi_i\in B_{r_0/32}(x_i^\ast)$, and a point
\[
 p_0=(x_{1,0},x_{2,0},t_0)\in E,
 \qquad t_0\le t^\ast,
\]
such that $h+\eta_0+\eta_1+\eta_2$ attains a minimum at $p_0$ on the
backward cylinder
\[
 \overline{B}_{r_0}(x_1^\ast)\times
 \overline{B}_{r_0}(x_2^\ast)\times
 [t^\ast-r_0,t_0].
\]
The parameters can be chosen so that
\begin{equation}
 -\eta_0'(t_0)+\Delta_{\mathrm H,1}^{+}\eta_1(x_{1,0})
 +\Delta_{\mathrm H,2}^{+}\eta_2(x_{2,0})\le\delta.
\end{equation}
\end{proposition}

\begin{proof}
For almost every radius \(r_0>0\), both
\(\m_1(\partial B_{r_0}(x_1^\ast))\) and
\(\m_2(\partial B_{r_0}(x_2^\ast))\) vanish: spheres of distinct radii are
disjoint, and bounded balls have finite measure, so only countably many
radii can have a boundary of positive measure.  Choose such an \(r_0\), small
enough that the closed product cylinder is contained in the neighborhood on
which \(p^\ast\) is a minimum.  For
$i=1,2$, let $C_i$ be the constant in \cref{lem:product-barrier} for the
spatial set $B_{r_0}(x_i^\ast)$ and the vertex set
$\{x_i^\ast\}\cup\overline{B}_{r_0/32}(x_i^\ast)$.  Choose $\kappa>0$ so
small that
\begin{equation}\label{eq:normalized-perturbation-kappa}
 \kappa r_0+2\kappa(C_1+C_2)\le\delta.
\end{equation}
On this product cylinder, set
\[
 \widetilde h:=h+\frac\kappa2\dist_1^2(\cdot,x_1^\ast)
      +\frac\kappa2\dist_2^2(\cdot,x_2^\ast)
      +\frac\kappa2(t-t^\ast)^2.
\]
Because \(h\ge h(p^\ast)\) on the cylinder, the added term is at least
\(\kappa r_0^2/2\) on both the lateral boundary and the lower time face.
Thus \(\widetilde h\) has the strict boundary gap required in
\eqref{eq:sliced-boundary-gap} from
\cref{subsec:sliced-contact-selection}.  Its almost-everywhere slice bound is
\[
 \DeltaZ[\widetilde h(\cdot,\cdot,t)]
 \le\bigl(L+\kappa(C_1+C_2)\bigr)\m_{\Z},
\]
by \cref{lem:product-barrier}.

Let
\(D:=\overline{B}_{r_0/32}(x_1^\ast)\times
\overline{B}_{r_0/32}(x_2^\ast)\).
Properness makes \(D\) compact, and full support gives
\(\m_{\Z}(D)>0\).  The choice of \(r_0\) also gives zero measure to the
boundary of the spatial product domain.  The assumed local time-Lipschitz
bound for \(h\), together with the temporal perturbation, verifies
\eqref{eq:uniform-time-lipschitz-contact} from
\cref{subsec:sliced-contact-selection}.
Choose the elliptic opening $0<a\le\kappa$ sufficiently small that the
vertex condition \eqref{eq:sliced-vertex-smallness} in
\cref{prop:sliced-contact-selection} holds for this $D$.  Apply that
proposition on the product space $\Z$.  The resulting positive-measure
backward contact set intersects $E$.  Writing the selected product vertex as
$(\xi_1,\xi_2)$ gives the perturbations defined above and the asserted
backward minimum.

The temporal perturbation satisfies
$|\eta_0'|\le\kappa r_0$.  The constants \(C_i\) are uniform over the
chosen vertex sets.  Hence the factorwise bounds in
\cref{lem:product-barrier} and the distributional-to-heat-flow implication of
Gigli--Mondino--Semola \cite[Theorem~1.1]{GigliMondinoSemola2024}, applied
with a constant right-hand side, give
\[
 \Delta_{\mathrm H,1}^{+}\eta_1+\Delta_{\mathrm H,2}^{+}\eta_2
 \le(\kappa+a)(C_1+C_2).
\]
Since $a\le\kappa$, \eqref{eq:normalized-perturbation-kappa} gives the
required tolerance.
\end{proof}

\section{Hamilton--Jacobi estimate and Lipschitz regularity}\label{sec:HJ-final}

We combine the fixed-target good sets from
\cref{sec:fixed-target-good-set} with the normalized perturbation from
\cref{prop:normalized-rcd-perturbation}.  The resulting supersolution estimate
for the Hamilton--Jacobi envelope yields the spatial Lipschitz bound.

\subsection{Localized Hamilton--Jacobi auxiliary function}

Throughout this section, write
\(\Z:=X\times X\) and \(\m_{\Z}:=\m\otimes\m\), and let \(\DeltaZ\) denote
the measure-valued Laplacian on \(\Z\).
Recall also the two-point distance
\(F(x,y,t)=\dist_Y(u(x,t),u(y,t))\), defined in
\eqref{eq:F-definition} of \cref{prop:F-two-point};
its slice-Laplacian and time-Lipschitz bounds are collected in
\cref{prop:F-two-point}.

On a smooth Riemannian source, Zhang--Zhu formulate the corresponding step
using viscosity supersolutions
\cite[Definition~7.2 and Lemma~7.6]{ZhangZhu2026} and then invoke the
classical viscosity--distributional equivalence of Ishii
\cite{Ishii1995}.  Since that classical equivalence does
not apply on an \RCD{} source, we use the heat-test formulation below.  The
passage from this formulation to the weak one, proved in
\cref{prop:heat-test-weak-formulation}, is an additional contribution needed
for the \RCD{} source setting.

\begin{definition}[Heat tests and heat residuals]
\label{def:heat-test}
Let \(U\subset X\) be open, let \(I\subset\mathbb R\) be an open interval,
and put $Q=U\times I$.  A heat test on $Q$ is a function $v\in C(Q)$ such that
$t\mapsto v^t$ is locally $C^1$ in the topology of locally uniform
convergence, $\partial_tv\in C(Q)$, and
\(v^t\in D_{\loc}(\boldsymbol\Delta,U)\), with
\(\boldsymbol\Delta v^t=g(\cdot,t)\m\) for a single function $g\in C(Q)$.
Its pointwise heat residual is
\(\mathcal H_v(x,t):=g(x,t)-\partial_tv(x,t)\).
\end{definition}

\begin{definition}[Parabolic touching and supersolutions]
\label{def:heat-test-supersolution}
Let \(Q=U\times I\), and let \(v\) be a heat test on \(Q\) in the sense of
\cref{def:heat-test}.
The function \(v\) \emph{touches \(f\) from below at}
\((x_0,t_0)\in Q\) if there is \(\rho>0\) such that
\[
 \begin{gathered}
 B_\rho(x_0)\times(t_0-\rho^2,t_0]\subset Q,\\
 v\le f\ \text{on }B_\rho(x_0)\times(t_0-\rho^2,t_0],
 \qquad v(x_0,t_0)=f(x_0,t_0).
 \end{gathered}
\]
A lower semicontinuous $f$ is a heat-test supersolution of
$(\boldsymbol\Delta-\partial_t)f\le0$ if every heat test touching $f$ from
below in this backward parabolic sense has $\mathcal H_v\le0$ at the contact
point.
\end{definition}

Fix $R>0$ with $B_{16R}(\bar x)\Subset\Omega$ and choose an open interval
\([t_*,T]\Subset J_{\mathrm{amb}}=(\tau_-,\tau_+)\Subset(t_*/2,T+1)\).
Applying \cref{thm:RCD-time-Lipschitz} to a finite covering of
\(\overline{B}_{12R}(\bar x)\times\overline{J_{\mathrm{amb}}}\) gives a common
pointwise time-Lipschitz constant \(c_L\) on that cylinder.  We use
\cref{prop:simultaneous-P} with
\(U=B_{8R}(\bar x)\), \(I=J_{\mathrm{amb}}\), and
\cref{prop:RCD-paired-mean-value} with
\(U_0=B_{8R}(\bar x)\), \(U=B_{12R}(\bar x)\), and
\(I=J_{\mathrm{amb}}\).
Denote the resulting full-measure sets by \(T_{\mathrm{dist}}\), \(G_{\mathrm{hk}}\), and
\(E_{\mathrm{pair}}\).  In particular, \(E_{\mathrm{pair}}\) has full measure in
\(B_{8R}(\bar x)\times B_{8R}(\bar x)\times J_{\mathrm{amb}}\).

Base points will lie in $B_{5R}(\bar x)$ and maximizing vertices in
$B_{6R}(\bar x)$.  Since the representative is continuous and $\m$ has full
support, the almost-everywhere image bound from
\cref{subsec:semigroup-HMHF}, specifically \eqref{eq:bounded-image-ball},
implies the corresponding pointwise bound: otherwise continuity would make
\(\dist_Y(u,P_0)>M\) hold on a nonempty open set, which has positive measure
by full support.  Thus \(\dist_Y(u(x,t),P_0)\le M\) on
\(B_{16R}(\bar x)\times J_{\mathrm{amb}}\).
For
\(\varepsilon>0\) define
\begin{equation}\label{eq:HJ-envelope-def}
 \mathcal V_\varepsilon(x,t):=
 \max_{y\in\overline{B}_{6R}(\bar x)}
 \left\{
 F(x,y,t)-\frac{e^{-2Kt}}{2\varepsilon}\dist^2(x,y)
 \right\}.
\end{equation}
We call any point \(y\) attaining the maximum in
\eqref{eq:HJ-envelope-def} a \emph{maximizing vertex} for \((x,t)\).
The lemmas below localize the maximizing vertices and establish the
regularity of this envelope.  Its heat-test supersolution property is proved
in \cref{thm:HJ-supersolution} and converted to the weak formulation in
\cref{prop:heat-test-weak-formulation}.

\medskip

\begin{lemma}[Localization of maximizing vertices]
\label{lem:maximizer-localization}
There are positive constants \(\varepsilon_0\) and \(C_{\mathrm{vert}}\).
The former depends only on \((K,R,J_{\mathrm{amb}},M)\), and the latter only
on \((K,J_{\mathrm{amb}},M)\).  For
\(0<\varepsilon<\varepsilon_0\) and
\((x,t)\in B_{5R}(\bar x)\times J_{\mathrm{amb}}\), every maximizing vertex in
\eqref{eq:HJ-envelope-def} lies in the interior of \(B_{6R}(\bar x)\).  Moreover, if
\(y\) is such a maximizing vertex, then
\begin{equation}\label{eq:maximizer-sqrt-eps}
  \dist(x,y)\le C_{\mathrm{vert}}\sqrt\varepsilon.
\end{equation}
\end{lemma}

\begin{proof}
Set \(c_J:=\min_{t\in\overline{J}_{\mathrm{amb}}}e^{-2Kt}>0\).
The competitor \(y=x\) in the envelope definition
\eqref{eq:HJ-envelope-def} from \cref{sec:HJ-final} gives
\(\mathcal V_\varepsilon(x,t)\ge0\).  Hence every maximizing vertex \(y\), with
\(s:=\dist(x,y)\), satisfies
\begin{equation}\label{eq:maximizer-basic}
 0\le F(x,y,t)-\frac{e^{-2Kt}}{2\varepsilon}s^2
 \le 2M-\frac{c_J}{2\varepsilon}s^2.
\end{equation}
It follows that \(s\le 2\sqrt{M/c_J}\,\sqrt\varepsilon\).
This proves the localization estimate with
\(C_{\mathrm{vert}}=2\sqrt{M/c_J}\).  If \(M>0\), choose
\(\varepsilon_0<c_JR^2/(4M)\); then \eqref{eq:maximizer-basic} rules out
\(s\ge R\), and in particular rules out \(y\in\partial B_{6R}(\bar x)\).
If \(M=0\), \eqref{eq:maximizer-basic} gives \(s=0\), so the same conclusions
hold for every \(\varepsilon>0\).
\end{proof}

\begin{lemma}[Local Sobolev and time regularity of $\mathcal V_\varepsilon$]
\label{lem:HJ-envelope-regularity}
For $0<\varepsilon<\varepsilon_0$, the function $\mathcal V_\varepsilon$ is locally
bounded and continuous.  Moreover, on every compact positive-time cylinder,
the maps $t\mapsto \mathcal V_\varepsilon(x,t)$ have Lipschitz constants uniform in
$x$, and
\begin{equation}\label{eq:HJ-envelope-energy-class}
 \mathcal V_\varepsilon\in
 L^2_{\loc}(J_{\mathrm{amb}};W^{1,2}_{\loc}(B_{5R}(\bar x))).
\end{equation}
More precisely, on every smaller cylinder there is $C_\varepsilon<\infty$
such that
\begin{align}\label{eq:HJ-envelope-point-difference}
 |\mathcal V_\varepsilon(x,t)-\mathcal V_\varepsilon(x',t')|
 \le C_\varepsilon\bigl(&\dist(x,x')+|t-t'|\notag\\
 &+\dist_Y(u(x,t),u(x',t'))\bigr).
\end{align}
\end{lemma}

\begin{proof}
The vertex set \(\overline{B}_{6R}(\bar x)\) is compact by the properness in
\cref{sec:prelim}, and the expression in the envelope
definition \eqref{eq:HJ-envelope-def} from \cref{sec:HJ-final} is continuous
in \((x,y,t)\).  Hence the maximum in that definition is attained, and the
compactness of the parameter set and uniform continuity of the integrand give
continuity of the maximum.
For the quantitative estimate, fix two points \((x,t)\), \((x',t')\), choose a
maximizing vertex \(y\) for \((x,t)\), and use the same \(y\) as a
competitor for \((x',t')\).  The difference of the squared-distance terms is
bounded on the localized ball by
\(C_\varepsilon(\dist(x,x')+|t-t'|)\), using the Lipschitz dependence of
\(e^{-2Kt}\) on \(J_{\mathrm{amb}}\).  The
difference of the two \(F\) terms is controlled by the triangle inequality and
\cref{thm:RCD-time-Lipschitz}, specifically
\eqref{eq:pointwise-time-Lip}.  Interchanging the two points gives
\eqref{eq:HJ-envelope-point-difference}.

For the Sobolev assertion, fix a compact spatial ball
\(U'\Subset B_{5R}(\bar x)\) and a time \(t\) for which the spatial slice
belongs to the required energy class, and write
\[
 \Phi^t(x,y)=\frac{e^{-2Kt}}{2\varepsilon}\dist^2(x,y)-F(x,y,t),
 \qquad y\in \overline{B}_{6R}(\bar x).
\]
The squared-distance part has spatial upper gradient bounded by
\(C_\varepsilon\) on the localized ball.  For the second term, with \(y\) fixed,
\(x\mapsto F(x,y,t)=\dist_Y(u(x,t),u(y,t))\) has upper gradient bounded by
\(\sqrt{(n+2)e_{u^t}}\) by the fixed-target post-composition estimate in
\cref{prop:KS-properties}.  Since
\(-\mathcal V_\varepsilon=\inf_y\Phi^t(\cdot,y)\), choose a countable dense
family \(\{y_j\}\) in the compact vertex set.  The finite minima of
\(\Phi^t(\cdot,y_j)\) have the same upper-gradient bound by the Sobolev
lattice property \cite[Proposition~4.8]{AmbrosioGigliSavare2014}; passing to
their decreasing limit by
\cite[Lemma~4.3(b) and Lemma~4.4]{AmbrosioGigliSavare2014} gives
\[
  |\mathrm D\mathcal V_\varepsilon(\cdot,t)|
  \le C_\varepsilon+C\sqrt{e_{u^t}}
\]
on smaller balls.  For every compact interval \(I'\Subset J_{\mathrm{amb}}\),
the energy monotonicity \eqref{eq:EVI-energy-monotonicity} from
\cref{subsec:semigroup-HMHF} and the normalization
\(\FlowE=(n+2)\MapE/2\) give
\[
 \int_{I'}\int_{U'}e_{u^t}\,\dd\m\,\dd t
 \le \frac{2|I'|}{n+2}\FlowE(u_0).
\]
Integrating the preceding gradient bound proves the Sobolev assertion.  This
is the \RCD{} analogue of \cite[Lemma~7.1(3)]{ZhangZhu2026}.
\end{proof}

\begin{lemma}[Heat-semigroup expansion for heat tests]
\label{lem:test-heat-expansion}
Let $v$ be a heat test on $Q=U\times I$.  If $(x,t)\in Q$ and
$\widetilde v$ is any bounded extension agreeing with $v$ on a neighborhood
of $x$ for all times near $t$, then
\begin{equation}\label{eq:test-heat-expansion}
 \lim_{s\downarrow0}
 \frac{\Pheat_s[\widetilde v(\cdot,t-s)](x)-v(x,t)}s
 =\mathcal H_v(x,t).
\end{equation}
The limit is independent of the chosen extension.
\end{lemma}

\begin{proof}
Decompose the difference quotient as
\[
 \frac{\Pheat_s[\widetilde v(\cdot,t-s)-\widetilde v(\cdot,t)](x)}s
 +\frac{\Pheat_s[\widetilde v(\cdot,t)](x)-v(x,t)}s.
\]
The first term converges to $-\partial_tv(x,t)$ by the locally uniform
$C^1$ dependence on time and \cref{lem:RCD-tail}.  The second converges to
$g(x,t)$ by the local heat-flow characterization of the measure-valued
Laplacian.  Its right-hand side is continuous by the definition of a heat
test; see Gigli--Mondino--Semola
\cite[Theorem~1.1]{GigliMondinoSemola2024}.  If two extensions are used, their
difference is bounded and vanishes near \(x\).  By \cref{lem:RCD-tail}, its
heat average at \(x\) is \(O(e^{-c/s})=o(s)\); hence the quotient in
\eqref{eq:test-heat-expansion} has the same limit for both extensions.
\end{proof}

For a heat test \(v\), set
\begin{equation}\label{eq:H-def}
 H(x,y,t):=
 \frac{e^{-2Kt}}{2\varepsilon}\dist^2(x,y)-F(x,y,t)-v(x,t).
\end{equation}
If $v$ touches $-\mathcal V_\varepsilon$ from below at
$(x^\ast,t^\ast)$ and $y^\ast$ is a maximizing vertex in
\eqref{eq:HJ-envelope-def} from \cref{sec:HJ-final}, then $H$ has a local minimum at
$(x^\ast,y^\ast,t^\ast)$.  Indeed, near the touching point,
\[
 H(x,y,t)\ge-\mathcal V_\varepsilon(x,t)-v(x,t)\ge0,
\]
and both inequalities are equalities at
\((x^\ast,y^\ast,t^\ast)\).

\begin{proposition}[Regularity and slice bound for the auxiliary function]
\label{prop:H-admissible}
Let \(v\) be a heat test and let \(H\) be defined by \eqref{eq:H-def}.  On
every compact product cylinder
\(U\times V\times I\Subset
B_{8R}(\bar x)\times B_{8R}(\bar x)\times J_{\mathrm{amb}}\),
the maps $t\mapsto H(x,y,t)$ have Lipschitz constants
uniform in $(x,y)$,
\(H^t\in W^{1,2}_{\loc}(U\times V)\) for a.e. \(t\), and
\begin{equation}\label{eq:HJ-aux-slice-bound}
 \DeltaZ H^t
 \le \left[
 \frac{e^{2|K|\sup I}}{\varepsilon}C_{\mathrm{diag}}(K,N,R_{UV})
 +2c_L+C_v\right]\m_{\Z},
\end{equation}
where \(R_{UV}\) and \(C_{\mathrm{diag}}(K,N,R_{UV})\) are as in
\cref{lem:diagonal-square}, while \(g\) is the continuous density in
\(\boldsymbol\Delta v^t=g(\cdot,t)\m\) from \cref{def:heat-test}, and
\(C_v:=\|\max\{-g,0\}\|_{L^\infty(U\times I)}\).
Consequently, \(H\) satisfies the hypotheses of
\cref{prop:normalized-rcd-perturbation} on every smaller cylinder.
\end{proposition}

\begin{proof}
The time regularity follows from \cref{prop:F-two-point}, specifically
\eqref{eq:F-time-Lip}, the smooth
coefficient \(e^{-2Kt}\), and the time regularity of \(v\).  The spatial
Sobolev assertion follows from the definition of a heat test and the product
upper-gradient bounds for \(F\) in \cref{prop:F-two-point}.

By \eqref{eq:diagonal-square-measure} of \cref{lem:diagonal-square},
\[
 \DeltaZ\!\left[\frac{e^{-2Kt}}{2\varepsilon}\dist^2(x,y)\right]
 \le \frac{e^{2|K|\sup I}}{\varepsilon}
 C_{\mathrm{diag}}(K,N,R_{UV})\,\m_{\Z}.
\]
Moreover, \eqref{eq:F-slice-lower} of \cref{prop:F-two-point} gives
\(\DeltaZ(-F^t)\le2c_L\m_{\Z}\).  Since \(v\)
depends only on the first variable,
\(\DeltaZ(-v^t)=(-g(\cdot,t)\m)\otimes\m\le C_v\m_{\Z}\).
By the separated product identity \eqref{eq:product-lap} preceding
\cref{lem:sliced-product-lap}, adding the three inequalities proves
\eqref{eq:HJ-aux-slice-bound}.
\end{proof}

\subsection{Supersolution property}

\begin{theorem}[Hamilton--Jacobi supersolution property]
\label{thm:HJ-supersolution}
For every $0<\varepsilon<\varepsilon_0$, the function
$-\mathcal V_\varepsilon$, with $\mathcal V_\varepsilon$ defined in
\eqref{eq:HJ-envelope-def} of \cref{sec:HJ-final}, is a local heat-test
supersolution on
\(B_{5R}(\bar x)\times J_{\mathrm{amb}}\):
\begin{equation}\label{eq:HJ-negative-envelope-supersolution}
 (\boldsymbol\Delta-\partial_t)(-\mathcal V_\varepsilon)\le0
\end{equation}
in the sense of \cref{def:heat-test-supersolution}.
\end{theorem}

We use the following four estimates from the proof of
\cite[Lemma~7.6]{ZhangZhu2026} in the form needed at the perturbed contact
point.  Fix
$x_0,y_0\in X$ and $t_0>0$, and let $\Pi_s$ be an optimal coupling of
$\mu_s^{x_0}:=\Hheat_s\delta_{x_0}$ and
$\mu_s^{y_0}:=\Hheat_s\delta_{y_0}$.  Let $v$ be a heat test near
$(x_0,t_0)$, and let $\eta_0,\eta_1,\eta_2$ be quadratic perturbations of the
form used in \cref{prop:normalized-rcd-perturbation}.  Using bounded extensions of
$v,\eta_1,\eta_2$ which agree with them near $x_0$ and $y_0$, define
\begin{align}
 I_1(s)&:=\frac1{2\varepsilon}\int
 [e^{-2K(t_0-s)}\dist^2(x,y)
 -e^{-2Kt_0}\dist^2(x_0,y_0)]\,\dd\Pi_s,
 \label{eq:I1}\\
 I_2(s)&:=-\int[F(x,y,t_0-s)-F(x_0,y_0,t_0)]\,\dd\Pi_s,
 \label{eq:I2}\\
 I_3(s)&:=-\int[v(x,t_0-s)-v(x_0,t_0)]\,\dd\mu_s^{x_0}(x),
 \label{eq:I3}\\
 I_4(s)&:=\eta_0(t_0-s)-\eta_0(t_0)
 +\int[\eta_1(x)-\eta_1(x_0)]\,\dd\mu_s^{x_0}(x)\notag\\
 &\qquad
 +\int[\eta_2(y)-\eta_2(y_0)]\,\dd\mu_s^{y_0}(y).
 \label{eq:I4}
\end{align}

\begin{lemma}[Estimates at a perturbed contact point]
\label{lem:four-term-estimates}
Under the preceding notation,
\begin{align}
 \limsup_{s\downarrow0}\frac{I_1(s)}s&\le0,
 \label{eq:I1-limsup}\\
 \limsup_{s\downarrow0}\frac{I_2(s)}s&\le0
 \quad\text{if }(x_0,t_0),(y_0,t_0)\in G_{\mathrm{hk}},
 \label{eq:I2-limsup}\\
 \lim_{s\downarrow0}\frac{I_3(s)}s
 &=-\mathcal H_v(x_0,t_0),
 \label{eq:I3-limit}\\
 \limsup_{s\downarrow0}\frac{I_4(s)}s
 &\le-\eta_0'(t_0)+\Delta_{\mathrm H,1}^{+}\eta_1(x_0)
 +\Delta_{\mathrm H,2}^{+}\eta_2(y_0).
 \label{eq:I4-limsup}
\end{align}
\end{lemma}

\begin{proof}
The contraction \eqref{eq:Wp-contraction} from \cref{sec:prelim} gives
\[
 \int\dist^2(x,y)\,\dd\Pi_s
 \le e^{-2Ks}\dist^2(x_0,y_0).
\]
Since
$e^{-2K(t_0-s)}e^{-2Ks}=e^{-2Kt_0}$, this proves the first assertion.

For \(I_2\), put
$F_0:=F(x_0,y_0,t_0)$.  If $F_0=0$, then
$I_2(s)\le0$.  If $F_0>0$, let $Q_m$ be the midpoint of
$u(x_0,t_0)$ and $u(y_0,t_0)$.  For $P,S,Q,R\in Y$, with $Q_m$ the
midpoint of $Q$ and $R$, the $\CAT(0)$ midpoint four-point inequality reads
\begin{align*}
 [\dist_Y(P,S)-\dist_Y(Q,R)]\dist_Y(Q,R)
 &\ge \dist_Y^2(P,Q_m)-\dist_Y^2(P,Q)-\dist_Y^2(Q_m,Q)\\
 &\quad+\dist_Y^2(S,Q_m)-\dist_Y^2(S,R)-\dist_Y^2(Q_m,R);
\end{align*}
see \cite[Corollary~2.1.3]{KorevaarSchoen1993}.  Apply it with
\[
 P=u(x,t_0-s),\quad S=u(y,t_0-s),\quad
 Q=u(x_0,t_0),\quad R=u(y_0,t_0).
\]
By the definition \eqref{eq:wP-definition} of $w_{Q_m,s}$, the resulting
inequality is
\[
 [F(x,y,t_0-s)-F_0]F_0
 \ge-w_{Q_m,s}(x_0,x,t_0)-w_{Q_m,s}(y_0,y,t_0).
\]
After integration against $\Pi_s$, its two marginal identities give
\[
 I_2(s)\le\frac1{F_0}\left[
 \int w_{Q_m,s}(x_0,x,t_0)\,\dd\mu_s^{x_0}(x)
 +\int w_{Q_m,s}(y_0,y,t_0)\,\dd\mu_s^{y_0}(y)\right].
\]
The point $Q_m$ lies in the closed convex hull of the bounded image.  Hence
\eqref{eq:RCD-wP-mean-value} from
\cref{prop:RCD-paired-mean-value} applies with the same target point $Q_m$ at
both $(x_0,t_0)$ and $(y_0,t_0)$.  Dividing the preceding inequality by $s$
and taking the upper limit proves the second assertion.

By \eqref{eq:test-heat-expansion} from
\cref{lem:test-heat-expansion},
\[
 \lim_{s\downarrow0}\frac1s
 \int[v(x,t_0-s)-v(x_0,t_0)]\,\dd\mu_s^{x_0}(x)
 =\mathcal H_v(x_0,t_0).
\]
The negative sign in the definition of \(I_3\) proves the third assertion.  For the
temporal part of \(I_4\),
\[
 \frac{\eta_0(t_0-s)-\eta_0(t_0)}s\longrightarrow-\eta_0'(t_0),
\]
whereas \eqref{eq:heat-flow-upper-laplacian} and
\cref{lem:RCD-tail} give
\[
 \limsup_{s\downarrow0}\frac1s
 \int[\eta_1(x)-\eta_1(x_0)]\,\dd\mu_s^{x_0}(x)
 \le\Delta_{\mathrm H,1}^{+}\eta_1(x_0).
\]
The analogous estimate for \(\eta_2\) at \(y_0\) proves the final assertion.
\end{proof}

\begin{proof}[Proof of Theorem~\ref{thm:HJ-supersolution}]
Suppose, to the contrary, that a heat test $v$ touches
$-\mathcal V_\varepsilon$ from below at $(x^\ast,t^\ast)$ and
\begin{equation}\label{eq:positive-heat-residual}
 \theta_0:=\mathcal H_v(x^\ast,t^\ast)>0.
\end{equation}
After shrinking the touching cylinder, continuity of $\mathcal H_v$ gives
\begin{equation}\label{eq:heat-residual-gap}
 \mathcal H_v\ge\frac{\theta_0}{2}.
\end{equation}
Choose a maximizing vertex $y^\ast$ as in the envelope definition
\eqref{eq:HJ-envelope-def} from \cref{sec:HJ-final}.  By the definition of
$H$ in \eqref{eq:H-def} and the
touching property in \cref{def:heat-test-supersolution}, $H$ has a local minimum at
$(x^\ast,y^\ast,t^\ast)$.

Set
\begin{equation}\label{eq:contact-full-measure-set}
 E^*:=E_{\mathrm{pair}}
 \cap\bigl(B_{8R}(\bar x)\times B_{8R}(\bar x)
            \times T_{\mathrm{dist}}\bigr).
\end{equation}
Combining the definition \eqref{eq:contact-full-measure-set}, the
full-measure conclusion \eqref{eq:E-paired-good-set} of
\cref{prop:RCD-paired-mean-value}, and \cref{prop:simultaneous-P} shows that
$E^*$ has full measure in the ambient product cylinder.  On the time slices in
$T_{\mathrm{dist}}$, \cref{prop:H-admissible}, specifically
\eqref{eq:HJ-aux-slice-bound}, verifies the slice hypotheses of
\cref{prop:normalized-rcd-perturbation}.  Apply that proposition on the
smaller touching cylinder to $H$ and $E^*$ with $\delta=\theta_0/8$.  It gives
$(x_0,y_0,t_0)\in E^*$ and perturbations $\eta_0,\eta_1,\eta_2$ such that
\begin{equation}\label{eq:perturbed-H}
 H_1:=H+\eta_0+\eta_1+\eta_2
\end{equation}
attains a minimum at $(x_0,y_0,t_0)$ on a backward product cylinder, and
\begin{equation}\label{eq:contact-perturbation-bound}
 -\eta_0'(t_0)+\Delta_{\mathrm H,1}^{+}\eta_1(x_0)
 +\Delta_{\mathrm H,2}^{+}\eta_2(y_0)
 \le\frac{\theta_0}{8}.
\end{equation}

Let $\mu_s^{x_0}:=\Hheat_s\delta_{x_0}$ and
$\mu_s^{y_0}:=\Hheat_s\delta_{y_0}$, and let $\Pi_s$ be an optimal coupling
of these measures.  Use bounded extensions of $v,\eta_1,\eta_2$ that agree
with them near the selected point.  The function $F$ is bounded, while the
diagonal squared-distance term has quadratic growth.  The minimum property in
\eqref{eq:perturbed-H} and the off-diagonal estimate
\cref{lem:RCD-tail} give the following estimate.  For small \(s\), the
integrand in the negative part vanishes on the backward product cylinder on
which \(H_1\) is minimized.  Outside that cylinder it is uniformly bounded
above after the bounded extensions are fixed; the nonnegative diagonal term
can only decrease this upper bound.  By \cref{lem:RCD-tail}, the two
marginals, and hence the coupling, assign mass \(o(s)\) to the complement.
Consequently,
\begin{equation}\label{eq:S-tail}
 \int_{\Z}
 \max\!\left\{H_1(x_0,y_0,t_0)-H_1(x,y,t_0-s),0\right\}
 \,\dd\Pi_s=o(s).
\end{equation}
By \eqref{eq:S-tail},
\begin{equation}\label{eq:S-liminf}
 0\le\liminf_{s\downarrow0}S_s,
 \qquad
 S_s:=\frac{1}{s}\int_{\Z}
 [H_1(x,y,t_0-s)-H_1(x_0,y_0,t_0)]\,\dd\Pi_s.
\end{equation}
Expanding $H_1$ by \eqref{eq:H-def} from \cref{prop:H-admissible} and
\eqref{eq:perturbed-H} gives
\[
 sS_s=I_1(s)+I_2(s)+I_3(s)+I_4(s),
\]
where the terms are defined in \eqref{eq:I1}--\eqref{eq:I4} of
\cref{lem:four-term-estimates}.  Moreover,
$(x_0,t_0),(y_0,t_0)\in G_{\mathrm{hk}}$ because
$(x_0,y_0,t_0)\in E^*$.  Applying
\cref{lem:four-term-estimates}, specifically
\eqref{eq:I1-limsup}--\eqref{eq:I4-limsup}, together with
\eqref{eq:heat-residual-gap} and \eqref{eq:contact-perturbation-bound}, gives
\begin{equation}\label{eq:I-limsup}
 \begin{aligned}
 \limsup_{s\downarrow0}\frac{I_1(s)}s&\le0,
 &\qquad \limsup_{s\downarrow0}\frac{I_2(s)}s&\le0,\\
 \limsup_{s\downarrow0}\frac{I_3(s)}s&\le-\frac{\theta_0}{2},
 &\limsup_{s\downarrow0}\frac{I_4(s)}s&\le\frac{\theta_0}{8}.
 \end{aligned}
\end{equation}
Combining \eqref{eq:S-liminf} and \eqref{eq:I-limsup} yields
\begin{equation}
 0\le\liminf_{s\downarrow0}S_s
 \le\limsup_{s\downarrow0}S_s
 \le\sum_{j=1}^4\limsup_{s\downarrow0}\frac{I_j(s)}s
 \le-\frac{3\theta_0}{8}<0.
\end{equation}
This contradicts the nonnegativity in \eqref{eq:S-liminf}.  Hence no heat
test touching $-\mathcal V_\varepsilon$ from below can satisfy
\eqref{eq:positive-heat-residual}, which proves the theorem.
\end{proof}

\medskip

\subsection{Weak formulation and the spatial Lipschitz bound}

The following semigroup comparison converts the heat-test supersolution
property into the weak formulation required for the local boundedness
estimate.

\begin{lemma}[Local backward heat-semigroup inequality]
\label{lem:heat-test-to-comparison}
Let
$f\in C(U\times I)\cap L^2_{\loc}(I;W^{1,2}_{\loc}(U))$ be locally bounded
and a heat-test supersolution of
$(\boldsymbol\Delta-\partial_t)f\le0$.  Let
\(U_0\Subset U_1\Subset U\) and \([a,b]\Subset I\),
and choose $\chi\in\Lip_c(U_1)$ with $0\le\chi\le1$ and $\chi=1$ on a
neighborhood of $U_0$.  After adding a constant so that
$f\ge0$ on $\overline{U_1}\times[a,b]$, there are constants $c,C>0$ such that
for $x\in U_0$ and $a<t-s<t<b$,
\begin{equation}\label{eq:localized-backward-semigroup}
 \Pheat_s\!\left[\chi f(\cdot,t-s)\right](x)
 \le f(x,t)+C e^{-c/s}.
\end{equation}
The constants are uniform when $t-s$ and $t$ range in a fixed compact
subinterval of $(a,b)$.
\end{lemma}

\begin{proof}
Fix $t-s\in(a,b)$ and put
\[
 h(x,r):=\Pheat_{r-(t-s)}[\chi f(\cdot,t-s)](x),
 \qquad t-s\le r\le t.
\]
The heat-kernel regularity of Jiang \cite{Jiang2015}, together with the
semigroup identity
\(\partial_\tau\Pheat_\tau g=\boldsymbol\Delta\Pheat_\tau g\) for
\(\tau>0\), shows that $h$ is a heat test on compactly contained cylinders
and satisfies
$(\boldsymbol\Delta-\partial_r)h=0$.  On the lower face,
$h=\chi f(\cdot,t-s)\le f$.  Since
$\dist(\supp\chi,X\setminus U_1)>0$, the Gaussian upper bound
\eqref{eq:Gaussian} from \cref{sec:prelim} gives local constants
$c,C>0$ such that
\[
 \sup_{z\in X\setminus U_1}\sup_{0<\tau\le s}
 \Pheat_\tau[\chi f(\cdot,t-s)](z)
 \le C\|f\|_{L^\infty(U_1\times(a,b))}e^{-c/s}=:\omega_s.
\]
Consequently $h-\omega_s\le0\le f$ on the lateral boundary of
$U_1\times[t-s,t]$.

Suppose that $h(x,t)>f(x,t)+\omega_s$ at some $x\in U_0$.  For sufficiently
small $\eta>0$, the function
\(h_\eta(z,r):=h(z,r)-\omega_s-\eta(r-t+s)\)
still exceeds $f$ somewhere at time $t$ and is below $f$ on the parabolic
boundary.  If $m_\eta:=\max(h_\eta-f)>0$, then $h_\eta-m_\eta$ touches $f$
from below at an interior point, whereas
\((\boldsymbol\Delta-\partial_r)(h_\eta-m_\eta)=\eta>0\),
contradicting the heat-test supersolution property.  This proves the lemma.
\end{proof}

\begin{proposition}[Weak supersolution property]
\label{prop:heat-test-weak-formulation}
Let
$f\in C(U\times I)\cap L^2_{\loc}(I;W^{1,2}_{\loc}(U))$ be locally bounded.
Assume that, on every compact cylinder, the maps $t\mapsto f(x,t)$ have
Lipschitz constants uniform in $x$.  If $f$ is a heat-test supersolution of
$(\boldsymbol\Delta-\partial_t)f\le0$, then, for every nonnegative
$\Phi\in\Lip_c(U\times I)$,
\begin{equation}\label{eq:elliptic-slice-weak}
 -\iint\ip{\nabla f}{\nabla\Phi}\,\dd\m\,\dd t
 +\iint f\,\partial_t\Phi\,\dd\m\,\dd t\le0.
\end{equation}
\end{proposition}

\begin{proof}
Choose \(U_0\Subset U_1\Subset U\) and \(I_0\Subset I_1\Subset I\)
so that $\supp\Phi\subset U_0\times I_0$, and choose
$\chi\in\Lip_c(U_1)$ with $\chi=1$ on a neighborhood of $U_0$.  Add a
constant to $f$ so that $f\ge0$ on $\overline{U_1}\times\overline{I_1}$; this
does not change the desired weak inequality.  Mollify $f$ in time to
produce a continuous right-hand side for the equivalence theorem of
Gigli--Mondino--Semola \cite[Theorem~1.1]{GigliMondinoSemola2024}.  Choose an
even function $\rho\in C_c^\infty((-1,1))$ such that $\rho\ge0$ and
$\int_{\mathbb R}\rho\,\dd r=1$, and put
\[
 \rho_\delta(r):=\frac1\delta\rho\!\left(\frac r\delta\right),
 \qquad
 0<\delta<\frac13\dist_{\mathbb R}(I_0,\partial I_1).
\]
Thus $\rho_\delta\in C_c^\infty((-\delta,\delta))$, $\rho_\delta\ge0$, and
$\int\rho_\delta\,\dd r=1$.  Whenever
$[t-\delta,t+\delta]\subset I_1$, define
\[
 f^{(\delta)}(x,t)=\int\rho_\delta(r)f(x,t-r)\,\dd r.
\]
For fixed $\delta$, the map $t\mapsto f^{(\delta)}(\cdot,t)$ is smooth in the
topology of locally uniform convergence, with
\[
 \partial_t f^{(\delta)}(x,t)
 =\int\rho_\delta'(r)f(x,t-r)\,\dd r.
\]
On every compact cylinder on which $f(x,\cdot)$ is uniformly
$L$-Lipschitz, one also has $|\partial_t f^{(\delta)}|\le L$.  Moreover, as
$\delta\downarrow0$,
\begin{equation}\label{eq:time-mollifier-convergence}
 f^{(\delta)}\longrightarrow f
 \quad\text{locally uniformly and in }
 L^2_{\loc}(I;W^{1,2}_{\loc}(U)).
\end{equation}
The first convergence follows from the continuity of $f$.  For the second,
restrict first to \(I'\Subset I\) and \(U'\Subset U\); the standard
approximate-identity theorem in the Banach space
\(L^2(I';W^{1,2}(U'))\) applies.  In particular, for almost every \(t\in I_0\),
\[
 (f^{(\delta)})^t\in W^{1,2}_{\loc}(U),\qquad
 \nabla (f^{(\delta)})^t
 =\int\rho_\delta(r)\nabla f^{t-r}\,\dd r
 \quad\text{locally in }L^2.
\]
The last identity follows first for finite Riemann sums from the linearity of
the differential and then for the convolution from its closedness in
\(L^2\).
Integrating \eqref{eq:localized-backward-semigroup} from
\cref{lem:heat-test-to-comparison} in the
convolution variable and then replacing $t$ by $t+s$ gives, uniformly for
$(x,t)\in U_0\times I_0$,
\begin{equation}\label{eq:localized-slice-heat-bound}
 \Pheat_s[\chi f^{(\delta)}(\cdot,t)](x)
 \le f^{(\delta)}(x,t+s)+C e^{-c/s}.
\end{equation}

Since \(\chi=1\) near \(U_0\), the function
\(\chi f^{(\delta)}(\cdot,t)\), extended by zero outside \(U_1\), is a
bounded global extension of the spatial slice \(f^{(\delta)}(\cdot,t)\) near
every point of \(U_0\).  Divide
\eqref{eq:localized-slice-heat-bound} by \(s\), subtract
\(f^{(\delta)}(x,t)/s\), and use \(e^{-c/s}=o(s)\).  The \(C^1\) time
dependence gives
\begin{equation}
 \limsup_{s\downarrow0}
 \frac{\Pheat_s[\chi f^{(\delta)}(\cdot,t)](x)
       -f^{(\delta)}(x,t)}s
 \le\partial_tf^{(\delta)}(x,t)
 \qquad (x,t)\in U_0\times I_0.
\end{equation}
Each spatial slice therefore has the heat-flow Laplacian bound
\[
 \HeatDelta f^{(\delta)}(\cdot,t)
 \le\partial_tf^{(\delta)}(\cdot,t)
 \qquad\text{on }U_0.
\]
The preceding properties of the time mollification show that both
$f^{(\delta)}$ and $\partial_tf^{(\delta)}$ are continuous and bounded on
$\overline{U_0}\times\overline{I_0}$.  In particular,
$\partial_tf^{(\delta)}(\cdot,t)$ is bounded and continuous on $U_0$ for every
fixed $t$.
Thus the continuous-right-hand-side hypothesis in the
heat-flow-to-distributional implication of Gigli--Mondino--Semola
\cite[Theorem~1.1]{GigliMondinoSemola2024} is satisfied.  It gives,
for every nonnegative
$\zeta\in\Lip_c(U_0)$,
\[
 -\int\ip{\nabla f^{(\delta)}}{\nabla\zeta}\,\dd\m
 \le\int\partial_tf^{(\delta)}\,\zeta\,\dd\m.
\]
Apply this to the time slices of $\Phi$, integrate in time, and integrate by
parts.  The two convergences in
\eqref{eq:time-mollifier-convergence} allow $\delta\downarrow0$ and prove the
proposition.
\end{proof}

By \cref{thm:HJ-supersolution,lem:HJ-envelope-regularity},
$-\mathcal V_\varepsilon$ is a heat-test supersolution and satisfies the
regularity hypotheses of
\cref{prop:heat-test-weak-formulation}.  Applying that proposition
to \(f=-\mathcal V_\varepsilon\) gives
\begin{equation}\label{eq:HJ-envelope-weak-subsolution}
 (\boldsymbol\Delta-\partial_t)\mathcal V_\varepsilon\ge0
\end{equation}
weakly on \(B_{5R}(\bar x)\times J_{\mathrm{amb}}\).

\medskip

Recall from \eqref{eq:lip-r-definition} that
\(\lip_r u^t(x)=\sup_{0<s<r}\sup_{y\in B_s(x)}F(x,y,t)/s\), and that its
scale-uniform energy estimate is \cref{prop:lip-r-energy}.

\begin{proposition}[Metric Hamilton--Jacobi bound]\label{prop:metric-HJ-bound}
For $0<\varepsilon<\varepsilon_0$,
\begin{equation}\label{eq:metric-HJ-bound}
 0\le\frac{\mathcal V_\varepsilon(x,t)}\varepsilon
 \le2e^{2|K|\tau_+}[\lip_{C_{\mathrm{vert}}\sqrt\varepsilon}u^t(x)]^2
\end{equation}
for \((x,t)\in B_{5R}(\bar x)\times J_{\mathrm{amb}}\), where
$C_{\mathrm{vert}}$ is the constant from
\cref{lem:maximizer-localization}, specifically
\eqref{eq:maximizer-sqrt-eps}.
\end{proposition}

\begin{proof}
Let \(y\) be a maximizing vertex in \eqref{eq:HJ-envelope-def} from
\cref{sec:HJ-final}, and set \(s:=\dist(x,y)\).  The localization estimate
\eqref{eq:maximizer-sqrt-eps} from \cref{lem:maximizer-localization} gives
\(s\le C_{\mathrm{vert}}\sqrt\varepsilon\).
If \(s=0\), the assertion follows directly.  If \(s>0\), then the definition
\eqref{eq:lip-r-definition} from \cref{subsec:two-point-estimates} gives
\(F(x,y,t)\le\lip_{C_{\mathrm{vert}}\sqrt\varepsilon}u^t(x)s\).
Since the competitor \(y=x\) shows that \(\mathcal V_\varepsilon\ge0\),
the maximizing identity also gives
\(0\le\mathcal V_\varepsilon(x,t)=F(x,y,t)-e^{-2Kt}s^2/(2\varepsilon)\).
Thus \(s\le2\varepsilon e^{2Kt}
\lip_{C_{\mathrm{vert}}\sqrt\varepsilon}u^t(x)\).
Substituting this bound into the preceding estimate for \(F\) and using
\(e^{2Kt}\le e^{2|K|\tau_+}\) proves the proposition; this is the
calculation in \cite[Lemma~7.7]{ZhangZhu2026}.
\end{proof}

\begin{theorem}[Local spatial Lipschitz regularity]\label{thm:spatial-Lipschitz}
Let $B_{16R}(\bar x)\Subset\Omega$ and $0<t_*<T$.  There is $C<\infty$ such
that
\begin{equation}\label{eq:local-spatial-Lipschitz}
 \dist_Y(u(x,t),u(y,t))\le C\dist(x,y)
\end{equation}
for $x,y\in B_R(\bar x)$ and $t\in[t_*,T]$.  The constant depends only on
$K,N,R,t_*,T$, the fixed local doubling--Poincar\'e and volume data on
$B_{16R}(\bar x)$, the image radius $M$ from
\cref{subsec:semigroup-HMHF}, specifically \eqref{eq:bounded-image-ball}, and
$\FlowE(u_0)$.
\end{theorem}

\begin{proof}
Choose bounded open intervals $J,J'$ such that
\begin{equation}\label{eq:HJ-fixed-time-intervals}
 [t_*,T]\Subset J\Subset J'\Subset J_{\mathrm{amb}},
\end{equation}
and set
\begin{equation}\label{eq:HJ-fixed-cylinders}
 Q:=B_{3R}(\bar x)\times J,
 \qquad Q':=B_{4R}(\bar x)\times J'.
\end{equation}
For $q_\varepsilon:=\mathcal V_\varepsilon/\varepsilon$, the competitor
$y=x$ in the envelope definition \eqref{eq:HJ-envelope-def} from
\cref{sec:HJ-final} gives $q_\varepsilon\ge0$.  The weak subsolution
conclusion obtained from
\cref{thm:HJ-supersolution,lem:HJ-envelope-regularity,prop:heat-test-weak-formulation}, namely
\eqref{eq:HJ-envelope-weak-subsolution}, gives the differential inequality.
Thus
\begin{equation}\label{eq:qepsilon-subsolution}
 q_\varepsilon\ge0,\qquad
 (\boldsymbol\Delta-\partial_t)q_\varepsilon\ge0
 \quad\text{weakly on }B_{5R}(\bar x)\times J_{\mathrm{amb}}.
\end{equation}
Moreover, \cref{prop:metric-HJ-bound}, specifically
\eqref{eq:metric-HJ-bound}, gives
\begin{equation}\label{eq:Veps-lipr-control}
 0\le \frac{\mathcal V_\varepsilon}{\varepsilon}
 \le 2e^{2|K|\tau_+}
 [\lip_{C_{\mathrm{vert}}\sqrt\varepsilon}u^t]^2
 \quad\text{on }Q'.
\end{equation}

By the properness in \cref{sec:prelim}, cover
$\overline{B}_{4R}(\bar x)$ by finitely many balls $B_\rho(x_i)$, where
$\rho>0$ is fixed so that
\begin{equation}\label{eq:HJ-covering-balls}
 B_{16\rho}(x_i)\Subset B_{12R}(\bar x)
 \qquad\text{for every }i.
\end{equation}
Decrease $\varepsilon_0$, if necessary, so that
$C_{\mathrm{vert}}\sqrt{\varepsilon_0}<\rho/4$.  Integrating
\eqref{eq:Veps-lipr-control} and applying
\cref{prop:lip-r-energy}, specifically \eqref{eq:lip-r-energy}, on the
balls in \eqref{eq:HJ-covering-balls} yield, for
$0<\varepsilon<\varepsilon_0$,
\begin{equation}\label{eq:Veps-L1-bound}
 \int_{Q'}q_\varepsilon\,\dd\m\,\dd t
 \le C_{\mathrm{loc}}\left[
 \int_{J'}\int_{B_{16R}(\bar x)}e_{u^t}\,\dd\m\,\dd t
 +c_L^2R^2\m(B_{16R}(\bar x))|J'|\right].
\end{equation}

By \eqref{eq:HJ-fixed-time-intervals}--\eqref{eq:HJ-fixed-cylinders}, finitely
many backward cylinders cover $Q$ and have doubled cylinders compactly
contained in $Q'$.  Apply Part~(i) of
\cref{prop:scalar-parabolic-estimates}, specifically
\eqref{eq:expanded-local-boundedness} with $A=0$, to $q_\varepsilon$ on each
of them, using
\eqref{eq:qepsilon-subsolution} and \eqref{eq:Veps-L1-bound}.  Since
$q_\varepsilon$ is continuous by \cref{lem:HJ-envelope-regularity}, the
resulting essential supremum is an ordinary supremum, and
\begin{equation}\label{eq:epsilon-uniform-bound}
 \sup_Q\frac{\mathcal V_\varepsilon}{\varepsilon}
 \le C_{\mathrm{loc}}\left[
 \int_{J'}\int_{B_{16R}(\bar x)}e_{u^t}\,\dd\m\,\dd t
 +c_L^2R^2\m(B_{16R}(\bar x))|J'|\right]
 =:C_{\mathrm{HJ}}.
\end{equation}

For $x\in B_{3R}(\bar x)$, $t\in J$, and
$y\in\overline{B}_{6R}(\bar x)$, the envelope definition
\eqref{eq:HJ-envelope-def} from \cref{sec:HJ-final} and
\eqref{eq:epsilon-uniform-bound} give
\begin{equation}\label{eq:small-pair-HJ}
 \frac{F(x,y,t)}\varepsilon-
 \frac{e^{-2Kt}\dist^2(x,y)}{2\varepsilon^2}
 \le\frac{\mathcal V_\varepsilon(x,t)}\varepsilon\le C_{\mathrm{HJ}}.
\end{equation}
Choose $0<\varepsilon_1<\min\{\varepsilon_0,R\}$.  If
$x\in B_{3R}(\bar x)$, $y\in B_{4R}(\bar x)$, and
$0<\dist(x,y)<\varepsilon_1$, set $\varepsilon=\dist(x,y)$ in
\eqref{eq:small-pair-HJ}.  Then
\begin{equation}\label{eq:local-small-pair-lipschitz}
 F(x,y,t)
 \le C_{\mathrm{HJ}}\dist(x,y)
   +\frac{e^{-2Kt}}{2}\dist(x,y)
 \le C_1\dist(x,y)
 \qquad (t\in J),
\end{equation}
where $C_1:=C_{\mathrm{HJ}}+\frac12\sup_{t\in J}e^{-2Kt}$.

For arbitrary \(x,y\in B_R(\bar x)\), let \(\sigma\) be a minimizing
geodesic from \(x\) to \(y\).  For every point \(z\) on \(\sigma\),
\[
 \dist(z,\bar x)\le \dist(z,x)+\dist(x,\bar x)
 < \dist(x,y)+R<3R,
\]
so the whole geodesic lies in \(B_{3R}(\bar x)\).  Subdivide \(\sigma\) into
segments of length below \(\varepsilon_1\).  Each pair of consecutive
subdivision points lies in the domain of
\eqref{eq:local-small-pair-lipschitz} and has distance below
\(\varepsilon_1\).  Summing that estimate and
using the triangle
inequality in \(Y\) gives
\[
 \dist_Y(u(x,t),u(y,t))\le C_1\sum_j\dist(z_j,z_{j+1})
 =C_1\dist(x,y),
\]
which proves the spatial Lipschitz estimate.

The dependence on the image radius $M$ enters through $C_{\mathrm{vert}}$; the
remaining quantities in \eqref{eq:epsilon-uniform-bound} satisfy
\begin{equation}\label{eq:local-constant-bounds}
 c_L\le C_{\mathrm{loc}}\sqrt{\frac{\FlowE(u_0)}{t_*}},\qquad
 \int_{t_*/2}^{T+1}\int_{B_{16R}(\bar x)}e_{u^t}\,\dd\m\,\dd t
 \le\frac{2(T+1-t_*/2)}{n+2}\FlowE(u_0).
\end{equation}
Indeed, \eqref{eq:pointwise-time-Lip} from
\cref{thm:RCD-time-Lipschitz} gives $c_L\le C_{\mathrm{loc}}L$.
By Mayer \cite[Theorem~2.17, Corollary~2.18, and
Theorem~2.36]{Mayer1998}, the metric speed is nonincreasing and its square is
the rate of energy dissipation.  Hence
$L^2\le 2\FlowE(u_0)/t_*$.  The second inequality in
\eqref{eq:local-constant-bounds} follows from energy monotonicity
\eqref{eq:EVI-energy-monotonicity} from
\cref{subsec:semigroup-HMHF} and the normalization
\(\FlowE=(n+2)\MapE/2\).  The localization estimate from
\cref{lem:maximizer-localization}, specifically
\eqref{eq:maximizer-sqrt-eps}, together with the local formulas
\eqref{eq:epsilon-uniform-bound} and \eqref{eq:local-constant-bounds}, gives
the asserted dependence of the
Lipschitz constant.
\end{proof}

\medskip

\begin{proof}[Proof of \cref{main-thm}]
\emph{(i) Local Lipschitz regularity.}

We use the H\"older representative given by
\cref{thm:RCD-holder-map-flow}.
The assertion is local, so fix a compact cylinder
\(B_{16R}(x_0)\Subset\Omega\), \(0<t_*<T<\infty\).
The positive-time estimate \cref{thm:RCD-time-Lipschitz} gives
\(\dist_Y(u(z,t),u(z,s))\le c_L|t-s|\) for
\(z\in B_R(x_0)\) and \(s,t\in[t_*,T]\).
The spatial estimate \cref{thm:spatial-Lipschitz}, applied with the same
localization ball, gives a constant $C_S<\infty$ such that
\(\dist_Y(u(x,t),u(y,t))\le C_S\dist(x,y)\) for
\(x,y\in B_R(x_0)\) and \(t\in[t_*,T]\).
Hence, for \(x,y\in B_R(x_0)\) and \(s,t\in[t_*,T]\),
\begin{align*}
 \dist_Y(u(x,t),u(y,s))
 &\le\dist_Y(u(x,t),u(y,t))
       +\dist_Y(u(y,t),u(y,s))\\
 &\le C_S\dist(x,y)+c_L|t-s|.
\end{align*}
This gives the space--time Lipschitz estimate, with the
dependence listed in \cref{main-thm} through the larger ball
\(B_{16R}(x_0)\).  A finite covering yields local space--time Lipschitz
regularity on \(\Omega\times(0,\infty)\).

The quantitative bounds \eqref{eq:local-constant-bounds} from
\cref{thm:spatial-Lipschitz}, together with
the image bound \eqref{eq:bounded-image-ball} from
\cref{subsec:semigroup-HMHF}, give the dependence stated in part~(i).

\medskip

\emph{(ii) The Eells--Sampson-type Bochner inequality.}

Let $p\ge2$ be an integer and let $q=p/(p-1)$.  For $\varepsilon>0$ define
\begin{equation}\label{eq:HJ-envelope-def-p}
 \mathcal V_{\varepsilon,p}(x,t):=
 \max_{y\in\overline{B}_{6R}(x_0)}
 \left\{
 F(x,y,t)-\frac{e^{-pKt}}{p\varepsilon^{p-1}}\dist^p(x,y)
 \right\}.
\end{equation}

Let \(y\) be a maximizing vertex in \eqref{eq:HJ-envelope-def-p} for
\((x,t)\in B_{5R}(x_0)\times[t_*,T]\), and put \(s:=\dist(x,y)\).  Since
\(y=x\) is an admissible competitor, \(\mathcal V_{\varepsilon,p}\ge0\), and
the bounded-image estimate \eqref{eq:bounded-image-ball} from
\cref{subsec:semigroup-HMHF} gives
\(e^{-pKt}s^p/(p\varepsilon^{p-1})\le F(x,y,t)\le2M\).
For \(0<\varepsilon\le1\), it follows that
\(s\le e^{|K|T}(2Mp)^{1/p}\varepsilon^{1-1/p}\le C_0\sqrt\varepsilon\),
where \(C_0\) is independent of \(p\ge2\) and \(\varepsilon\).  Decrease
\(\varepsilon_0>0\) so that \(C_0\sqrt{\varepsilon_0}<R/2\).  Then every
maximizing vertex satisfies \(s<R/2\) and, in particular, belongs to
\(B_{6R}(x_0)\).  Since
\(B_{8R}(x)\subset B_{13R}(x_0)\Subset\Omega\), part~(i), applied at \(x\)
with radius \(R/2\), gives a uniform constant
\(C_S<\infty\) such that \(F(x,y,t)\le C_Ss\).  Combining this estimate
with the maximizing inequality above, if \(s>0\), yields
\begin{equation}\label{eq:p-maximizer-localization}
 s\le p^{1/(p-1)}e^{qKt}C_S^{1/(p-1)}\varepsilon
 \le C_*\varepsilon,
\end{equation}
where \(C_*\) is independent of \(p\ge2\) and \(\varepsilon\); the same
conclusion is immediate when \(s=0\).  Consequently, when a heat test touches
\(-\mathcal V_{\varepsilon,p}\) from below, the corresponding auxiliary
function has a local minimum in an ambient product neighborhood, as required
by \cref{prop:normalized-rcd-perturbation}.

On an \RCD{} source, the normalized perturbation in
\cref{prop:normalized-rcd-perturbation} replaces the contact-selection step in
\cite[Section~8, especially Lemma~8.2]{ZhangZhu2026}.  We first verify that
the perturbation proposition applies to the \(p\)-power penalty.  On a fixed
product cylinder \(U\times V\times I\), \cref{lem:diagonal-power} gives
\[
 \DeltaZ\!\left[
  \frac{e^{-pKt}}{p\varepsilon^{p-1}}\dist^p(x,y)
 \right]
 \le \frac{e^{p|K|\sup I}}{\varepsilon^{p-1}}
 C_p(K,N,R_{UV})\,\m_{\Z}.
\]
Together with the bounds for \(-F\) in \cref{prop:F-two-point} and for a
heat test in \cref{def:heat-test}, this is the slice upper-Laplacian bound
required by \cref{prop:normalized-rcd-perturbation}.  At the selected contact
point, use an optimal \(W_p\)-coupling.  The contraction
\eqref{eq:Wp-contraction} from \cref{sec:prelim} yields
\[
 \int_{\Z}\dist^p(x,y)\,\dd\Pi_s(x,y)
 \le e^{-pKs}\dist^p(x_0,y_0),
\]
and hence
\begin{align*}
 &\frac{e^{-pK(t_0-s)}}{p\varepsilon^{p-1}}
   \int_{\Z}\dist^p(x,y)\,\dd\Pi_s(x,y)
   -\frac{e^{-pKt_0}}{p\varepsilon^{p-1}}\dist^p(x_0,y_0)\\
 &\hspace{35mm}\le0.
\end{align*}
Thus the \(p\)-penalty satisfies precisely the estimate
\eqref{eq:I1-limsup} in \cref{lem:four-term-estimates}.  The estimates
\eqref{eq:I2-limsup}--\eqref{eq:I4-limsup} from that lemma do not involve
the exponent of the penalty.  The contradiction step of
\cite[Lemma~8.2]{ZhangZhu2026}, with \eqref{eq:I1-limsup} supplied by the
preceding $p$-penalty estimate, therefore shows that
\(-\mathcal V_{\varepsilon,p}\) is a heat-test supersolution.  The argument of
\cref{lem:HJ-envelope-regularity}, with the penalty in
\eqref{eq:HJ-envelope-def-p}, supplies the local Sobolev and time regularity;
hence \cref{prop:heat-test-weak-formulation} gives
\begin{equation}\label{eq:p-envelope-subsolution}
 \mathcal V_{\varepsilon,p}\ge0,
 \qquad
 (\boldsymbol\Delta-\partial_t)\mathcal V_{\varepsilon,p}\ge0
\end{equation}
weakly.

Set \(g_{\varepsilon,p}:=\mathcal V_{\varepsilon,p}/\varepsilon\).  The
localization estimate \eqref{eq:p-maximizer-localization} gives
\begin{equation}\label{eq:p-envelope-uniform-bound}
 0\le g_{\varepsilon,p}\le C
\end{equation}
on every smaller cylinder, with \(C\) independent of \(p\) and
\(\varepsilon\), since
\(g_{\varepsilon,p}\le F(x,y,t)/\varepsilon
\le C_SC_*\).
The local Caccioppoli estimate applied to
\eqref{eq:p-envelope-subsolution}, together with
\eqref{eq:p-envelope-uniform-bound}, makes
\(\{g_{\varepsilon,p}\}_{\varepsilon,p}\) locally bounded in \(V_2\),
uniformly in both parameters; this is the argument in
\cite[proof of Theorem~8.1, equations~(8.8)--(8.10)]{ZhangZhu2026}.
Moreover, the metric Hamilton--Jacobi limit
\cite[Lemma~4.4]{ZhangZhongZhu2019}, applied to the metric
\(\dist_t=e^{-Kt}\dist\), gives the precise almost-everywhere limit
\begin{equation}\label{eq:p-envelope-limit}
 \lim_{\varepsilon\downarrow0}g_{\varepsilon,p}(x,t)
 =\frac{e^{qKt}}q\bigl(\lip_x u(x,t)\bigr)^q.
\end{equation}
Here the sign is positive because \(\mathcal V_{\varepsilon,p}\) is the
negative of the infimum envelope used in the cited formula.  The bound
\eqref{eq:p-envelope-uniform-bound}, weak compactness, and
\eqref{eq:p-envelope-limit} show that
\[
 \frac{e^{qKt}}q(\lip_x u)^q\in V_{2,\loc},
 \qquad
 (\boldsymbol\Delta-\partial_t)
 \left[\frac{e^{qKt}}q(\lip_x u)^q\right]\ge0.
\]
Letting \(p\to\infty\), so that \(q\downarrow1\), the local
\(L^\infty\)-bound from \cref{main-thm}(i) gives
\(e^{qKt}(\lip_x u)^q/q\to
e^{Kt}\lip_x u\) strongly in local \(L^2\).  The uniform \(V_2\)
bound and weak lower semicontinuity then yield
\begin{equation}\label{eq:linear-lip-bochner}
 \lip_x u\in V_{2,\loc}\cap L^\infty_{\loc},
 \qquad
 (\boldsymbol\Delta-\partial_t)\lip_x u
 \ge K\,\lip_x u
\end{equation}
in the sense of distributions.

To pass from \eqref{eq:linear-lip-bochner} to the inequality for the square,
put \(\lambda:=\lip_x u\) and fix a nonnegative
\(\phi\in\Lip_c(\Omega\times(0,\infty))\).  On a cylinder containing
\(\supp\phi\), spatial Lipschitz approximation and temporal Steklov averaging
extend \eqref{eq:linear-lip-bochner} to bounded compactly supported energy
tests.  If \(\lambda_h\) denotes the temporal Steklov average, apply the
averaged inequality with test \(2\min\{\lambda_h,R\}\phi\) and then let
\(h\downarrow0\).  The Steklov averages converge to \(\lambda\) strongly in the local
\(L^2_tW^{1,2}_x\) topology, which permits passage to both terms containing
\(\lambda\).
Since \(\lambda\) is locally bounded, choose \(R\) larger than its
\(L^\infty\)-bound on the cylinder.  The Sobolev chain rule gives
\[
 -\int\!\!\int\ip{\nabla\lambda^2}{\nabla\phi}\,\dd\m\,\dd t
 +\int\!\!\int\lambda^2\partial_t\phi\,\dd\m\,\dd t
 \ge2\int\!\!\int
 \bigl(|\mathrm D\lambda|^2+K\lambda^2\bigr)\phi\,\dd\m\,\dd t.
\]
This is precisely
\[
 (\boldsymbol\Delta-\partial_t)(\lip_x u)^2
 \ge 2|\mathrm D\lip_x u|^2+2K(\lip_x u)^2.
\]
\end{proof}

\section*{Declaration}

The authors used OpenAI Codex during manuscript preparation for language
editing, cross-reference and bibliography checks, and suggestions on
manuscript organization.
The authors independently checked all mathematical statements and proofs and
take full responsibility for the manuscript.

\end{document}